\documentclass{article}
\usepackage[utf8]{inputenc}
\usepackage{fullpage}
\usepackage{amsmath}
\usepackage{amssymb}
\usepackage{amsthm}
\usepackage{epsfig}
\usepackage{epstopdf}
\usepackage{mathtools}
\usepackage{relsize}
\usepackage{url}
\usepackage{hyperref}
\hypersetup{
    colorlinks=true,     % false: boxed links; true: colored links
    linkcolor=blue,     % color of internal links (change box color with linkbordercolor)
    citecolor=blue,     % color of links to bibliography
}
\usepackage[hyphenbreaks]{breakurl}
\usepackage[numbers,sort&compress]{natbib}
\usepackage{enumitem}
\usepackage{algorithm,algpseudocode}
\usepackage{setspace}
\usepackage{caption}
\usepackage{subcaption}
\usepackage{xspace}
\usepackage{authblk}
\usepackage{dsfont}
\usepackage{multicol}
\usepackage[export]{adjustbox}
\usepackage{thmtools,thm-restate}
\usepackage{xcolor}
\usepackage{rotating}
\usepackage{booktabs}

\title{Bounding-Focused Discretization Methods for the Global Optimization of Nonconvex Semi-Infinite Programs}

\date{June 12, 2025}

\author[1]{Evren M. Turan}
\author[1]{Johannes J\"aschke}
\author[2]{Rohit Kannan}
\affil[1]{Department of Chemical Engineering, Norwegian University of Science and Technology (NTNU), \protect\\ Trondheim, Norway. E-mail: evren.m.turan@ntnu.no, johannes.jaschke@ntnu.no}
\affil[2]{Grado Department of Industrial and Systems Engineering, Virginia Tech, Blacksburg, VA, USA, \protect\\ E-mail: rohitkannan@vt.edu}

\DeclareMathOperator*{\argmin}{arg\,min}
\DeclareMathOperator*{\argmax}{arg\,max}

\newcommand{\tr}[1]{\ensuremath{{#1}^\text{T}}}

\newcommand{\uset}[2]{\ensuremath{\underset{#1}{#2}}}

\DeclarePairedDelimiter\abs{\lvert}{\rvert}%
\DeclarePairedDelimiter\norm{\lVert}{\rVert}%

\newcommand{\A}{\mathcal{A}}

\newcommand{\E}{\mathcal{E}}

\newcommand{\I}{\mathcal{I}}

\newcommand{\R}{\mathbb{R}}

\providecommand{\keywords}[1]
{
  \small	
  \textbf{Key words:} #1
}

\newtheorem{theorem}{Theorem}[]
\newtheorem{lemma}[theorem]{Lemma}
\newtheorem{assumption}{Assumption}[]
\newtheorem{remark}{Remark}
\newtheorem{definition}{Definition}
\newtheorem{corollary}[theorem]{Corollary}
\newtheorem{proposition}[theorem]{Proposition}
\newtheorem{example}{Example}
\newtheorem{conjecture}[theorem]{Conjecture}

\let\oldtheorem\theorem
\renewcommand{\theorem}{\oldtheorem\normalfont}

\let\oldlemma\lemma
\renewcommand{\lemma}{\oldlemma\normalfont}

\let\oldassumption\assumption
\renewcommand{\assumption}{\oldassumption\normalfont}

\let\oldremark\remark
\renewcommand{\remark}{\oldremark\normalfont}

\let\olddefinition\definition
\renewcommand{\definition}{\olddefinition\normalfont}

\let\oldcorollary\corollary
\renewcommand{\corollary}{\oldcorollary\normalfont}

\let\oldproposition\proposition
\renewcommand{\proposition}{\oldproposition\normalfont}

\let\oldexample\example
\renewcommand{\example}{\oldexample\normalfont}

\let\oldconjecture\conjecture
\renewcommand{\conjecture}{\oldconjecture\normalfont}

\begin{document}

\maketitle

\begin{abstract}
We use sensitivity analysis to design \textit{bounding-focused} {discretization} (cutting-surface) methods for the global optimization of {nonconvex} {semi-infinite} programs (SIPs).
We begin by formulating the {optimal} bounding-focused discretization of SIPs as a max-min problem and {propose} {variants} that are more computationally tractable.
We then use {parametric} {sensitivity} theory to design an effective heuristic approach for solving these max-min problems.
We also show how our new iterative discretization methods may be modified to ensure that the solutions of their {discretizations} converge to an optimal solution of the SIP.
We then formulate optimal bounding-focused {\textit{generalized}} {discretization} of SIPs as {max-min} problems and design heuristic {algorithms} for their solution.
Numerical experiments on standard nonconvex SIP test instances from the literature demonstrate that our new {bounding-focused} {discretization} methods can significantly reduce the number of {iterations} for convergence {relative} to a state-of-the-art feasibility-focused discretization method. \\[0.1in]
\end{abstract}

\section{Introduction}
\label{sec: intro}

Semi-infinite programs (SIPs) are mathematical optimization problems with a finite number of decision variables and an infinite number of constraints.
They can be used to model several problems in science and engineering, such as robust optimization, controller design, Chebyshev approximation, and design centering~\cite{grossmann1983optimization,turan2022design,lopez2007semi,djelassi2021recent}.
Our focus is on the \textit{global} optimization of SIPs of the form:
\begin{align}
\label{eqn:sip}
v^* := \min_{x \in X} \:\: & f(x) \tag{SIP} \\
\text{s.t.} \:\: & g(x,y) \leq 0, \quad \forall y \in Y, \label{eqn:sic}
\end{align}
where $X \subset \R^{d_x}$ and $Y \subset \R^{d_y}$ are nonempty compact sets, $\abs{Y} = \infty$, and functions $f: \R^{d_x} \to \R$ and $g: \R^{d_x} \times \R^{d_y} \to \R$ are continuous.
We assume \eqref{eqn:sip} is feasible, but do not assume that $X$, $Y$, $f$, or $g$ is convex.
We detail extensions to SIPs with multiple semi-infinite constraints in Section~\ref{sec:generalizations}.

A key challenge in solving~\eqref{eqn:sip} is that evaluating feasibility of a candidate solution $x \in X$ requires the \textit{global} solution of the following lower-level problem:
\begin{equation}
    \label{eqn:llp}
G(x) := \max_{y \in Y} \: g(x,y). \tag{LLP($x$)}
\end{equation}
Clearly, $x \in X$ is feasible for~\eqref{eqn:sip} if and only if $G(x) \leq 0$.

Several papers propose algorithms for the global minimization of nonconvex SIPs~\cite{bhattacharjee2005global,floudas2008adaptive,mitsos2008relaxation,mitsos2011global,tsoukalas2011feasible,stein2012adaptive,djelassi2017hybrid,marendet2020standard,djelassi2020discretization,djelassi2021recent}. 
They primarily construct lower bounds for~\eqref{eqn:sip} by replacing the semi-infinite constraint~\eqref{eqn:sic} with a \textit{finite} discretization (cf.\ \cite{cheney1959newton,kelley1960cutting}):
\begin{align}
\label{eqn:disc-lbp}
\min_{x \in X} \:\: & f(x) \tag{LBP} \\
\text{s.t.} \:\: & g(x,y) \leq 0, \quad \forall y \in Y_d, \nonumber
\end{align}
where $Y_d \subsetneq Y$ with $\abs{Y_d} < \infty$.
A lower bound for the optimal value $v^*$ of~\eqref{eqn:sip} can be obtained by solving this (nonconvex) problem to {\it global} optimality.

The choice of discretization $Y_d$ can greatly impact the {tightness} of the lower bound obtained by solving~\eqref{eqn:disc-lbp}.
Because na\"ive {discretization} approaches may necessitate a large discretization for~\eqref{eqn:disc-lbp} to {effectively} approximate \eqref{eqn:sip}~\cite{still2001discretization}, techniques for adaptively populating~$Y_d$ are of interest.
Global {optimization} methods for~\eqref{eqn:sip} {predominantly} rely on the {\textit{feasibility-focused}} discretization method of {Blankenship} and Falk (BF; see Algorithm~\ref{alg:bfdisc})~\cite{blankenship1976infinitely}, which is \textit{the} {state-of-the-art discretization method} for nonconvex SIPs.
This method populates $Y_d$ with points in $Y$ corresponding to the largest violation of constraint~\eqref{eqn:sic} at incumbent solutions of~\eqref{eqn:disc-lbp}.

\begin{algorithm}[t]
\caption{The Blankenship and Falk algorithm \cite{blankenship1976infinitely}}
\label{alg:bfdisc}
{
\begin{algorithmic}[1]
\State \textbf{Input:} feasibility tolerance $\varepsilon_f \geq 0$, initial discretization $Y_d = \emptyset$.

\For{$k = 1, 2, \dots$}

\State Solve problem~\eqref{eqn:disc-lbp} globally to get solution $x^k$, lower bound $LBD^k$.

\State Solve problem~\eqref{eqn:llp} with $x = x^k$ globally to get solution $y^{BF,k} \in Y$.

\If{$G(x^k) \leq \varepsilon_f$}

\State \textbf{Terminate} with $\varepsilon_f$-feasible solution $x^k$ to~\eqref{eqn:sip}.

\Else

\State Set $Y_d \leftarrow Y_d \cup \{ y^{BF,k} \}$.

\EndIf

\EndFor

\end{algorithmic}
}
\end{algorithm}

The sequence of non-decreasing lower bounds $\{LBD^k\}_k$ generated by the BF algorithm converges to $v^*$ under our assumptions on~\eqref{eqn:sip} (see, e.g.,~\citep[Theorem~2.1]{blankenship1976infinitely}).
However, as Example~\ref{exm:dp} below illustrates, the BF algorithm may require an excessively large discretization $Y_d$ before the sequence $\{LBD^k\}_k$ converges to within a specified tolerance of the optimal value $v^*$.

\begin{example}{\citep[Example (DP)]{mitsos2009test}}
\label{exm:dp}
Consider~\eqref{eqn:sip} with $d_x = 1$, $d_y = 1$, $X = [0,6]$, $Y = [2,6]$, $f(x) = 10 - x_1$, and $g(x,y) = \frac{y^2_1}{1 + \exp(-40(x_1-y_1))} + x_1 - y_1 - 2$.
The global solution is $x^* = 2$ with objective $v^* = 8$.

Solving~\eqref{eqn:disc-lbp} with $Y_d = \{2\}$ (prescribed by our new discretization methods) yields a lower bound of $v^*$. 
However, as shown in Table~\ref{tab:exdp}, the state-of-the-art BF algorithm needs more than $20$ iterations to even yield a lower bound that is within $10\%$ of $v^*$. 
Figure~\ref{fig:plot_dp_bf} contrasts the above discretizations of the semi-infinite constraint.
% \vspace*{-0.1in}
\end{example}

\begin{table}[t]
\centering
\begin{tabular}{l|c|c|c|c|c|c|c|c|c|c}
\hline
\textbf{Iteration No.} & 1 & 2    & 3    & 4    & 5    & 10   & 15   & 20   & 25   & 28 \\ \hline
\textbf{Lower Bound}     & 4 & 4.19 & 4.38 & 4.56 & 4.74 & 5.62 & 6.41 & 7.12 & 7.73 & 8  \\ \hline
\end{tabular}
\caption{Lower bounds generated by the BF algorithm~\ref{alg:bfdisc} on Example~\ref{exm:dp}. This algorithm needs more than $20$ iterations to approximate $v^* = 8$ to within $10\%$.}
\label{tab:exdp}
% \vspace*{-0.15in}
\end{table}

Example~\ref{exm:dp} motivates our study of \textit{new} \emph{bounding-focused} discretization methods for~\eqref{eqn:sip} that can mitigate the slow convergence of the feasibility-focused BF algorithm.
The main idea of our bounding-focused discretization methods is to populate the discretization $Y_d$ with points in $Y$ such that the lower bound from~\eqref{eqn:disc-lbp} is maximized.
Since determining optimal bounding-focused discretizations may be challenging, we consider more tractable variants and design efficient heuristic approaches to determine such discretizations.
We also study how to construct bounding-focused \textit{generalized} discretizations.

While we only investigate discretization methods that yield tighter lower bounds, our ideas may be adapted to design discretization methods for finding feasible solutions faster (cf.\ \cite{mitsos2011global,tsoukalas2011feasible}).
We assume for our theoretical results that all subproblems solved to global optimality are solved exactly in finite time; our approaches may also be adapted to the setting where such subproblems are only solved to $\varepsilon$-global optimality for some $\varepsilon > 0$~\cite{djelassi2017hybrid,djelassi2020discretization,harwood2021note}.

\begin{figure}[t]
\centering
\begin{subfigure}{\textwidth}
\includegraphics[width=0.33\linewidth]{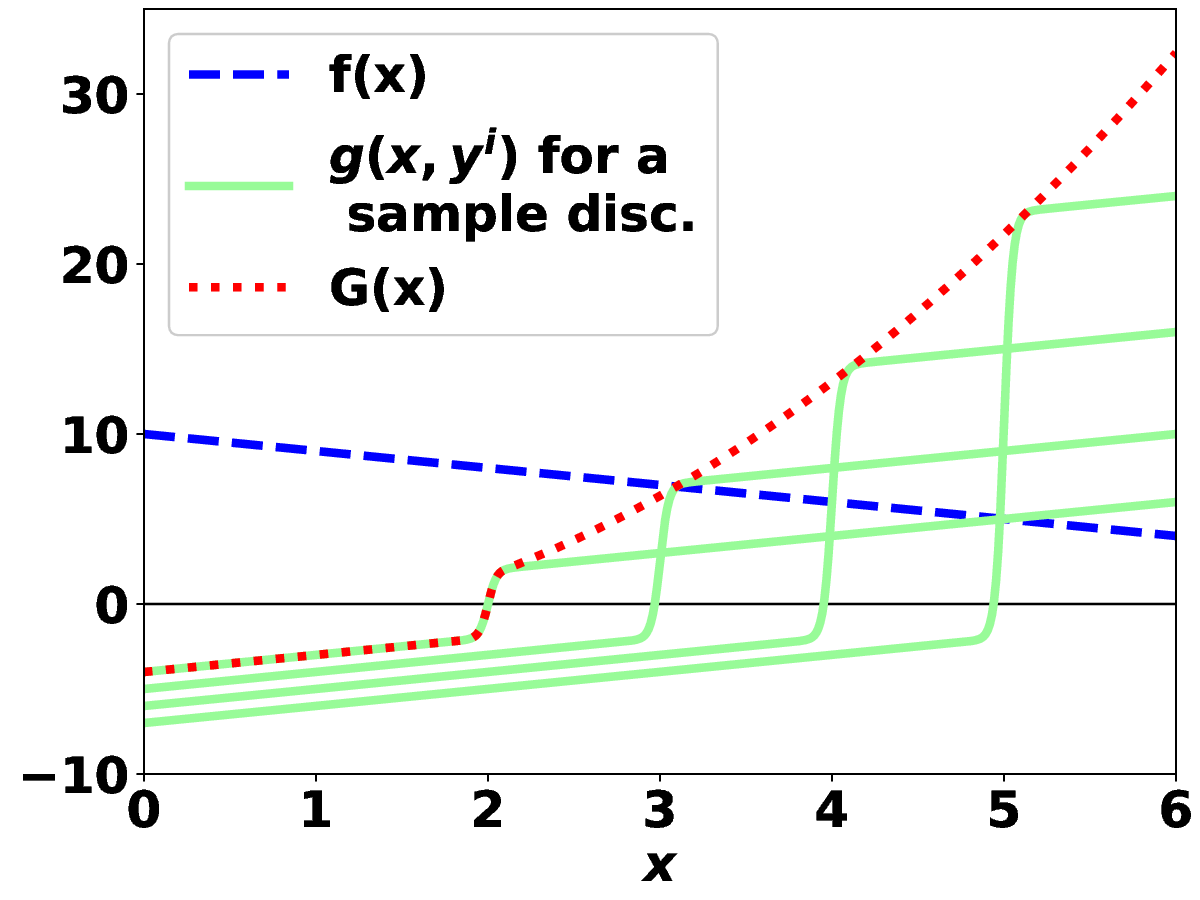}%
\hfill
\includegraphics[width=0.33\linewidth]{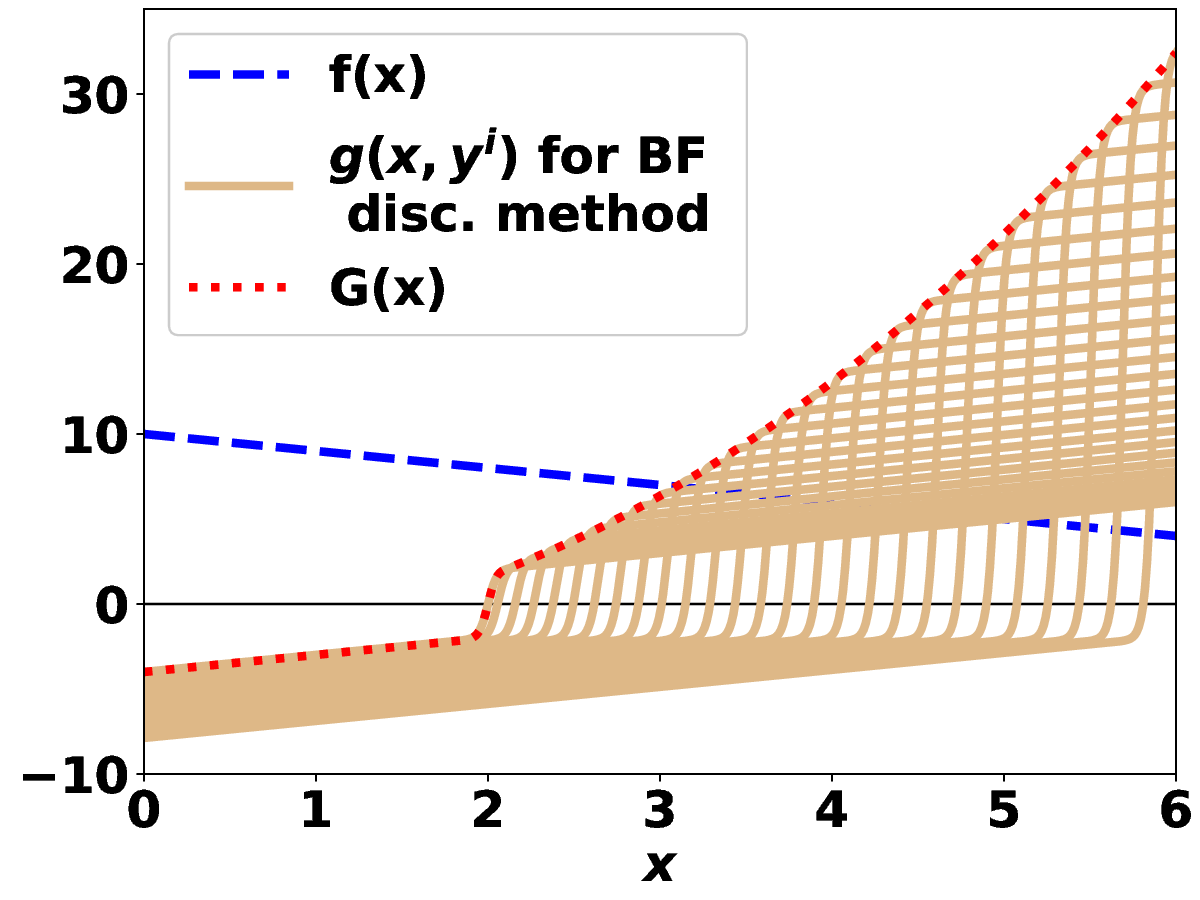}%
\hfill
\includegraphics[width=0.33\linewidth]{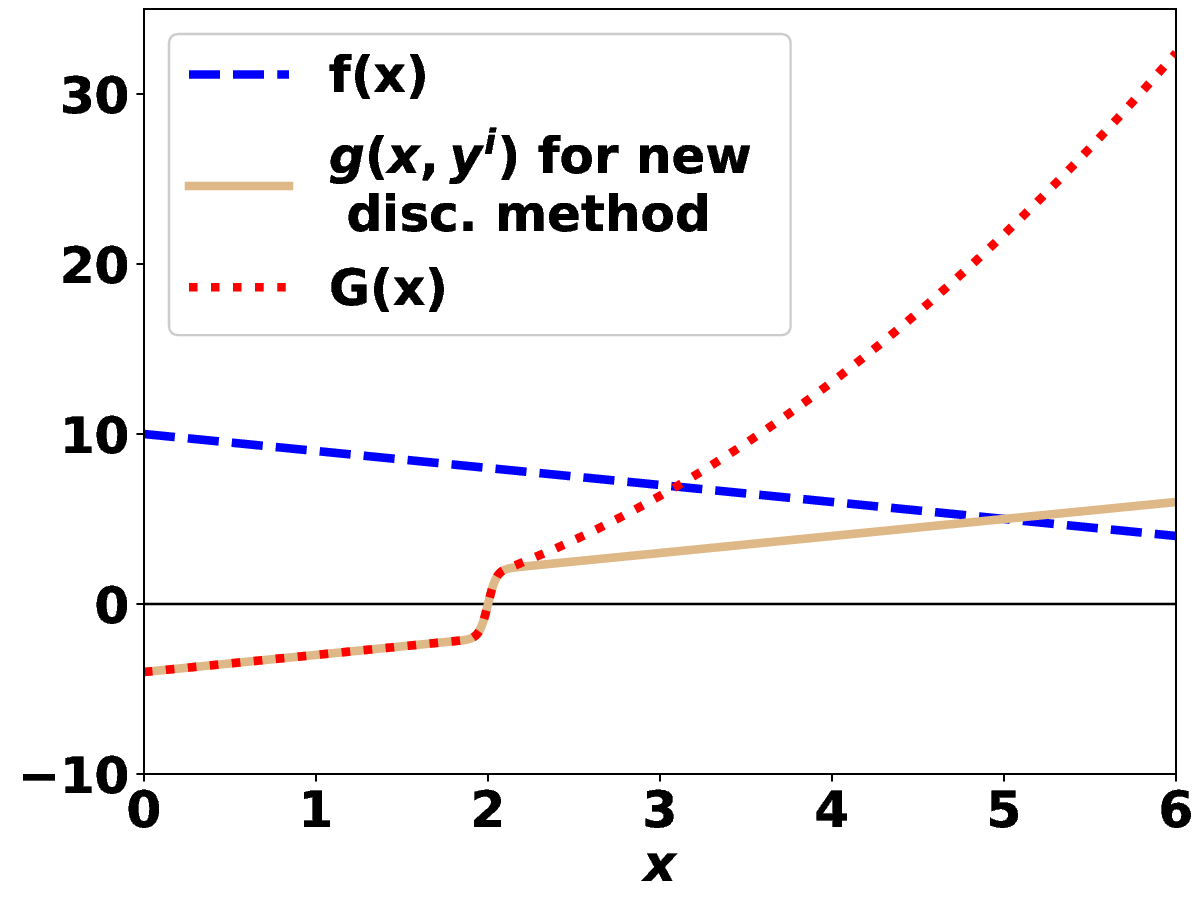}
\end{subfigure}
\caption{
Benefits of bounding-focused discretization: 
Left: objective $f(x)$ and constraint $G(x)$ for Example~\ref{exm:dp} along with discretization $Y_d = \{2,3,4,5\}$ (cf.\ \citep[Figure~3.1]{djelassi2020discretization}). 
Middle: BF method needs $27$ discretization points for its lower bounds to converge to $v^*$.
Right: our proposed bounding-focused discretization method in Section~\ref{sec:accelerated_discretization} only requires a single discretization point ($\abs{Y_d} = 1$).
}
\label{fig:plot_dp_bf}
% \vspace*{-0.1in}
\end{figure}

This paper is organized as follows.
Section~\ref{sec:accelerated_discretization} proposes new bounding-focused discretization methods, designs effective heuristic solution approaches, and presents theoretical guarantees.
Section~\ref{sec:generalized_discretization} proposes bounding-focused generalized discretization methods for~\eqref{eqn:sip} and designs effective heuristic solution strategies.
Section~\ref{sec:generalizations} outlines extensions.
Section~\ref{sec:numerical_results} presents detailed computational results that show our bounding-focused discretization methods significantly reduce the number of iterations for convergence relative to the BF algorithm.
Section~\ref{sec:conclusion} concludes with avenues for future work.
{The appendix contains a review of parametric sensitivity theory as well as omitted proofs.}

\paragraph*{Notation} Let $[n] := \{1,\dots,n\}$ and $\text{mid}(v^1,v^2,v^3)$ denote the componentwise median of vectors $v^1, v^2, v^3$. Given $\delta > 0$, $v \in \R^n$, $S \subset \R^n$, let $\norm{v}$ denote the Euclidean norm of $v$, $B_{\delta}(v)$ denote the open Euclidean ball of radius $\delta$ centered at $v$, $\text{proj}_S(v)$ denote (an element of) the Euclidean projection of $v$ onto $S$, and $\text{diam}(S)$ denote the diameter of~$S$ with respect to the Euclidean norm.
We say~\eqref{eqn:sip} is convex if $X$ is convex and $f$ and $g(\cdot,y)$ are convex on~$X$ for each $y \in Y$ (note that $Y$ and $-g(x,\cdot)$ need not be convex in this case).

\section{Bounding-focused discretization methods}
\label{sec:accelerated_discretization}

We propose new bounding-focused discretization methods for~\eqref{eqn:sip} that can achieve faster rate of convergence of lower bounds than the BF algorithm~\ref{alg:bfdisc}. 
The key idea of these discretization methods is to populate $Y_d$ with points in $Y$ that yield the highest lower bound.
In the first iteration, instead of updating $Y_d$ with a solution $y^{BF,1}$ of~\eqref{eqn:llp} at $x = x^1$ as in the BF algorithm~\ref{alg:bfdisc}, we propose to solve the following max-min problem to determine a discretization $Y_d = \{\bar{y}^1\}$ that results in the highest lower bound:
\begin{align}
\label{eqn:max_min_iter1}
\bar{y}^1 \in \argmax_{y^1 \in Y} \: & \min_{x \in X} \:\: f(x) \\
& \:\:\:\text{s.t.} \:\: g(x,y^1) \leq 0. \nonumber
\end{align}
We assume for simplicity that the maxima in all of our subproblems is attained (otherwise, we may pick any $\varepsilon$-optimal solution for some small $\varepsilon > 0$).
Techniques for solving {problem~\eqref{eqn:max_min_iter1}} are discussed in Section~\ref{subsec:solving_max_min}.
We consider two approaches for updating the discretization $Y_d$ at iteration $k > 1$.

The first approach discards the discretization $Y^{k-1}_d$ from iteration $k-1$ and determines a new discretization at iteration $k$ by solving the max-min problem:
\begin{alignat}{2}
\label{eqn:method-2}
(\bar{y}^1,\dots,\bar{y}^k) \in &\argmax_{(y^1,\dots,y^k) \in Y^k} 
\phi_k(y^1,\dots,y^k) := \:\: &&\min_{x \in X} \:\: f(x) \\
& && \:\:\:\text{s.t.} \:\: g(x,y^i) \leq 0, \quad \forall i \in [k], \nonumber
\end{alignat}
where $\phi_k: Y^k \to \mathbb{R}$ denotes the value function of the inner-minimization in problem~\eqref{eqn:method-2} at iteration $k$.
Formulation~\eqref{eqn:method-2} is inspired by the idea of strong partitioning proposed by Kannan et al.\ \cite{kannan2022learning}.
The resulting discretization {$Y^k_d := \{\bar{y}^1,\dots,\bar{y}^k\}$} at iteration $k$ yields the highest lower bound among all possible relaxations~\eqref{eqn:disc-lbp} with at most $k$ discretization points.
However, the outer-maximization in problem~\eqref{eqn:method-2} involves {$k \times d_y$} variables compared to only $d_y$ variables in problem~\eqref{eqn:max_min_iter1}.

To mitigate this increased computational burden, our second approach updates the discretization $Y^{k-1}_d := \{\bar{y}^1,\dots,\bar{y}^{k-1}\}$ at {iteration} $k-1$ by adding a single point $\bar{y}^k \in Y$ that maximizes the lower bound improvement.
This can be formulated as the max-min problem:
\begin{alignat}{2}
\label{eqn:method-1}
\bar{y}^k \in \argmax_{y^k \in Y} \:\: & \psi_k\bigl(y^k;Y^{k-1}_d\bigr) := &&\min_{x \in X} \:\: f(x) \\
& &&\:\:\:\text{s.t.} \:\: g(x,y) \leq 0, \quad \forall y \in Y^{k-1}_d, \nonumber \\
& &&\quad\quad\:\: g(x,y^k) \leq 0, \nonumber
\end{alignat}
where $\psi_k: Y \to \mathbb{R}$ denotes the value function of the inner minimization, and $Y^k_d = Y^{k-1}_d \cup \{\bar{y}^k\}$ is the discretization specified at iteration $k$.
Problem~\eqref{eqn:method-1} is a greedy approximation of problem~\eqref{eqn:method-2}. 
We introduce two variants of these discretization methods in Section~\ref{subsec:outline-opt-disc-alg} and establish theoretical guarantees for our bounding-focused discretization methods in  Section~\ref{subsec:convergence_guarantees}.

In contrast with the approach of Tsoukalas and Rustem~\cite{tsoukalas2011feasible}, which treats the violation of the semi-infinite constraint~\eqref{eqn:sic} and the objective of~\eqref{eqn:sip} as two competing objectives, problems~\eqref{eqn:method-2} and~\eqref{eqn:method-1} directly optimize the discretization for the best lower bound.
Baltean-Lugojan et al.\ \cite{baltean2019scoring} propose bounding-focused cuts with a similar flavor but tailored for outer-approximating semidefinite programs, which are a family of \textit{convex} SIPs.
Similar to problem~\eqref{eqn:method-2}, Coniglio and Tieves~\cite{coniglio2015generation} formulate bounding-focused cut selection for integer linear programs as a bilevel problem and reformulate it as a single-level bilinear program using linear programming duality (cf.\ Section~\ref{subsec:solving_max_min}).
Paulus et al.\ \cite{paulus2022learning} consider bounding-focused cut selection for mixed-integer linear programs (MILPs), where they use explicit enumeration to choose the best bounding-focused cut from a \textit{finite} list of candidate cuts.
Finally, Das et al.\ \cite{das2022near} use the well-known result that problem~\eqref{eqn:method-2} with $k = d_x$ discretization points yields an exact reformulation of \textit{convex} SIPs under mild assumptions (see Proposition~\ref{prop:alg_opt} in Section~\ref{subsec:convergence_guarantees}). They use simulated annealing to solve the resulting max-min problem and report encouraging results on small-scale convex SIPs.

\subsection{Outline of bounding-focused discretization algorithms}\label{subsec:outline-opt-disc-alg}

Algorithm~\ref{alg:prototype_disc} outlines a prototype bounding-focused discretization method for~\eqref{eqn:sip}.
Similar to the BF algorithm~\ref{alg:bfdisc}, at each iteration, it solves~\eqref{eqn:disc-lbp} to \textit{global} optimality to determine a candidate solution $x^k$ and a corresponding lower bound $LBD^k$.
It then solves the lower-level problem~\eqref{eqn:llp} with $x = x^k$ to \textit{global} optimality to determine a point $\hat{y}^k$, which is used to check if $x^k$ is $\varepsilon_f$-feasible for~\eqref{eqn:sip}.
The key difference between Algorithm~\ref{alg:prototype_disc} and the BF algorithm~\ref{alg:bfdisc} is on lines 8--11 of Algorithm~\ref{alg:prototype_disc}. 
If $x^k$ is not $\varepsilon_f$-feasible, Algorithm~\ref{alg:prototype_disc} solves a max-min problem (heuristically) to identify new points that may be used to update the discretization $Y_d$ when a sufficient bound increase condition holds.
In contrast, the BF algorithm~\ref{alg:bfdisc} always adds $\hat{y}^k$ to the discretization.

We consider four realizations of Algorithm~\ref{alg:prototype_disc} that only vary on lines 8--11: 
\texttt{OPT}, \texttt{GREEDY}, \texttt{2GREEDY}, and \texttt{HYBRID}.

\begin{itemize}
\item \texttt{OPT} seeks to discard the discretization $Y^{k-1}_d$ at iteration {$k-1$} and replace it with a fresh discretization obtained by solving problem~\eqref{eqn:method-2}.
It initializes the solution of this max-min problem with $Y^{k-1}_d \cup \{\hat{y}^k\}$.

\item \texttt{GREEDY} adds a single point $\bar{y}^k$ to the discretization $Y^{k-1}_d := \{\bar{y}^1,\dots,\bar{y}^{k-1}\}$ at iteration $k-1$ by solving problem~\eqref{eqn:method-1} with the initialization $\hat{y}^k$.

\item \texttt{2GREEDY} first updates the discretization at iteration $k-1$ with $\hat{y}^k$, i.e., $Y^{k-1}_d \leftarrow Y^{k-1}_d \cup \{\hat{y}^k\}$.
It then solves problem~\eqref{eqn:method-1} to find another point to add to the discretization using a perturbation of $\hat{y}^k$ as the initialization.

\item \texttt{HYBRID} mitigates the increasing computational burden of Algorithm~\texttt{OPT} as the iteration count $k$ increases. 
For the first $K$ iterations, it solves problem~\eqref{eqn:method-2} to try and determine a fresh discretization with sufficient lower bound improvement (similar to \texttt{OPT}).
From iteration $K+1$, it then switches to the \texttt{GREEDY} strategy and solves problem~\eqref{eqn:method-1} to try and find a single best point to add to the previous discretization $Y^{k-1}_d$.
\end{itemize}

\noindent 
All four realizations of Algorithm~\ref{alg:prototype_disc} use the point $\hat{y}^k$ to either construct an initial guess, or to add to the discretization.
Algorithm~\texttt{2GREEDY} looks to add two discretization points per iteration (including $\hat{y}^k$) with the goal of reducing the number of global solves of~\eqref{eqn:disc-lbp} and~\eqref{eqn:llp} and the overall time required for the sequence of lower bounds $\{LBD^k\}$ to converge to $v^*$.

{We emphasize that except on line~9 of Algorithm~\ref{alg:prototype_disc}, we do \textit{not} require the inner-minimizations and outer-maximizations of our max-min problems to be solved to global optimality.
Although the additional global solve at each iteration of Algorithm~\ref{alg:prototype_disc} may seem expensive, it is worth noting that if the condition on line~10 is satisfied, then the global solve on line~3 can be skipped in the subsequent iteration.
Furthermore, this additional global solve need not be performed at every iteration; it can be performed infrequently (e.g., only every tenth iteration $k$) without compromising convergence guarantees or the validity of the lower bound $LBD^k$.
Finally, this global solve becomes redundant for convex SIPs under mild assumptions, since the inner-minimization problems are convex programs in this case.
In practice, we skip this extra global solve and instead rely heuristically on the local solution of the inner-minimization problem to verify the sufficient bound increase condition.
}

\begin{algorithm}[t]
\caption{Prototype bounding-focused discretization algorithm}
\label{alg:prototype_disc}
{
\begin{algorithmic}[1]
\State \textbf{Input}: feasibility tolerance $\varepsilon_{f} \geq 0$, minimum bound improvement $\delta \geq 0$, and initial discretization $Y_d = \emptyset$.

\For{$k = 1, 2, \dots$}

\State Solve problem~\eqref{eqn:disc-lbp} globally to get solution $x^k$, lower bound $LBD^k$.

\State Solve problem~\eqref{eqn:llp} with $x = x^k$ globally to get solution $\hat{y}^k \in Y$.

\If{$G(x^k) \leq \varepsilon_{f}$} 

\State \textbf{Terminate} with $\varepsilon_f$-feasible solution $x^k$ to~\eqref{eqn:sip}.

\Else

\State Solve a max-min problem (heuristically) to get $\{\bar{y}^{k,1},\dots,\bar{y}^{k,n_k}\}$.

\State Solve the inner-min problem to \textit{global} optimality at the max-min

\Statex \hspace*{0.4in} solution $\{\bar{y}^{k,1},\dots,\bar{y}^{k,n_k}\}$. Let $\eta^*_k$ denote its global optimal value.

\vspace*{0.03in}
\If{$\eta^*_k \geq LBD^k + \delta$}

\State Update the discretization $Y_d$ using $\{\bar{y}^{k,1},\dots,\bar{y}^{k,n_k}\}$.

\Else

\State Set $Y_d \leftarrow Y_d \cup \{ \hat{y}^k \}$.

\EndIf

\vspace*{0.03in}
\EndIf

\EndFor

\end{algorithmic}
}
\end{algorithm}
\subsection{Convergence guarantees}
\label{subsec:convergence_guarantees}

We show the sequence $\{LBD^k\}_k$ of lower bounds determined by Algorithms \texttt{OPT}, \texttt{GREEDY}, \texttt{2GREEDY}, and \texttt{HYBRID} converge to $v^*$ under different assumptions.

Our first result is the most useful one in practice.
It allows our max-min formulations to be solved using \textit{any} heuristic under the following conditions: $(i)$ the inner-minimization problem is solved to global optimality \textit{once} at the candidate max-min solution (see line~9 of Algorithm~\ref{alg:prototype_disc}), and $(ii)$ the minimum bound improvement required at each iteration $\delta > 0$.

% \vspace*{-0.1in}
\begin{theorem}
\label{thm:conv_lbd_disc}
Consider Algorithm~\ref{alg:prototype_disc} with $\varepsilon_f = 0$, $\delta > 0$. Suppose the {discretization $Y_d$} is updated using Algorithm \texttt{OPT}, \texttt{GREEDY}, \texttt{2GREEDY}, or \texttt{HYBRID}.
Then $\underset{k \to \infty}{\lim} LBD^k = v^*$.
\end{theorem}
\begin{proof}
% \vspace*{-0.3in}
Since $f$ and $g$ are continuous, $X$ and $Y$ are compact, and~\eqref{eqn:sip} is assumed to be feasible, the optimal value $v^*$ is finite and bounded below by $\min_{x \in X} f(x) > -\infty$.
Line~11 of Algorithm~\ref{alg:prototype_disc} updates the discretization $Y_d$ using the points $\bar{y}^{k,1},\dots,\bar{y}^{k,n_k}$ only if this candidate discretization increases the lower bound in iteration $k+1$ by at least~$\delta$. 
Since $v^* - LBD^0 = v^* - \min_{x \in X} f(x) < \infty$, line~11 of Algorithm~\ref{alg:prototype_disc} can be executed only finitely many times before the lower bound converges to $v^*$.
Therefore, {if Algorithm~\ref{alg:prototype_disc} does not converge in a finite number of iterations, then} line~13 is executed for all $k$ large enough and the asymptotic behavior of Algorithm~\ref{alg:prototype_disc} is the same as that of the BF algorithm~\ref{alg:bfdisc}.
The result that $LBD^k \rightarrow v^*$ then follows from Lemma~2.2 of Mitsos~\cite{mitsos2011global} (cf.\ Theorem~3.1 of Harwood et al.\ \cite{harwood2021note}).
\end{proof}

The remaining results in this section are mainly of theoretical interest since they assume our max-min formulations are solved to \textit{global} optimality (which is in general impractical because this may be as hard as solving~\eqref{eqn:sip} itself).

The following result identifies favorable properties of Algorithm~\texttt{OPT} when the max-min problem~\eqref{eqn:method-2} is solved to global optimality at each iteration.

% \vspace*{-0.1in}
\begin{proposition}
\label{prop:alg_opt}
Consider Algorithm~\texttt{OPT} with $\varepsilon_f = \delta = 0$, and suppose the {max-min} problem~\eqref{eqn:method-2} is solved to \textit{global} optimality.
Then $\underset{k \to \infty}{\lim} LBD^k = v^*$.
Moreover, suppose \eqref{eqn:sip} is convex and $\exists \bar{x} \in X$ such that $G(\bar{x}) < 0$.
Then $LBD^k = v^*$ for each iteration $k \geq d_x$, i.e., Algorithm~\texttt{OPT} converges in at most $d_x$ iterations.
\end{proposition}

% \vspace*{-0.1in}
Our next results establish convergence rates of the sequence of lower bounds generated using Algorithms BF and \texttt{OPT}.
We begin with the following result on the number of iterations required for the BF algorithm to converge.

% \vspace*{-0.1in}
\begin{proposition}
\label{prop:convrate_bf}
Consider the BF algorithm~\ref{alg:bfdisc} with $\varepsilon_f > 0$.
Suppose $\{g(\cdot,y)\}_{y \in Y}$ is uniformly Lipschitz continuous on $X$ with Lipschitz constant $L_{g,x} > 0$, i.e.,
\[
\abs{g(x,y) - g(\bar{x},y)} \leq L_{g,x} \norm{x - \bar{x}}, \quad \forall x, \bar{x} \in X, \: y \in Y.
\]
Additionally, suppose $\{g(x,\cdot)\}_{x \in X}$ is uniformly Lipschitz continuous on $Y$ with Lipschitz constant $L_{g,y} > 0$, i.e.,
\[
\abs{g(x,y) - g(x,\bar{y})} \leq L_{g,y} \norm{y - \bar{y}}, \quad \forall y, \bar{y} \in Y, \: x \in X.
\]
Then Algorithm~\ref{alg:bfdisc} terminates in at most $N$ iterations, where
\[
N := \min\bigg\{\bigg\lceil\left(\frac{\textup{diam}(Y) L_{g,y}}{\varepsilon_f} + 1\right)^{d_y} \bigg\rceil, \bigg\lceil\left(\frac{\textup{diam}(X) L_{g,x}}{\varepsilon_f} + 1\right)^{d_x}\bigg\rceil\bigg\}.
\]
\end{proposition}

The estimate of the number of iterations needed for the BF algorithm to converge depends exponentially on the dimensions of~\eqref{eqn:sip}.
When $d_y \ll d_x$, as in many applications, the term $\big(\tfrac{\textup{diam}(Y) L_{g,y}}{\varepsilon_f} + 1\big)^{d_y}$ may be significantly smaller than $\big(\tfrac{\textup{diam}(X) L_{g,x}}{\varepsilon_f} + 1\big)^{d_x}$.
Proposition~\ref{prop:convrate_bf} may be sharpened by estimating the number of balls needed to cover an enlargement of $\{y^*(x) : x \in X\}$ instead of an enlargement of~$Y$.
We now link Proposition~\ref{prop:convrate_bf} to the convergence rate of the sequence of lower bounds furnished by Algorithms BF and \texttt{OPT}.

% \vspace*{-0.1in}
\begin{theorem}
\label{thm:calmness}
Suppose the value function $V(z) := \min \big\{ f(x) : x \in X, \: G(x) \leq z \big\}$
is Lipschitz continuous on $[0,\bar{\varepsilon}]$ with Lipschitz constant $L_V > 0$, {for some $\bar{\varepsilon} > 0$}.
Additionally, {suppose $\varepsilon_f = 0$, the family $\{g(\cdot,y)\}_{y \in Y}$ is uniformly Lipschitz continuous on $X$ with Lipschitz constant $L_{g,x} > 0$, and the family $\{g(x,\cdot)\}_{x \in X}$ is uniformly Lipschitz continuous on $Y$ with Lipschitz constant $L_{g,y} > 0$.}
Consider Algorithm~\texttt{OPT} with $\delta = 0$, and assume that the max-min problem~\eqref{eqn:method-2} is solved to global optimality.
Then, for any $\varepsilon \in (0,L_V\bar{\varepsilon})$, the lower bounds produced by Algorithms BF and \texttt{OPT} exceed $v^* - \varepsilon$ whenever $k \geq \min\left\lbrace\left\lceil\left(\frac{\textup{diam}(Y) L_{g,y} L_V}{{\varepsilon}} + 1\right)^{d_y}\right\rceil, \left\lceil\left(\frac{\textup{diam}(X) L_{g,x} L_V}{{\varepsilon}} + 1\right)^{d_x}\right\rceil\right\rbrace$.
\end{theorem}

The Lipschitz assumption in Theorem~\ref{thm:calmness} is satisfied by convex SIPs when the objective $f$ is Lipschitz and Slater's condition holds (see Corollary~2 to Theorem~6.3.2 in Clarke~\cite{clarke1990optimization}).
Chapter~6 of Clarke~\cite{clarke1990optimization} also details other constraint qualifications under which this Lipschitz assumption holds.

The following example illustrates that the above rate of convergence cannot be improved in general.
Specifically, this example shows that discretization-based lower bounding methods involving~\eqref{eqn:disc-lbp} (such as the BF algorithm and \textit{any} realization of Algorithm~\ref{alg:prototype_disc}, including \texttt{OPT}) may require an exponential number of discretization points in the problem dimensions to converge.
Although this behavior is expected, we are not aware of such an example in the SIP literature (similar examples are known for the classical cutting-plane method in convex optimization, see~\citep[Ex.\ 3.3.1]{nesterov2018lectures},~\citep[Ex.\ 1]{hijazi2014outer}).

% \vspace*{-0.15in}
\begin{example}{(Based on~\citep[Example 1]{hijazi2014outer})}
\label{exm:hijazi}
Consider~\eqref{eqn:sip} with $X = [-1,1]^{d_x}$, {$Y = \{y \in \R^{d_x} : \norm{y}^2 = d_x-1\}$}, $f(x) = -\norm{x}^2$, and $g(x,y) = \sum_{i=1}^{d_x} (x_i - y_i)y_i$.
Note that this semi-infinite constraint~\eqref{eqn:sic} may be reformulated as the convex constraint $\norm{x} \leq \sqrt{d_x-1}$.
Any $x \in X$ with $\norm{x} = \sqrt{d_x - 1}$ solves~\eqref{eqn:sip} with $v^* = 1 - d_x$.

Lemma~2.1 of Hijazi et al.\ \cite{hijazi2014outer} implies any discretization point $y \in Y$ can exclude at most one vertex of the cube $X$.
Because every vertex of $X$ is a solution to~\eqref{eqn:disc-lbp} with the discretization $Y_d = \emptyset$, \textit{any} discretization-based lower bounding algorithm that solves~\eqref{eqn:disc-lbp} requires exponentially many discretization points in the dimension $d_x$ for its sequence of lower bounds $\{LBD^k\}_k$ to exceed $v^* - 0.5$.
\end{example}

We end this section by demonstrating the sequence $\{LBD^k\}_k$ of lower bounds determined by Algorithm \texttt{GREEDY} \textit{may not} converge to $v^*$ when $\delta = 0$ (i.e., if the fallback to the BF algorithm on line~13 of Algorithm~\ref{alg:prototype_disc} is not used), even if the max-min problems~\eqref{eqn:method-1} are always solved to global optimality.

\vspace*{0.1in}
\noindent Example~\ref{exm:hijazi}: Consider Algorithm~\texttt{GREEDY} with $\varepsilon_f = \delta = 0$, and suppose the max-min problems~\eqref{eqn:method-1} are solved to global optimality at each iteration.
Assume without loss of generality that the incumbent solution $x^1 = \mathbf{1}$ at the first iteration, where $\mathbf{1}$ denotes a vector of ones.
Then $\hat{y}^1 = \sqrt{1 - \tfrac{1}{d_x}} \mathbf{1}$.
The sequence of vectors $\{\bar{y}^k\}$ with $\bar{y}^j = -\hat{y}^1$ for each $j \geq 1$ is \textit{one} sequence of optimal solutions to the max-min problems~\eqref{eqn:method-1} solved by Algorithm~\texttt{GREEDY}. However, $LBD^k = -d_x < v^*$ for each $k \geq 1$ since $x^1$ remains the incumbent solution at each iteration $k$.
Therefore, Algorithm~\texttt{GREEDY} \textit{may not} converge when $\delta = 0$.

\subsection{Solving the max-min problems}
\label{subsec:solving_max_min}

While solving the max-min problems~\eqref{eqn:method-2} or~\eqref{eqn:method-1} to global optimality is clearly desirable, this may be as difficult as solving~\eqref{eqn:sip} itself. 
Hence, we exploit the fact that solving these problems heuristically is sufficient to obtain a discretization with desirable convergence properties (see Theorem~\ref{thm:conv_lbd_disc} and the discussion in Section~\ref{subsec:outline-opt-disc-alg}).
Numerical experiments in Section~\ref{sec:numerical_results} empirically demonstrate that our heuristic approaches for solving problems~\eqref{eqn:method-2} or~\eqref{eqn:method-1} almost always yield better discretizations with faster rate of convergence than the BF algorithm.

\paragraph*{Properties of the value functions~$\phi_k$ and~$\psi_k$ in problems~\eqref{eqn:method-2} and~\eqref{eqn:method-1}}
Before we outline our approach for solving problems~\eqref{eqn:method-2} and~\eqref{eqn:method-1}, we plot the value functions~$\phi_k$ and $\psi_k$ of these problems for some examples from the literature.
We consider the following three examples in addition to Example~\ref{exm:dp}.

% \vspace*{-0.15in}
\begin{example}{\citep[Example 2.1]{seidel2020adaptive}}
\label{exm:seidel}
Consider~\eqref{eqn:sip} with $d_x = 2$, $d_y = 1$, $X = [-1,1]^2$, $Y = [-1,1]$, $f(x) = -x_1 + 1.5x_2$, and $g(x,y) = -y^2_1 + 2y_1x_1 - x_2$.
The global solution is $x^* = \bigl( \frac{1}{3}, \frac{1}{9} \bigr)$ with $v^* = -\frac{1}{6}$.
\end{example}

% \vspace*{-0.3in}
\begin{example}{\citep[Example 2.1]{tsoukalas2011feasible}}
\label{exm:tsou}
Consider~\eqref{eqn:sip} with $d_x = 1$, $d_y = 1$, $X = [-6,6]$, $Y = [-6,6]$, $f(x) = 10 - x_1$, and $g(x,y) = -x^4_1 + x^2_1 - x^2_1y^2_1 + 2x^3_1y_1 - 4$. 
The global solution is $x^* = 2$ with $v^* = 8$.
\end{example}

% \vspace*{-0.35in}
\begin{example}{\citep[Example (H)]{mitsos2009test}}
\label{exm:wath}
Consider~\eqref{eqn:sip} with $d_x = 2$, $d_y = 1$, $X = [0,1] \times [-10^3, 10^3]$, $Y = [-1,1]$, $f(x) = x_2$, and $g(x,y) = -(x_1 - y_1)^2 - x_2$.
Any $x^* = (\bar{x}_1, 0)$ with $\bar{x}_1 \in [0,1]$ is a global solution with optimal value $v^* = 0$.
\end{example}

\begin{figure}[t]
\centering
\begin{subfigure}{0.245\textwidth}
\includegraphics[width=0.98\columnwidth]{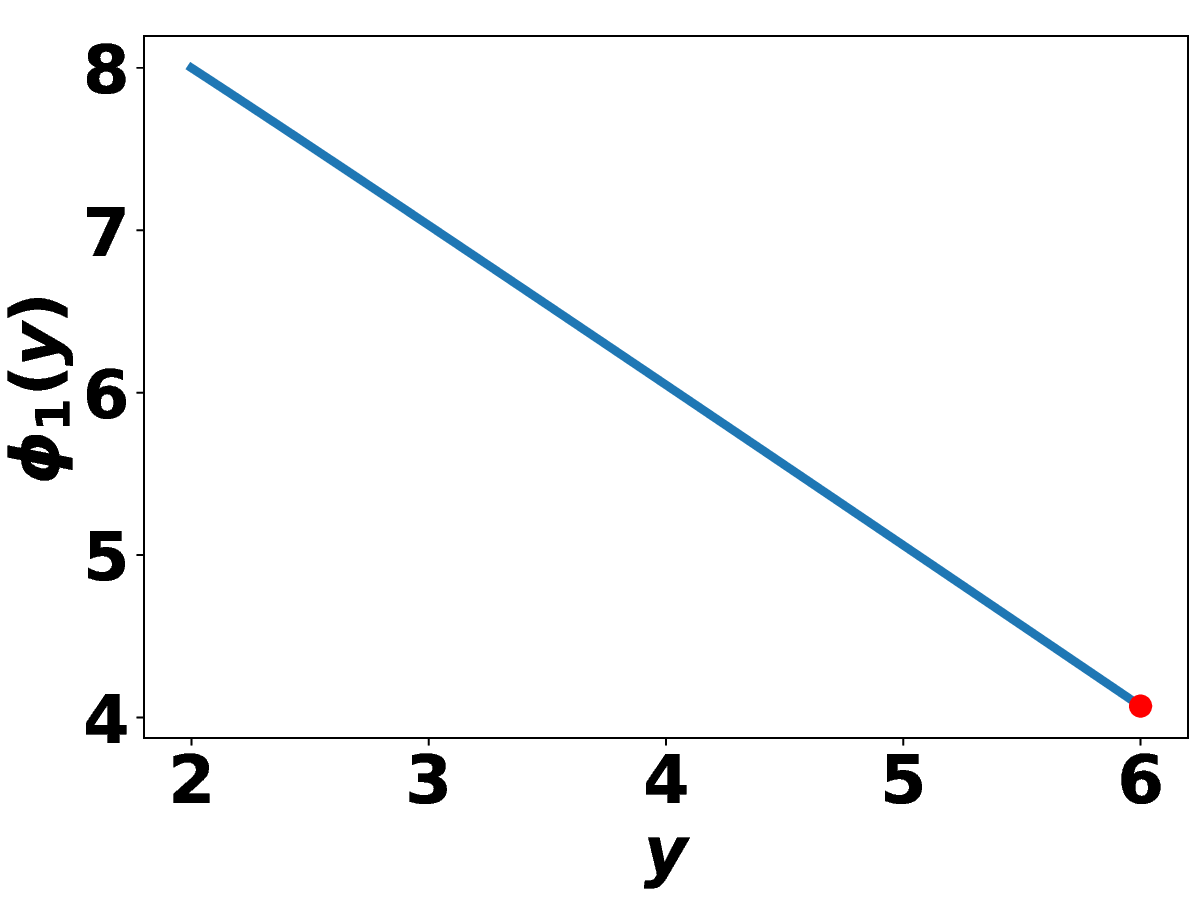}
\caption{Example~\ref{exm:dp}}
\end{subfigure}%
\begin{subfigure}{0.245\textwidth}
\includegraphics[width=0.98\columnwidth]{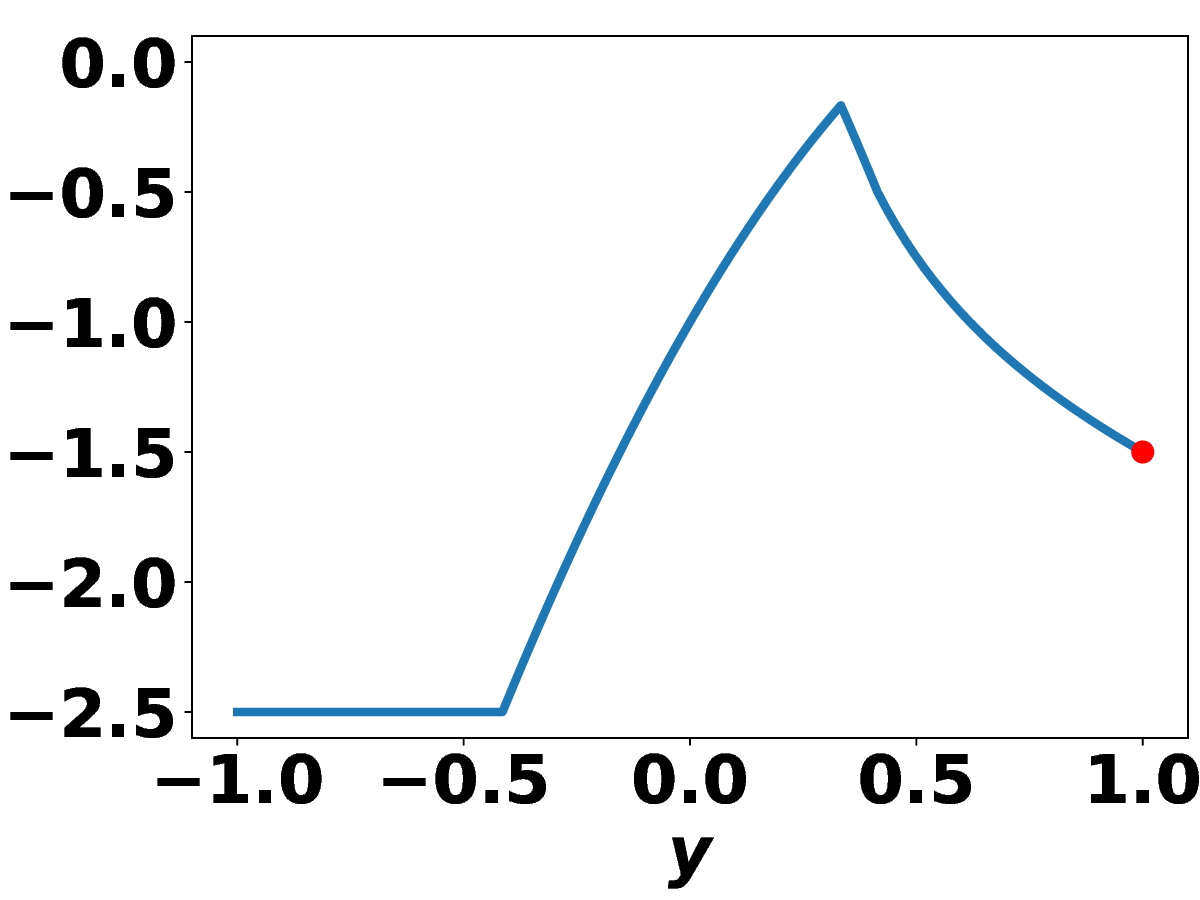}
\caption{Example~\ref{exm:seidel}}
\end{subfigure}%
\begin{subfigure}{0.245\textwidth}
\includegraphics[width=0.98\columnwidth]{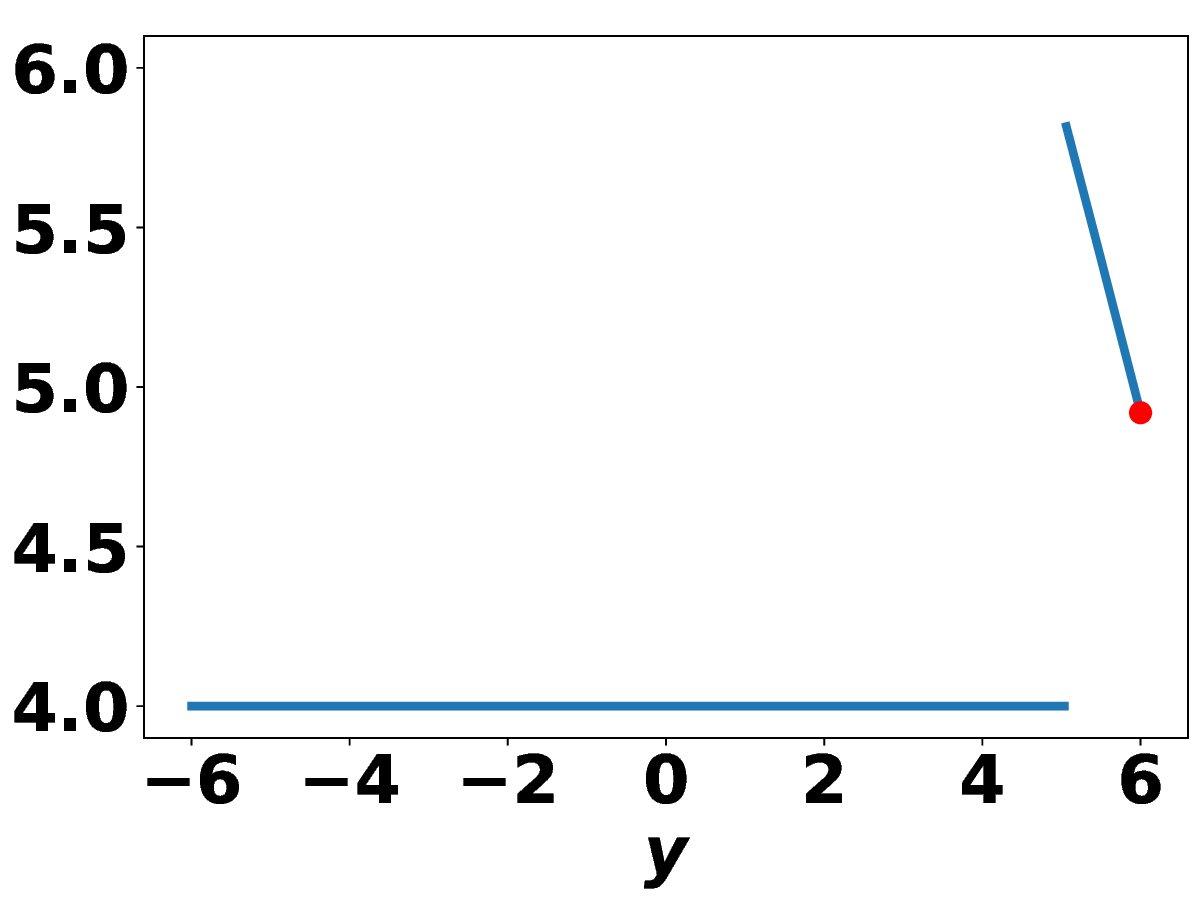}
\caption{Example~\ref{exm:tsou}}
\end{subfigure}%
\begin{subfigure}{0.245\textwidth}
\includegraphics[width=0.98\columnwidth]{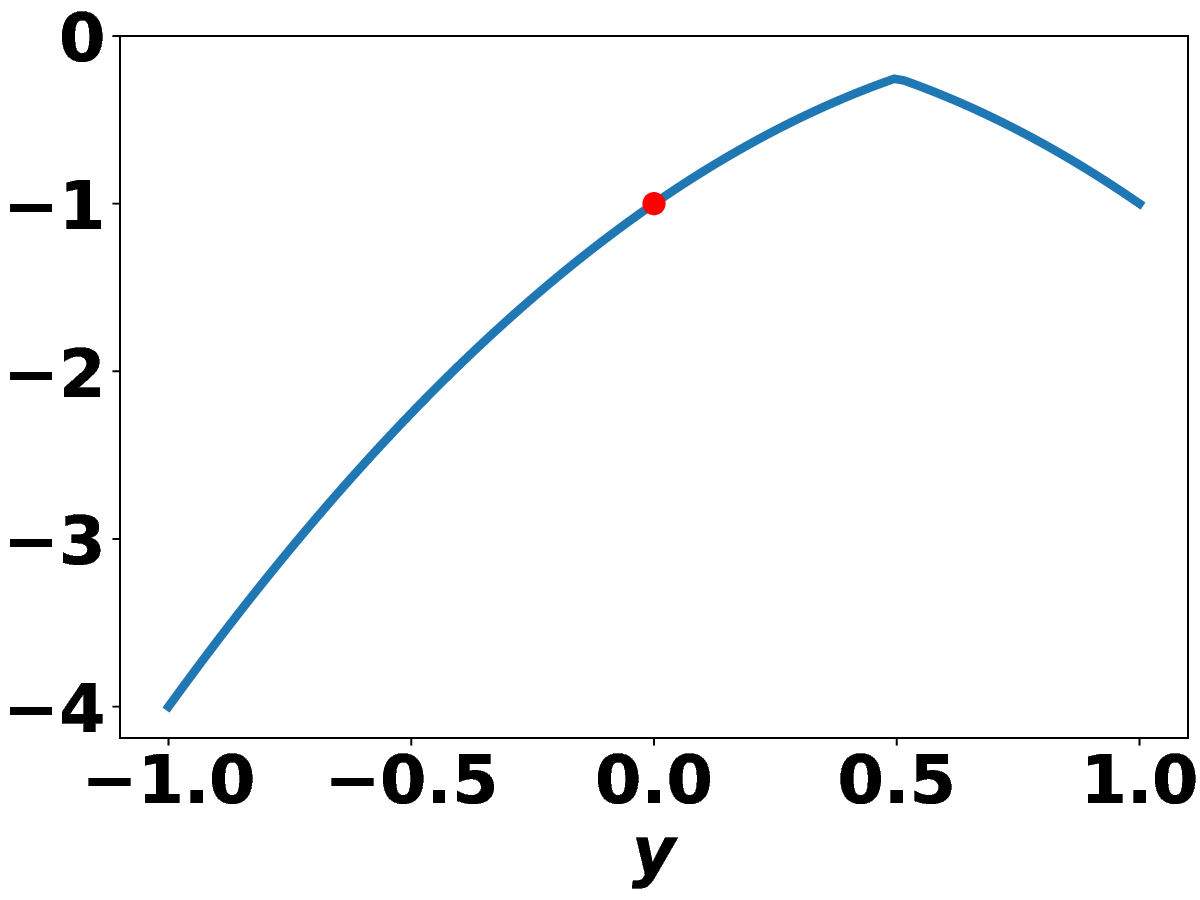}
\caption{Example~\ref{exm:wath}}
\end{subfigure}
\caption{Value functions $\phi_1$ (note the discontinuity for Example~\ref{exm:tsou}). The red dot indicates $\phi_1$ at the point $y^{BF,1}$ determined by the BF algorithm at iteration~$1$.}
\label{fig:value-funcs}
% \vspace*{-0.15in}
\end{figure}

% \vspace*{-0.15in}
Figure~\ref{fig:value-funcs} plots the (global) value function~$\phi_1$ for Examples~\ref{exm:dp},~\ref{exm:seidel},~\ref{exm:tsou}, and~\ref{exm:wath}.
It illustrates that $\phi_1$ may be nonconcave, nondifferentiable, or even discontinuous with large flat regions.
Additionally, the supremum in problem~\eqref{eqn:max_min_iter1} is not attained for Example~\ref{exm:tsou}.
Figure~\ref{fig:value-funcs} also empirically illustrates that the BF point $y^{BF,1}$ provides a good initial guess for solving problem~\eqref{eqn:max_min_iter1}.

Figure~\ref{fig:value-funcs_2} shows that the (global) value functions $\psi_k$ in Algorithm~\texttt{GREEDY} may become increasingly challenging to optimize over as $k$ increases\footnote{We do not plot $\psi_2$ and $\psi_3$ for Examples~\ref{exm:dp} and~\ref{exm:seidel} as $\uset{y \in Y}{\max} \: \phi_1(y) = v^*$ for these instances.}; however, solving the lower-level problem~\eqref{eqn:llp} at incumbent (lower bounding) solutions {$x^k$} empirically continues to yield a good initial guess {$\hat{y}^k$}.

Overall, Figures~\ref{fig:value-funcs} and~\ref{fig:value-funcs_2} highlight that exploiting (generalized) gradient information can yield effective heuristics for solving problems~\eqref{eqn:method-2} and~\eqref{eqn:method-1}.

\begin{figure}[t]
\centering
\begin{subfigure}{\textwidth}
\includegraphics[width=0.30\linewidth]{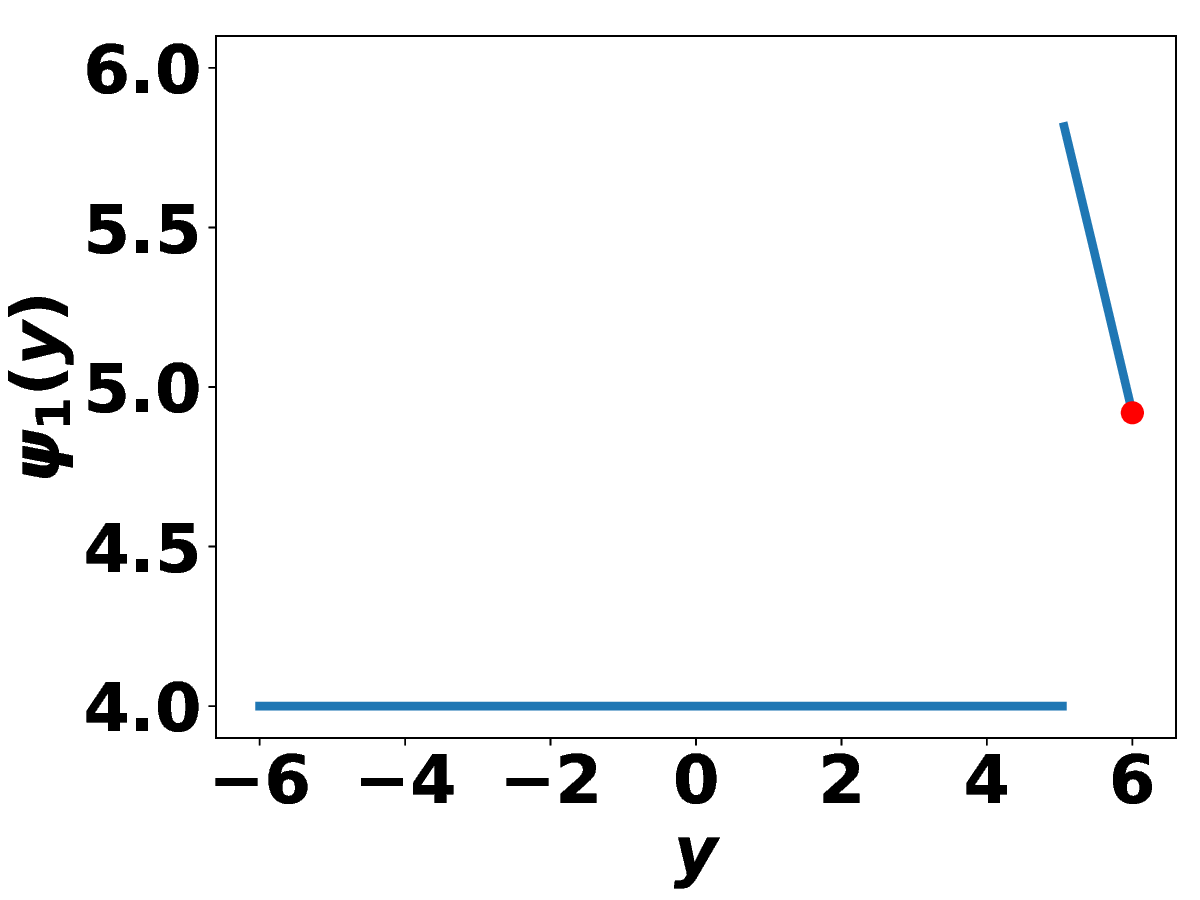}%
\hfill
\includegraphics[width=0.30\linewidth]{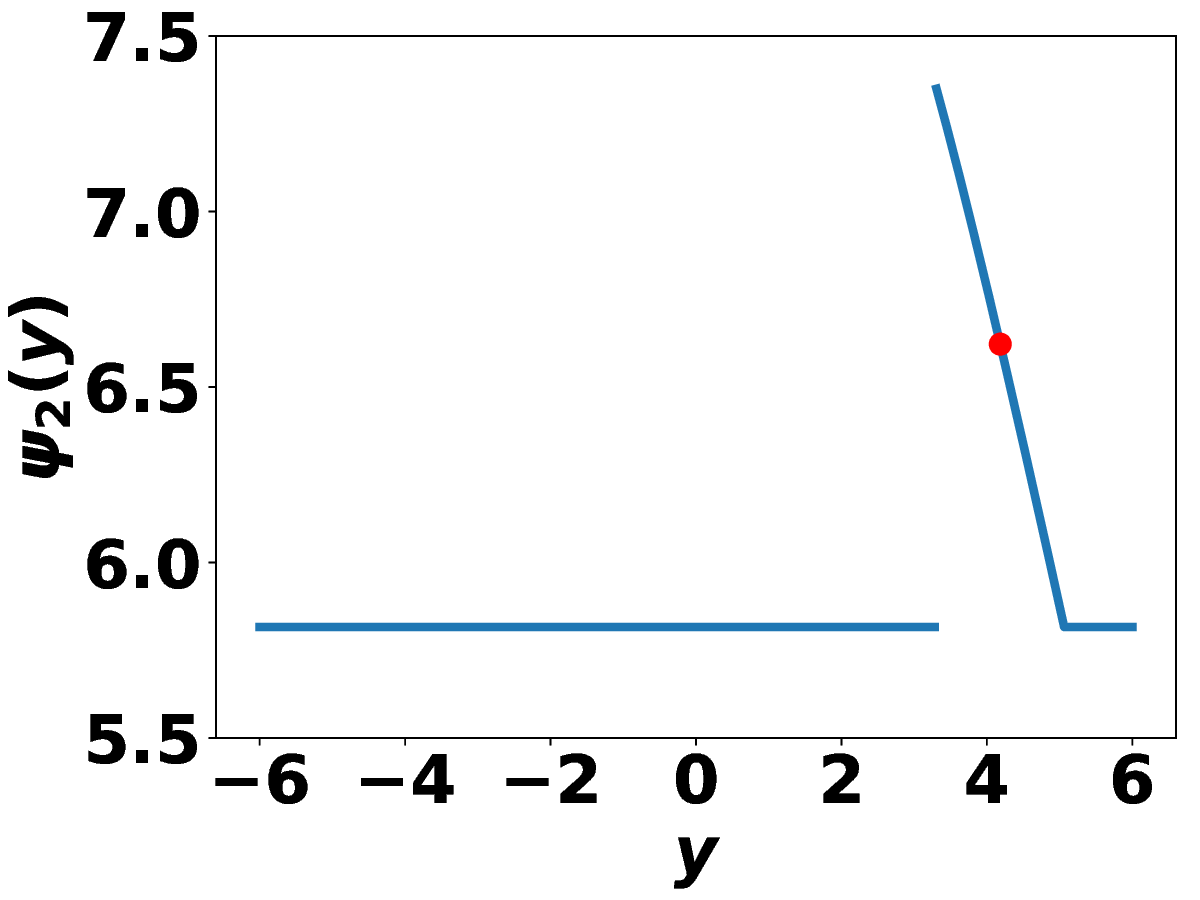}%
\hfill
\includegraphics[width=0.30\linewidth]{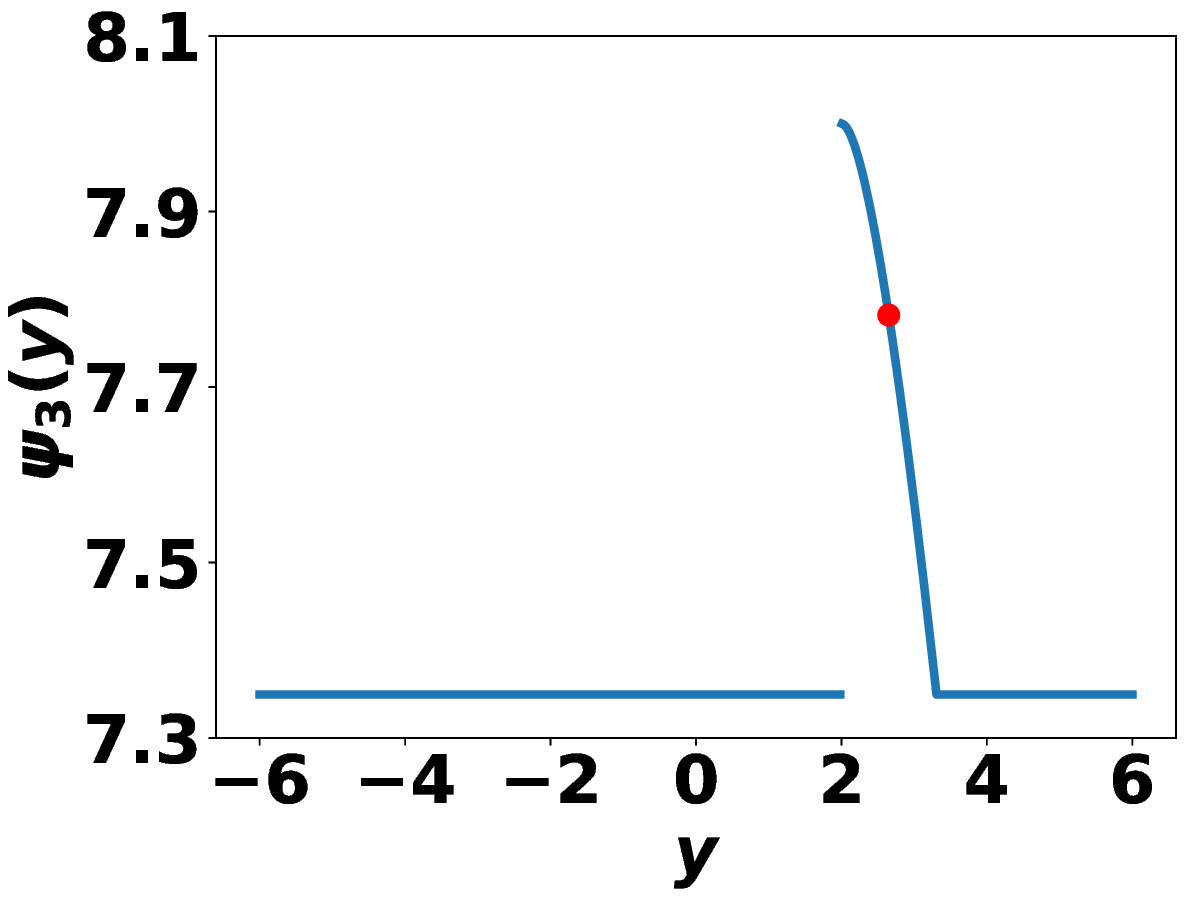}
\caption{Example~\ref{exm:tsou}}
\end{subfigure}\\
\begin{subfigure}{\textwidth}
\includegraphics[width=0.30\linewidth]{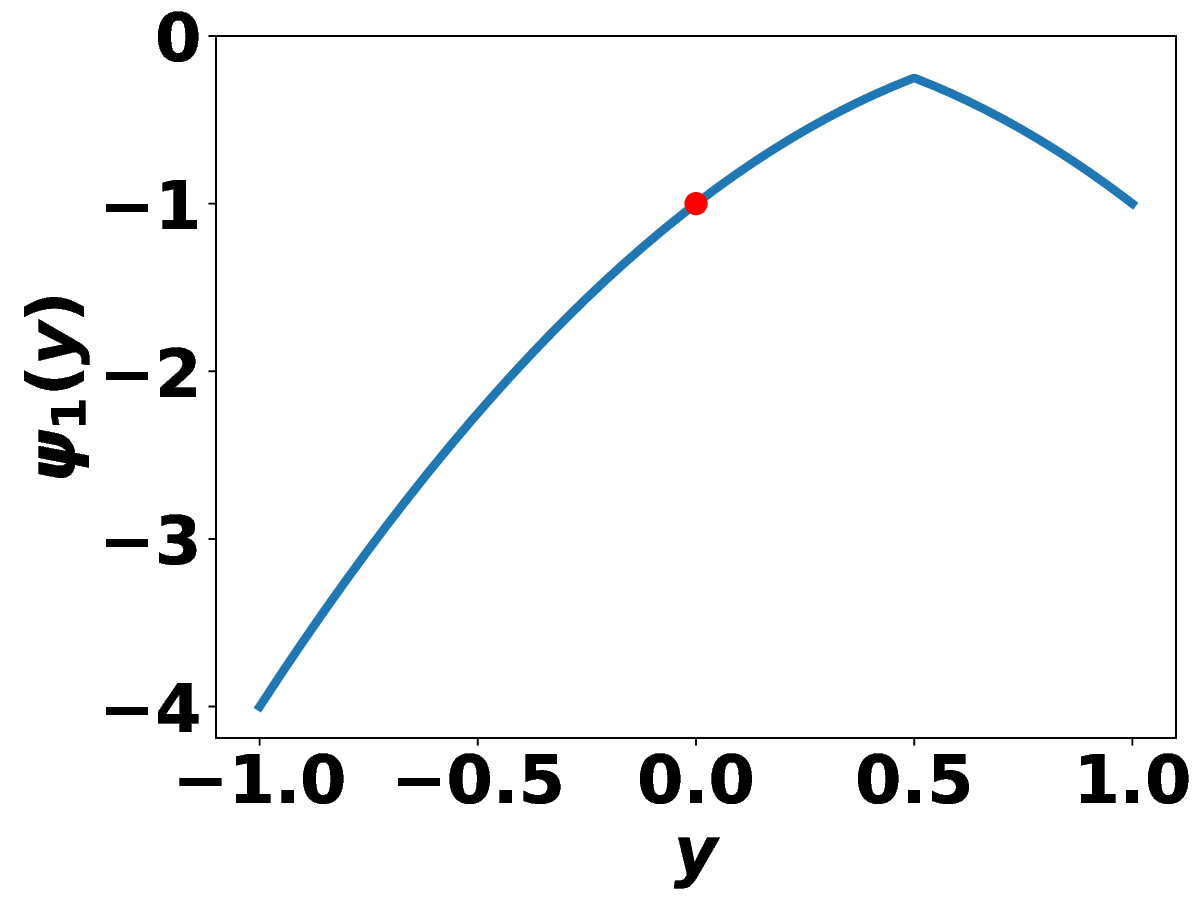}%
\hfill
\includegraphics[width=0.30\linewidth]{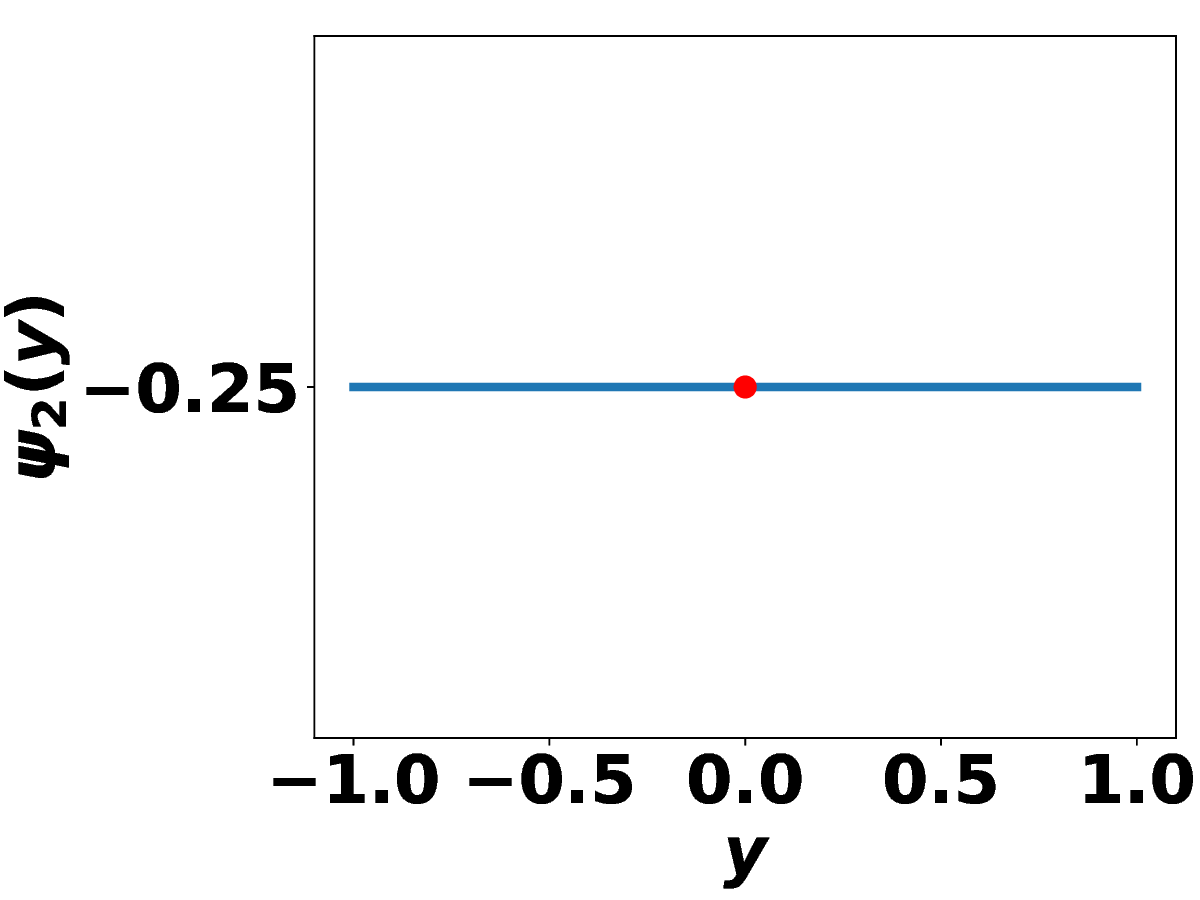}%
\hfill
\includegraphics[width=0.30\linewidth]{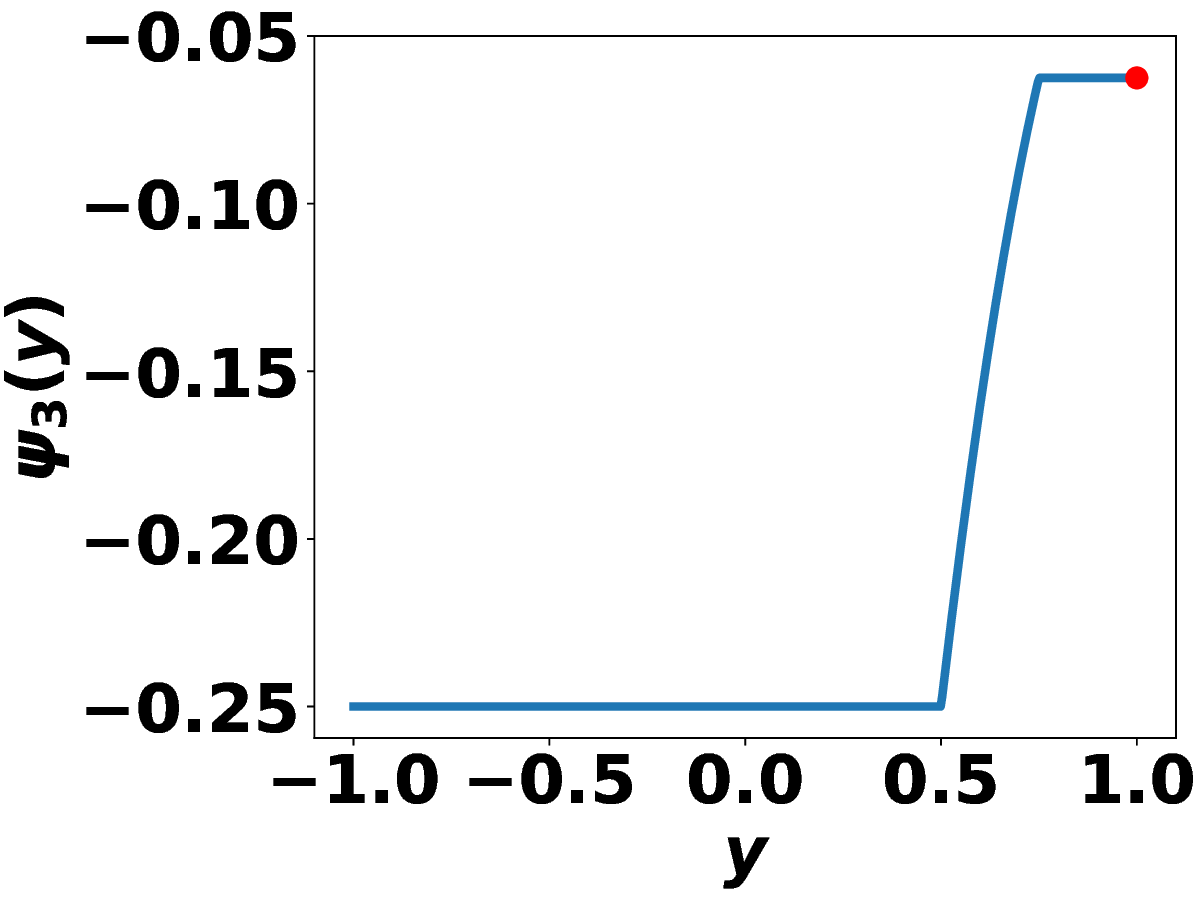}
\caption{Example~\ref{exm:wath}}
\end{subfigure}
\caption{Value functions $\psi_1, \psi_2, \psi_3$ in the first three iterations of problem~\eqref{eqn:method-1} for Examples~\ref{exm:tsou} and~\ref{exm:wath} (all three functions are discontinuous for Example~\ref{exm:tsou}). The red dot indicates $\psi_k$ at the initial guess $\hat{y}^k$ (see line 4 of Algorithm~\ref{alg:prototype_disc}) for problem~\eqref{eqn:method-1} when the discretization is set using Algorithm~\texttt{GREEDY} in Section~\ref{subsec:outline-opt-disc-alg}.
}
\label{fig:value-funcs_2}
% \vspace*{-0.1in}
\end{figure}

\paragraph*{An effective heuristic solution approach}

Due to the potential nonsmooth and discontinuous nature of the functions $\phi_k$ and $\psi_k$, we propose to solve the max-min problem of Algorithm \ref{alg:prototype_disc} using gradients (whenever they exist) of $\phi_k$ and $\psi_k$ within a bundle solver for nonsmooth nonconvex optimization~\cite{makela2003multiobjective}.
Each iteration of the bundle method requires function and generalized gradient evaluations.
We estimate the values of the functions $\phi_k$ and $\psi_k$ by solving the inner-minimization problems in problems~\eqref{eqn:method-2} and~\eqref{eqn:method-1} to \textit{local} optimality.
We then try and use Theorem~\ref{thm:parametric_sens} in Appendix~\ref{sec:sensitivity_theory} to compute gradients of the local minimum value function $\phi_k$ or $\psi_k$ (this involves the solution of a linear system of equations) when its assumptions hold. 
If some of the assumptions of Theorem~\ref{thm:parametric_sens} do not hold during the solution of problem~\eqref{eqn:method-2} or~\eqref{eqn:method-1}, we try and use the heuristics detailed below to estimate a generalized gradient of $\phi_k$ or $\psi_k$.
Note that we terminate the solution of the max-min problem and return its best found solution if the local solver fails to successfully solve the inner-minimization problem  at any step.

\paragraph*{Heuristics for estimating a generalized gradient of $\phi_k$ or $\psi_k$}

Whenever its assumptions hold, we use Theorem~\ref{thm:parametric_sens} to compute a gradient of the functions $\phi_k$ and $\psi_k$ in problems~\eqref{eqn:method-2} and~\eqref{eqn:method-1}. 
In practice, we find that {all of these assumptions} may not hold at \textit{every} iterate of the max-min solution. 

{When only the SC condition fails, the Lagrange multipliers are not unique. If SSOSC holds for all dual variables, then the value functions $\phi_k$ and $\psi_k$ are piecewise-differentiable, with a kink at the evaluation point due to changes in the active constraint set. In this case, a directional derivative can be computed by appropriately selecting a subset of weakly active constraints~\citep[Theorem~1]{ralph1995directional}. 
% {More generally, if SC and LICQ do not hold, but the Mangasarian-Fromovitz constraint qualification and generalized SSOSC holds then a directional derivative of $\phi_k$ or $\psi_k$ and the corresponding active set of constraints can be found by solving a quadratic program \citep{shapiro1988sensitivity,ralph1995directional}. 
Rather than explicitly checking whether only SC fails, we always handle weakly active constraints using the following heuristic procedure.
}
We exclude weakly active bound constraints by default since including them causes the derivative of the variable at the bound to be zero. Based on numerical experiments, we heuristically choose to include all other weakly active constraints.
An alternative to the above heuristic is to use the results of Stechlinski et al.\ \cite{stechlinski2019generalized} to rigorously compute generalized gradients when SC does not hold; we do not adopt this approach because of its higher computational cost.

{If either $J_z(p)$ or $H_{z,\lambda}$ in Theorem~\ref{thm:parametric_sens} is singular (e.g., if LICQ does not hold), we add a small {regularization} term before solving the corresponding linear system to determine sensitivities. If $H_{z,\lambda}$ is non-square, we follow the heuristic in Section~4.3 of Agrawal et al.\ \cite{agrawal2019differentiable} and compute the minimum-norm least squares solution to the linear system.
If an error occurs during the solution of the {max-min} problem (e.g., the local optimizer fails to converge), then we terminate the solution of the {max-min} problem and return the best found solution.}

\paragraph*{{Enhancements to the heuristic method for solving max-min problems}} 
We list ways in which Algorithms \texttt{OPT}, \texttt{GREEDY},  \texttt{2GREEDY}, and \texttt{HYBRID} can be enhanced in practice.
First, the solution of the sequence of inner-minimization problems in~\eqref{eqn:method-2} or~\eqref{eqn:method-1} can be effectively warm-started using active-set methods.
Second, solving these inner-minimization problems using multi-start techniques can increase the likelihood that we compute the global value functions $\phi_k$ and $\psi_k$.
Finally, as we saw in Examples~\ref{exm:seidel} to~\ref{exm:wath}, the assumptions of Theorem~\ref{thm:parametric_sens} may not hold throughout the domain of $\phi_k$ and $\psi_k$.
In such situations, we may either return the best found solution to problems~\eqref{eqn:method-2} and~\eqref{eqn:method-1}, or randomly perturb the current iterate in an attempt to avoid points of nondifferentiability.

\paragraph*{{Alternative approaches for solving the max-min problems~\eqref{eqn:method-2} and~\eqref{eqn:method-1}}}
If the inner-minimization problems in~\eqref{eqn:method-2} and~\eqref{eqn:method-1} are convex and satisfy a constraint qualification (e.g., Slater's condition), then we might be able to use strong duality to reformulate problems~\eqref{eqn:method-2} and~\eqref{eqn:method-1} into single-level maximization problems.
The resulting problems can then be solved to local optimality to update the discretization~$Y_d$.
Alternatively, suppose the functions in~\eqref{eqn:sip} are continuously differentiable and the inner-minimization problems in~\eqref{eqn:method-2} and~\eqref{eqn:method-1} satisfy a constraint qualification for each feasible point of the outer-maximization problems.
Then we can {relax} problems~\eqref{eqn:method-2} and~\eqref{eqn:method-1} into maximization problems with complementarity constraints using the KKT conditions for the inner-minimization problems, which can then be solved locally to update $Y_d$ (cf.\ \cite{stein2003solving}). 
We compare our heuristic approach for solving problems~\eqref{eqn:method-2} and~\eqref{eqn:method-1} with this KKT-based {relaxation} approach in Section~\ref{sec:numerical_results}.

When the assumptions of Theorem~\ref{thm:parametric_sens} fail to hold, we may also be able to either use generalized gradients~\cite{dempe2017bilevel,mordukhovich2009subgradients} of $\phi_k$ and $\psi_k$, or directional derivatives~\cite{ralph1995directional} of $\phi_k$ and $\psi_k$ and their generalizations~\cite{stechlinski2018generalized,stechlinski2019generalized} to solve problems~\eqref{eqn:method-2} and~\eqref{eqn:method-1} heuristically~\cite{burke2020gradient}.
Techniques for minimizing discontinuous functions (see, e.g., Ermoliev et al.\ \cite{ermoliev1995minimization}) may also be used to maximize $\phi_k$ and $\psi_k$ over their domains when none of the aforementioned approaches are applicable.

\section{Generalized bounding-focused discretizations}
\label{sec:generalized_discretization}

We propose new \textit{generalized} bounding-focused discretization methods for \eqref{eqn:sip} that can achieve faster rate of convergence of the sequence of lower bounds than \textit{any} traditional discretization method relying on~\eqref{eqn:disc-lbp}.

To motivate these generalized discretization methods, note that~\eqref{eqn:sip} can be equivalently reformulated as the bilevel problem~\cite{dempe2002foundations,stein2003bi}: 
\begin{align}
\label{eqn:bilevel_reform}
\min_{x \in X} \:\: & f(x) \\
\text{s.t.} \:\: & g(x,y^*(x)) \leq 0, \nonumber
\end{align}
where $y^* : X \to Y$ maps $x \in X$ to \textit{an} optimal solution to~\eqref{eqn:llp}.

The BF algorithm~\ref{alg:bfdisc} may be viewed as approxi\-mating $y^*(x)$ with solutions $\{y^*(x^k)\}_k$ of the lower-level problem~\eqref{eqn:llp} at incumbent solutions $\{x^k\}_k$ to~\eqref{eqn:disc-lbp}.
This zeroth-order approximation of $y^*$ may be crude.
Assuming $y^*$ is differentiable (cf.\ Theorem~\ref{thm:parametric_sens}), Seidel and K{\"u}fer~\cite{seidel2020adaptive} and Djelassi~\cite{djelassi2020discretization} instead propose to use the first-order approximation $y^*(x) \approx y^*(x^k) + J^*_y(x^k) (x - x^k)$ at incumbent solutions $\{x^k\}_k$, where $J^*_y(x^k)$ is the $d_y \times d_x$ Jacobian matrix with rows $\tr{(\nabla_x y^*_1(x^k))}, \dots, \tr{(\nabla_x y^*_{d_y}(x^k))}$.
In particular, Section~3.4 of Djelassi~\cite{djelassi2020discretization} suggests the following generalization of the lower bounding problem~\eqref{eqn:disc-lbp}:
\begin{align}
\label{eqn:higher-order-lbp}
\min_{x \in X} \:\: & f(x) \tag{G-LBP} \\
\text{s.t.} \:\: & g\bigl(x,\text{proj}_Y(A x + b)\bigr) \leq 0, \quad \forall (A,b) \in Y^G_d, \nonumber
\end{align}
where tuples $(A,b)$ in the generalized discretization $Y^G_d$ satisfy $A \in \R^{d_y \times d_x}$ and $b \in \R^{d_y}$.
Setting $Y^G_d = \{(0,y^{BF,1}), \dots, (0,y^{BF,k})\}$ recovers the BF lower bounding problem at iteration $k$. 
Djelassi~\cite{djelassi2020discretization} proposes to improve the BF lower bound using the generalized discretization $Y^G_d = \{(J^*_y(x^1), y^*(x^1) - J^*_y(x^1) x^1), \dots,$ $(J^*_y(x^k), y^*(x^k) - J^*_y(x^k) x^k)\}$ at iteration~$k$ given a sequence of candidate solutions $\{x^k\}_k \subset X$ to~\eqref{eqn:sip}.
The projection step ensures~\eqref{eqn:higher-order-lbp} is a relaxation of~\eqref{eqn:sip}, which implies solving~\eqref{eqn:higher-order-lbp} to \textit{global} optimality yields a lower bound on the optimal value $v^*$.
However, it also makes~\eqref{eqn:higher-order-lbp} nonsmooth and more challenging to solve than~\eqref{eqn:disc-lbp}.

We propose bounding-focused generalized discretizations of~\eqref{eqn:sip} in the form of~\eqref{eqn:higher-order-lbp} that can achieve faster rate of convergence of lower bounds than the discretization methods in Section~\ref{sec:accelerated_discretization}.
Our key idea is to populate the generalized discretization $Y^G_d$ with tuples $(A,b)$ that yield the highest lower bound.
In the first iteration, we solve the following max-min problem to determine a generalized discretization $Y^G_d = \{(\bar{A}^1,\bar{b}^1)\}$ that results in the highest lower bound:
\begin{align}
\label{eqn:genl_max_min_iter1}
(\bar{A}^1, \bar{b}^1) \in \argmax_{A^1 \in \R^{d_y \times d_x}, b^1 \in \R^{d_y}} \: & \min_{x \in X} \:\: f(x) \\
& \:\:\:\text{s.t.} \:\: g\bigl(x,\text{proj}_Y(A^1 x + b^1)\bigr) \leq 0. \nonumber
\end{align}
We again assume for simplicity that the maxima in all of our subproblems is attained.
We propose the following extensions of problems~\eqref{eqn:method-2} and~\eqref{eqn:method-1} for updating the bounding-focused generalized discretization at iteration $k > 1$. 

The first approach discards the generalized discretization $Y^{G,k-1}_d$ at {iteration}~$k-1$ and determines a fresh generalized discretization at iteration~$k$ by solving the max-min problem:
\begin{align}
\label{eqn:genl_method-2}
(\bar{A}^1,\bar{b}^2,\dots,\bar{A}^k,\bar{b}^k) \in \uset{b^1,\dots,b^k \in \R^{d_y}}{\argmax_{A^1,\dots,A^k \in \R^{d_y \times d_x}}} &\phi^G_k(A^1,b^1,\dots,A^k,b^k), \\
\phi^G_k(A^1,b^1,\dots,A^k,b^k) := \:\:  \min_{x \in X} \:\: &f(x) 
\nonumber \\
\:\:\:\text{s.t.} \:\: &g\bigl(x,\text{proj}_Y(A^i x + b^i)\bigr) \leq 0, \quad \forall i \in [k]. \nonumber
\end{align}
The resulting generalized discretization $Y^{G,k}_d := \{ (\bar{A}^1, \bar{b}^1), \dots, (\bar{A}^k, \bar{b}^k) \}$ at iteration $k$ yields the highest lower bound among all possible relaxations \eqref{eqn:higher-order-lbp} with {$\abs{Y^G_d} = k$}. 
To mitigate the computational cost of solving problem~\eqref{eqn:genl_method-2}, our second approach updates the generalized discretization $Y^{G,k-1}_d := \{ (\bar{A}^1, \bar{b}^1), \dots, (\bar{A}^{k-1}, \bar{b}^{k-1}) \}$ at iteration $k-1$ by adding a single tuple $(\bar{A}^k,\bar{b}^k)$ that maximizes lower bound improvement.
This can be formulated as:
\begin{align}
\label{eqn:genl_method-1}
(\bar{A}^k, \bar{b}^k) \in \argmax_{A^k \in \R^{d_y \times d_x}, b^k \in \R^{d_y}}  &\psi^G_k\bigl(A^k, b^k;Y^{G,k-1}_d\bigr), \\
\psi^G_k\bigl(A^k,b^k; Y^{G,k-1}_d\bigr) := \:\: \min_{x \in X} \:\: & f(x) \nonumber \\
\:\:\:\text{s.t.} \:\: & g(x,\text{proj}_Y(A x + b)) \leq 0, \quad \forall (A,b) \in Y^{G,k-1}_d, \nonumber \\*
\quad\quad\:\: & g(x,\text{proj}_Y(A^k x + b^k)) \leq 0. \nonumber
\end{align}
Problem~\eqref{eqn:genl_method-1} can be viewed as a greedy approximation of problem~\eqref{eqn:genl_method-2}.
We {propose} two variants of these bounding-focused generalized discretization methods in Section~\ref{subsec:genl_algorithm_outlines} and establish theoretical guarantees in  Section~\ref{subsec:genl-convergence-guarantees}.

\begin{algorithm}[t]
\caption{Prototype bounding-focused generalized discretization method}
\label{alg:prototype_genl_disc}
{
\begin{algorithmic}[1]
\State \textbf{Input}: feasibility tolerance $\varepsilon_{f} \geq 0$, minimum bound improvement $\delta \geq 0$, and initial generalized discretization $Y^G_d = \emptyset$.

\For{$k = 1, 2, \dots$}

\State Solve problem~\eqref{eqn:higher-order-lbp} globally to get solution $x^k$, lower bound $LBD^k$.

\State Solve problem~\eqref{eqn:llp} with $x = x^k$ globally to get solution $y^*(x^k)$.

\State If assumptions of Theorem~\ref{thm:parametric_sens} hold for~\eqref{eqn:llp} with $x = x^k$, compute 

\Statex \hspace*{0.2in} the Jacobian matrix $J^*_y(x^k)$.

\If{$G(x^k) \leq \varepsilon_{f}$} 

\State \textbf{Terminate} with $\varepsilon_f$-feasible solution $x^k$ to~\eqref{eqn:sip}.

\Else

\State Solve a max-min problem (heuristically) to get candidate

\Statex \hspace*{0.4in} generalized discretization tuples $\{(\bar{A}^{k,1},\bar{b}^{k,1}), \dots, (\bar{A}^{k,n_k},\bar{b}^{k,n_k}) \}$.

\State Solve the inner-min problem to \textit{global} optimality at the above

\Statex \hspace*{0.4in} candidate solution.
Let $\eta^*_k$ denote its global optimal value.

\vspace*{0.03in}
\If{$\eta^*_k \geq LBD^k + \delta$}

\State Update $Y^G_d$ using $\{(\bar{A}^{k,1},\bar{b}^{k,1}), \dots, (\bar{A}^{k,n_k},\bar{b}^{k,n_k}) \}$.

\Else

\State Set $Y^G_d \leftarrow Y^G_d \cup \{ (J^*_y(x^k), y^*(x^k) - J^*_y(x^k) x^k) \}$ if assumptions 

\Statex \hspace*{0.6in} of Theorem~\ref{thm:parametric_sens} hold, and $Y^G_d \leftarrow Y^G_d \cup \{ (0, y^*(x^k)) \}$ otherwise

\EndIf

\vspace*{0.03in}
\EndIf

\EndFor

\end{algorithmic}
}
\end{algorithm}

\subsection{Outline of bounding-focused generalized discretization methods}
\label{subsec:genl_algorithm_outlines}

Algorithm~\ref{alg:prototype_genl_disc} outlines a prototype bounding-focused generalized discretization method for~\eqref{eqn:sip} (cf.\ Algorithm~\ref{alg:prototype_disc}).
The key difference with the approach outlined in Djelassi~\cite{djelassi2020discretization} is on lines 9--12 of Algorithm~\ref{alg:prototype_genl_disc}.
If $x^k$ is not $\varepsilon_f$-feasible, Algorithm~\ref{alg:prototype_genl_disc} solves a max-min problem (heuristically) to identify new tuples that could be used to update the generalized discretization $Y^G_d$, whereas Djelassi~\cite{djelassi2020discretization} always looks to update $Y^G_d$ with the tuple $(J^*_y(x^k), y^*(x^k) - J^*_y(x^k) x^k)$ corresponding to a linear approximation of $y^*(x)$ at $x = x^k$. 

We consider four realizations of Algorithm~\ref{alg:prototype_genl_disc} that only vary on lines 9--12:
\texttt{G-OPT}, \texttt{G-GREEDY}, \texttt{G-2GREEDY}, and \texttt{G-HYBRID}.
Algorithms~\texttt{G-OPT}, \texttt{G-GREEDY}, and \texttt{G-HYBRID} are direct analogs of \texttt{OPT}, \texttt{GREEDY}, and \texttt{HYBRID} that rely on problems~\eqref{eqn:genl_method-2} and~\eqref{eqn:genl_method-1} instead of problems~\eqref{eqn:method-2} and~\eqref{eqn:method-1}.
Algorithm~\texttt{G-2GREEDY} first adds either $(J^*_y(x^k), y^*(x^k) - J^*_y(x^k) x^k)$ or $(0, y^*(x^k))$ to $Y^G_d$ (depending on whether the assumptions of Theorem~\ref{thm:parametric_sens} hold).
It then solves problem~\eqref{eqn:genl_method-1} to try and find another tuple to add to the generalized discretization.

\subsection{Convergence guarantees}
\label{subsec:genl-convergence-guarantees}

We begin by establishing convergence of the sequence of lower bounds $\{LBD^k\}_k$ generated by Algorithms \texttt{G-OPT}, \texttt{G-GREEDY}, \texttt{G-2GREEDY}, and \texttt{G-HYBRID} to $v^*$.
Like Theorem~\ref{thm:conv_lbd_disc}, this result also allows the max-min problems~\eqref{eqn:genl_method-2} and~\eqref{eqn:genl_method-1} to be solved using \textit{any} heuristic so long as $\delta > 0$ and the inner-minimization problem is solved to global optimality \textit{once} per iteration at the candidate max-min solution (see line 10 of Algorithm~\ref{alg:prototype_genl_disc}).
{Following the discussion in Section~\ref{subsec:outline-opt-disc-alg}, we note that the additional global solve on line~10 of Algorithm~\ref{alg:prototype_genl_disc} becomes redundant when the condition on line~11 is satisfied.
In practice, we skip this extra global solve and instead rely heuristically on the local solution of the inner-minimization problem to verify the sufficient bound increase condition on line~11.}

% \vspace*{-0.1in}
\begin{theorem}
\label{thm:conv_lbd_genl_disc}
Consider Algorithm~\ref{alg:prototype_genl_disc} with $\varepsilon_f = 0$ and $\delta > 0$.
Suppose the set $Y$ is convex and the generalized discretization $Y^G_d$ is updated using Algorithm \texttt{G-OPT}, \texttt{G-GREEDY}, \texttt{G-2GREEDY}, or \texttt{G-HYBRID}.
Then \mbox{$\underset{k \to \infty}{\lim} LBD^k = v^*$}.
\end{theorem}
\begin{proof}
% \vspace*{-0.3in}
The proof follows a similar outline as the proof of Theorem~\ref{thm:conv_lbd_disc} and Lemma 2.2 of Mitsos~\cite{mitsos2011global} (cf. Theorem 3.1 of Harwood et al.\ \cite{harwood2021note}).

{If the solution $x^l$ to the lower bounding problem~\eqref{eqn:higher-order-lbp} at iteration~$l \in \mathbb{N}$ is feasible to~\eqref{eqn:sip}, then we have $LBD^k = v^*$ for all $k \geq l$.
This holds because problem~\eqref{eqn:higher-order-lbp} is a relaxation of~\eqref{eqn:sip}, and the discretization methods \texttt{G-OPT}, \texttt{G-GREEDY}, \texttt{G-2GREEDY}, and \texttt{G-HYBRID} ensure that the lower bound $LBD^k$ is monotonically non-decreasing with respect to the iteration number $k$.
Hence, the stated result follows directly in this case.}

{Suppose that the lower bounding solution $x^k \in X$ is infeasible to~\eqref{eqn:sip} at each iteration $k \in \mathbb{N}$.}
Since $X$ is compact, we can assume (by moving to a subsequence) that $x^k \to x^* \in X$.
We show that $x^*$ is feasible to~\eqref{eqn:sip}, which implies $LBD^k \to v^*$.

By mirroring the arguments in Theorem~\ref{thm:conv_lbd_disc}, note that line 14 of Algorithm~\ref{alg:prototype_genl_disc} must be run infinitely often with $Y^G_d \leftarrow Y^G_d \cup \{ (J^*_y(x^k), y^*(x^k) - J^*_y(x^k) x^k) \}$ or $Y^G_d \leftarrow Y^G_d \cup \{ (0, y^*(x^k)) \}$.
Therefore, the asymptotic behavior of Algorithm~\ref{alg:prototype_genl_disc} is the same as the algorithm that adds at each iteration either $(J^*_y(x^k), y^*(x^k) - J^*_y(x^k) x^k)$ to $Y^G_d$ if Theorem~\ref{thm:parametric_sens} holds, or $(0, y^*(x^k))$ to $Y^G_d$ otherwise.
We show that $LBD^k \to v^*$ for the above algorithm.

Define the index sets $\mathcal{J}_k := \{j \in [k] : \text{Theorem~\ref{thm:parametric_sens} holds} \}$ and $\mathcal{L}_k := \{1,\dots,k\} \backslash \mathcal{J}_k$ corresponding to the above algorithm.
By construction, we have $\forall l, k$ such that $l > k$:
\[
g(x^l, \text{proj}_Y(\tilde{A}^k x^l + \tilde{b}^k)) \leq 0, \:\: \textbf{if } k \in \mathcal{J}_k, \quad \text{and} \quad g(x^l, y^*(x^k)) \leq 0, \:\: \textbf{if } k \in \mathcal{L}_k,
\]
where $\tilde{A}^k := J^*_y(x^k)$, $\tilde{b}^k := y^*(x^k) - J^*_y(x^k) x^k$ if $k \in \mathcal{J}_k$. 
Continuity of $g$, $\text{proj}_Y(\cdot)$ and compactness of $X$, $Y$ ensure uniform continuity, which implies that for any $\varepsilon > 0$, there exists $\kappa > 0$ such that for all $x \in X$ with $\norm{x - x^l} < \kappa$ and $\forall l, k$ with $l > k$:
\begin{align}
\label{eqn:int_step}
&g(x, \text{proj}_Y(\tilde{A}^k x + \tilde{b}^k)) < \varepsilon, \:\: \textbf{if } k \in \mathcal{J}_k, \quad \text{and} \quad g(x, y^*(x^k)) < \varepsilon, \:\: \textbf{if } k \in \mathcal{L}_k.
\end{align}
Since $x^k \to x^*$, we have $\norm{x^l - x^k} < \kappa$, $\forall l, k$ with $l > k \geq \bar{K}$.
Setting $x = x^k$ in \eqref{eqn:int_step} and noting $\tilde{A}^k x^k + \tilde{b}^k = y^*(x^k)$ if $k \in \mathcal{J}_k$ yields $0 < g(x^k, y^*(x^k)) < \varepsilon$, $\forall k \geq \bar{K}$. Therefore, $g(x^k, y^*(x^k)) = G(x^k) \to 0$ and continuity of $G$ implies $G(x^*) = 0$.
\end{proof}

Our next theorem establishes the rate of convergence of the lower bounds generated using Algorithm~\texttt{G-OPT} when the max-min problem~\eqref{eqn:genl_method-2} is solved to \textit{global} optimality at each iteration (this result is mainly of theoretical interest).

We require the following lemma, which is sharp for affine functions.

% \vspace*{-0.1in}
\begin{lemma}
\label{lem:pwlest}
Suppose $Z \subset \R^N$ is compact and convex and $F: Z \to \mathbb{R}^M$ is continuously differentiable with a Lipschitz continuous gradient.
Let $L_{\nabla F}$ denote the Lipschitz constant of $\nabla F$ on $Z$.
Then $\forall \varepsilon > 0$, there exist $J = \bigg\lceil\left(1 + \textup{diam}(Z)\sqrt{\frac{L_{\nabla F}}{2\varepsilon}}\right)^N\bigg\rceil$ affine functions $\{ \tr{(\alpha^j)} z + \beta^j \}_{j=1}^{J}$ with \mbox{$\uset{z \in Z}{\sup} \: \uset{j \in [J]}{\min} \: \big\lVert F(z) - \bigl(\tr{(\alpha^j)} z + \beta^j \bigr) \big\rVert \leq \varepsilon$}.
\end{lemma}

% \vspace*{-0.35in}
\begin{theorem}
\label{thm:convrate_gopt}
Consider Algorithm~\texttt{G-OPT} with $\varepsilon_f > 0$ and $\delta = 0$.
Suppose $X$ and $Y$ are convex sets and $\{g(x,\cdot)\}_{x \in X}$ is uniformly Lipschitz continuous on $Y$ with Lipschitz constant $L_{g,y} > 0$.
Additionally, suppose $y^*$ is continuously {differentiable} on $X$ with a Lipschitz continuous gradient. Let $L_{\nabla y}$ denote the {Lipschitz} {constant} of $\nabla y^*$ on $X$.
If the max-min problem~\eqref{eqn:genl_method-2} is solved to \textit{global} {optimality} at each iteration, then Algorithm~\texttt{G-OPT} terminates with an $\varepsilon_f$-feasible point in at most $\Big\lceil\left(1 + \textup{diam}(X)\sqrt{\frac{L_{\nabla y} L_{g,y}}{2\varepsilon_f}}\right)^{d_x}\Big\rceil$ iterations.
\end{theorem}

Chapter~3 of Fiacco~\cite{fiacco1983} presents conditions when the assumption on $y^*$ holds.
The convexity assumption on the set $X$ may be relaxed by considering any convex superset of $X$.
Note that the bound on the number of iterations in Theorem~\ref{thm:convrate_gopt} scales like $\varepsilon_f^{-0.5d_x}$ compared to the $\varepsilon_f^{-d_x}$ scaling in Proposition~\ref{prop:convrate_bf}.
The rate at which $LBD^k \to v^*$ for \texttt{G-OPT} can be derived similar to Theorem~\ref{thm:calmness}.
Since we may set $\bar{A}^k = 0$, $\forall k$, Algorithm~\texttt{G-OPT} generates tighter lower bounds than Algorithm~\texttt{OPT} if we solve the max-min problem~\eqref{eqn:genl_method-2} to global optimality.

\subsection{Solving the max-min problems}
\label{subsec:gen-solv-max-min}

In addition to the challenges outlined in Section~\ref{sec:accelerated_discretization}, solving problems~\eqref{eqn:genl_method-2} and~\eqref{eqn:genl_method-1} globally may also be challenging due to the nonsmooth projection operator.
Therefore, we design effective heuristic methods for solving these problems and empirically show in Section~\ref{sec:numerical_results} that they often yield good generalized discretizations with fast convergence of lower bounds.

\paragraph*{Properties of the optimal solution mapping $y^*$}
Before we outline our heuristic approach for solving problems~\eqref{eqn:genl_method-2} and~\eqref{eqn:genl_method-1}, we provide empirical motivation for (bounding-focused) generalized discretization methods.
Figure~\ref{fig:optimal_soln} plots the optimal solution mapping $y^*$ for Examples~\ref{exm:dp},~\ref{exm:seidel},~\ref{exm:tsou}, and~\ref{exm:wath}.
\begin{figure}[t]
\centering
\begin{subfigure}{0.245\textwidth}
\includegraphics[width=0.98\columnwidth]{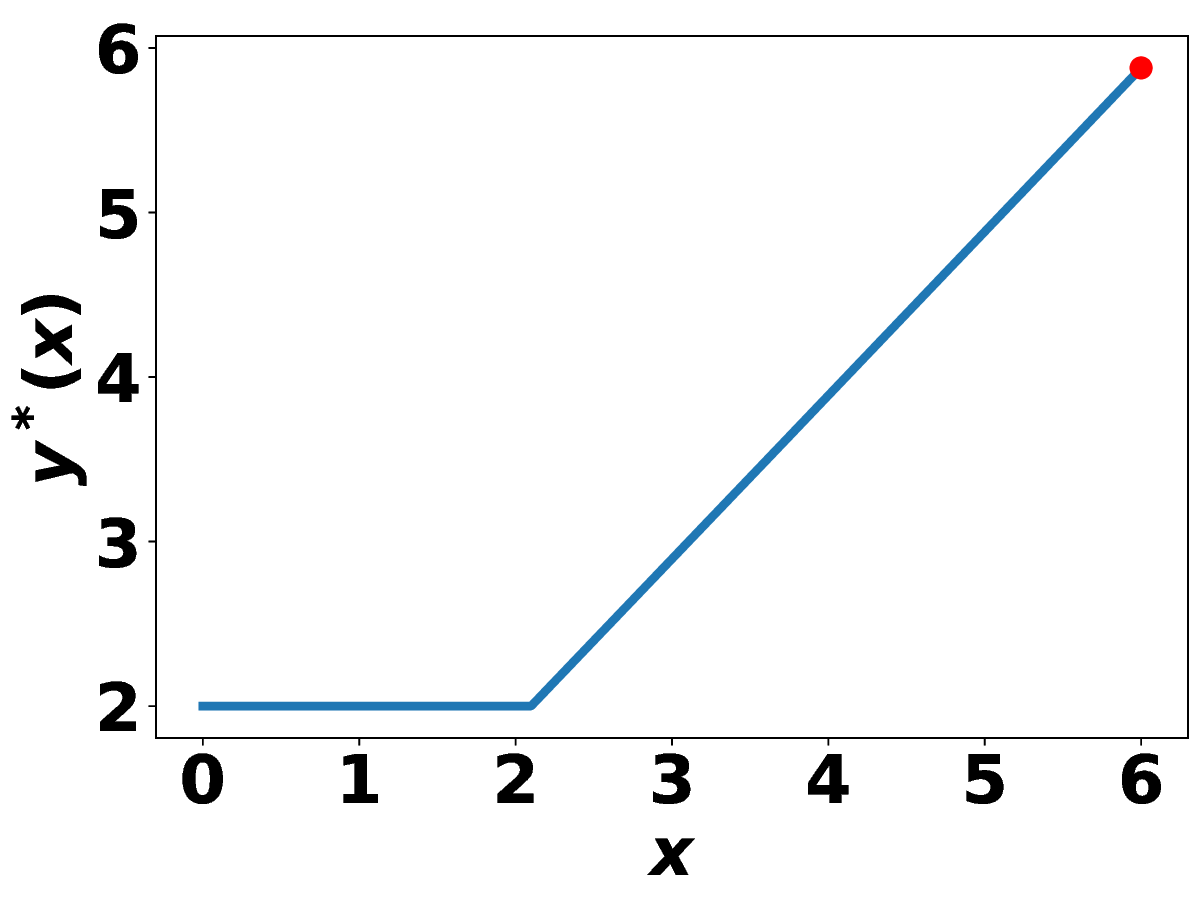}
    \caption{Example~\ref{exm:dp}}
\end{subfigure}%
\begin{subfigure}{0.245\textwidth}
\includegraphics[width=0.98\columnwidth]{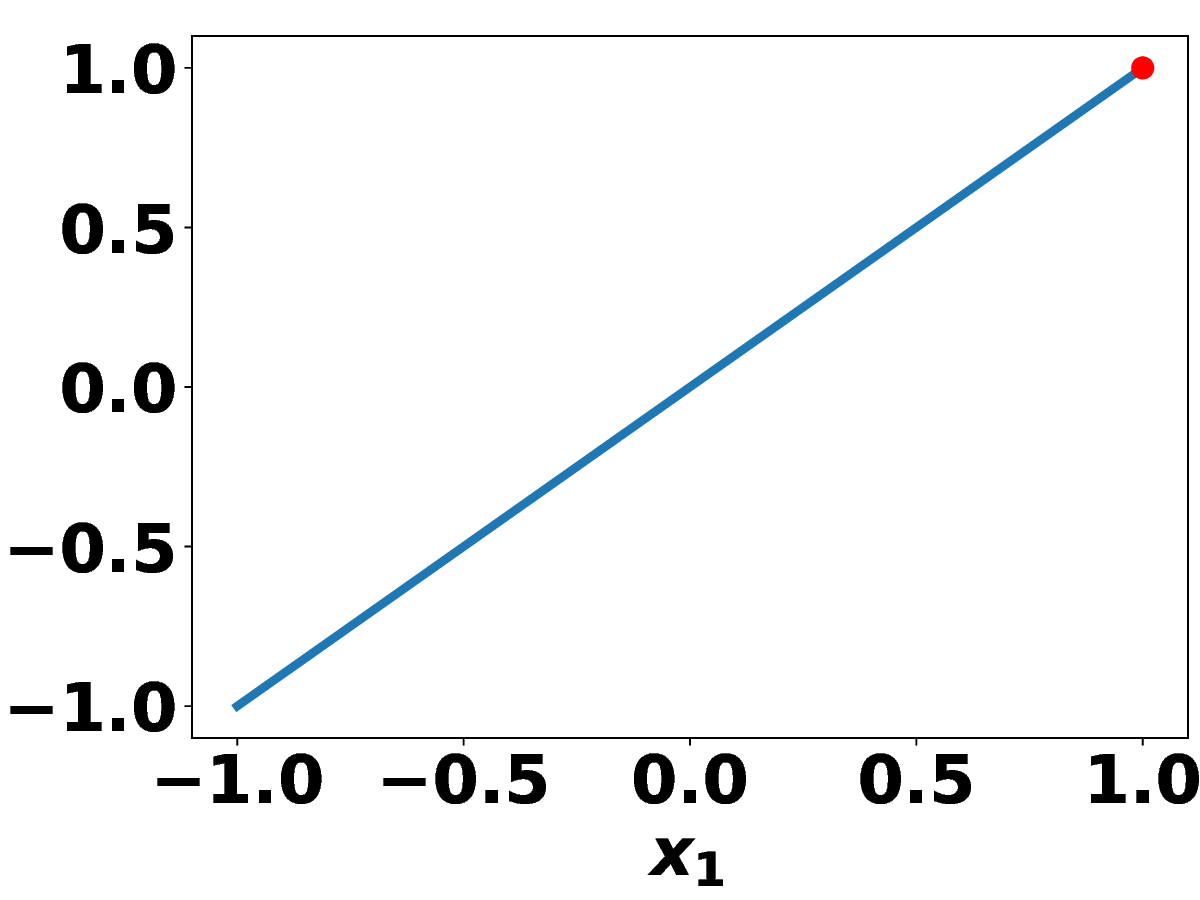}
\caption{Example~\ref{exm:seidel}}
\end{subfigure}%
\begin{subfigure}{0.245\textwidth}
\includegraphics[width=0.98\columnwidth]{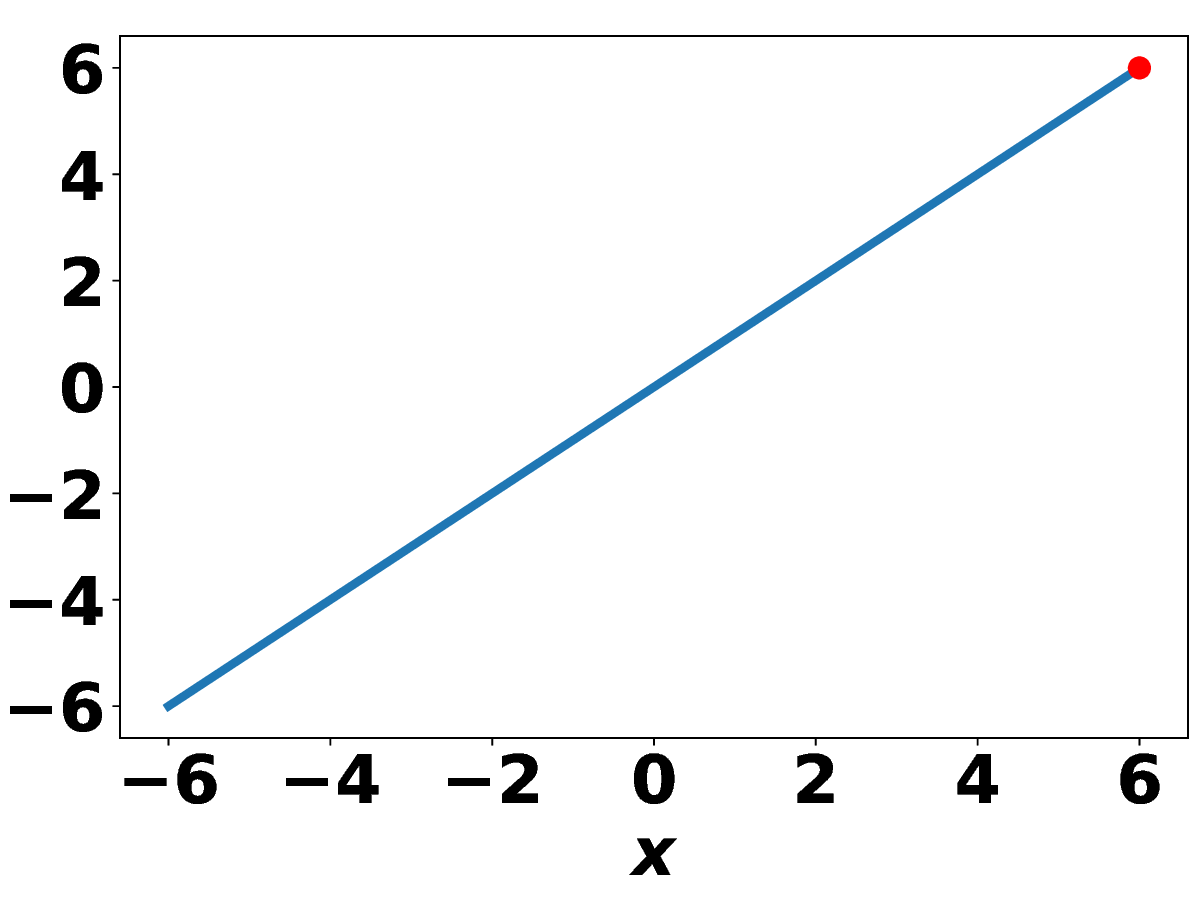}
    \caption{Example~\ref{exm:tsou}}
\end{subfigure}%
\begin{subfigure}{0.245\textwidth}
\includegraphics[width=0.98\columnwidth]{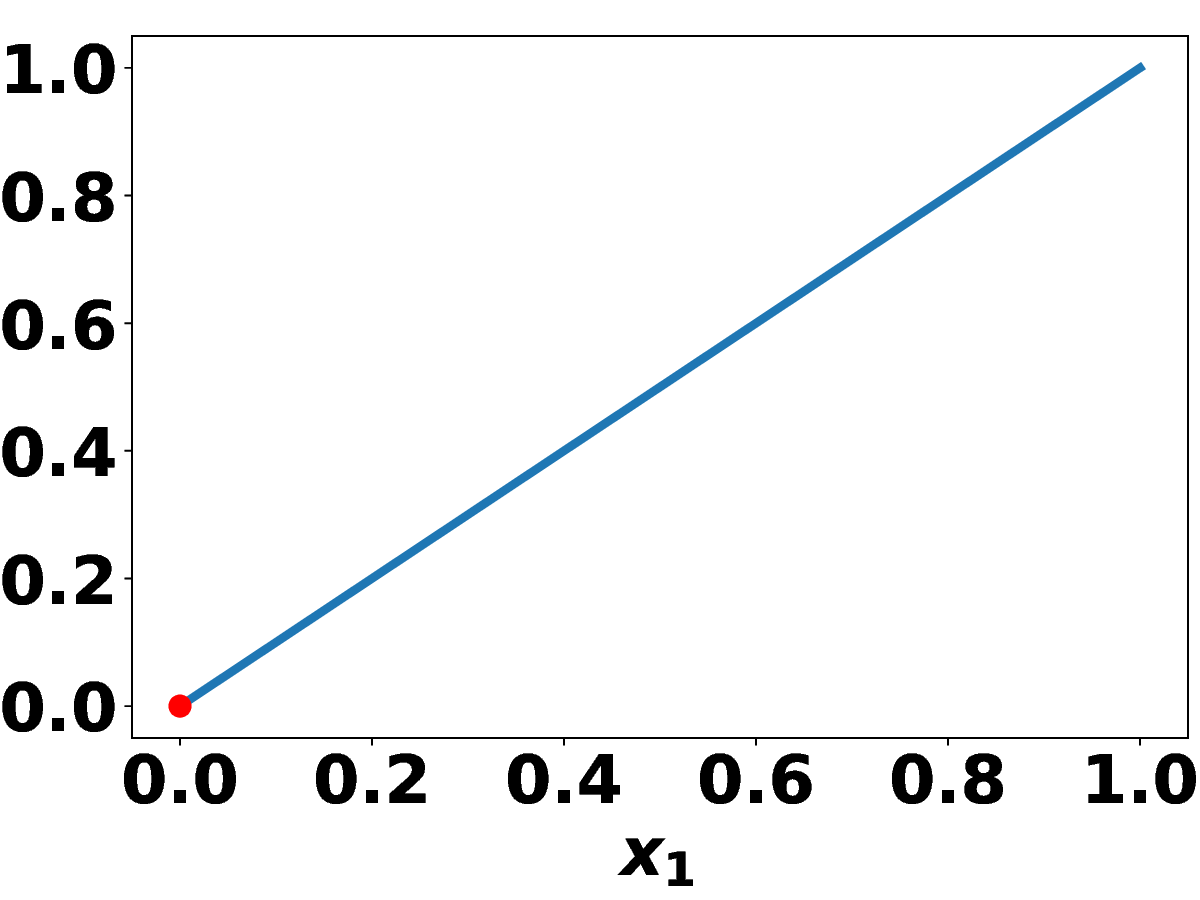}
    \caption{Example~\ref{exm:wath}}
\end{subfigure}
\caption{Optimal solution mapping~$y^*$ (note that it only depends on $x_1$ for Examples~\ref{exm:seidel} and~\ref{exm:wath}). The red dot indicates $y^*(x^1)$ at $x^1 \in \argmin_{x \in X} f(x)$.}
\label{fig:optimal_soln}
% \vspace*{-0.1in}
\end{figure}
Interestingly, this mapping is well-behaved for all four examples (it is piecewise-linear for Example~\ref{exm:dp} and linear for Examples~\ref{exm:seidel} to~\ref{exm:wath}).
Moreover, using $Y^G_d = \{(J^*_y(x^1), y^*(x^1) - J^*_y(x^1) x^1)\}$ in~\eqref{eqn:higher-order-lbp} at the point $(x^1, y^*(x^1))$ highlighted in these plots yields a lower bound equal to $v^*$ for all four examples.
However, this favorable situation may not always be the case, and the mapping $y^*$ may be nonconvex, nonsmooth, and even discontinuous in general.
For example, any~\eqref{eqn:sip} with $X = Y = [0,1]$ and $g(x,y) = (x - 0.5)y$ results in the optimal solution mapping $y^*(x) = \mathds{1}(x - 0.5)$, where $\mathds{1}(z) = 1$ if $z \geq 0$ and zero otherwise, which is discontinuous at $x = 0.5$.
Example~\ref{exm:hijazi} from Section~\ref{subsec:convergence_guarantees} provides another instance where $y^*$ is discontinuous.

\vspace*{0.1in}
\noindent Example~\ref{exm:hijazi}: 
Pick any $\bar{y} \in Y$, and
consider the optimal solution mapping $y^*$:
\[
y^*(x) := 
\begin{cases}
\frac{x}{\norm{x}} \sqrt{d_x - 1}, & \text{if } x \neq 0\\
\bar{y}, & \text{if } x = 0
\end{cases}.
\]
This mapping is discontinuous at $x = 0$ irrespective of the choice of {$\bar{y} \in Y$}.
However, setting $Y^G_d = \{(I,0)\}$ in~\eqref{eqn:higher-order-lbp}, where $I$ is the identity matrix, yields a lower bound equal to $v^*$ since $\text{proj}_{Y}(x) = y^*(x)$, $\forall x \in X$.
Therefore, a generalized discretization with $\abs{Y^G_d} = 1$ is sufficient for convergence, which is in stark contrast with the discretization methods in Section~\ref{subsec:convergence_guarantees} that require exponentially many iterations in the dimension $d_x$ to converge.

\paragraph*{A heuristic solution approach}

We only consider the setting where $Y = [y^L, y^U]$ for some $y^L, y^U \in \R^{d_y}$. 
The function $\text{proj}_Y(\cdot)$ can then be reformulated as the MILP-representable function $\text{mid}(y^L, \cdot, y^U)$ using the ``lambda formulation'' {(see~\citep[Formulation~(2.5)]{vielma2015mixed})}, and~\eqref{eqn:higher-order-lbp} can be written as a mixed-integer nonlinear program~(MINLP).
Section~3.4 of Djelassi~\cite{djelassi2020discretization} also considers more general settings for~$Y$.

Similar to the functions $\phi_k$ and $\psi_k$ in Section~\ref{sec:accelerated_discretization}, the functions $\phi_k^G$ and $\psi_k^G$ are potentially nonsmooth and discontinuous.
Therefore, we propose to solve the max-min problems in Algorithm \ref{alg:prototype_genl_disc} approximately by using gradients (whenever they exist) of a smooth approximation of $\phi_k^G$ and $\psi_k^G$ within a bundle solver for nonsmooth, nonconvex optimization \cite{makela2003multiobjective}.

We begin by replacing the function $\text{mid}\{y^L, \cdot, y^U\}$ in problems~\eqref{eqn:genl_method-2} and~\eqref{eqn:genl_method-1} with the approximation $-t^{-1}\log((\exp(ty^L)+\exp(t \: \cdot))^{-1}+\exp(-ty^U))$, where the smoothing parameter $t$ is set to $100$.
We estimate values of $\phi_k^G$ and $\psi_k^G$ by solving these smooth approximations of the inner-minimization problems to \emph{local} optimality. 
We then attempt to apply Theorem~\ref{thm:parametric_sens} to compute gradients of the smooth approximation of the local minimum value function of $\phi_k^G$ and $\psi_k^G$. 
Whenever the assumptions of Theorem~\ref{thm:parametric_sens} fail to hold, we estimate a generalized gradient of $\phi_k^G$ and $\psi_k^G$ using similar heuristics as in Section~\ref{subsec:solving_max_min} (however, based on numerical experiments, we exclude all weakly active constraints when SC does not hold).
If the bundle solver terminates after a single iteration, we restart its solution using a random initialization for $(A^k, b^k)$.

\section{Generalizations}
\label{sec:generalizations}

\paragraph*{Multiple semi-infinite constraints}
Suppose~\eqref{eqn:sip} includes $\abs{\mathcal{I}}$ semi-infinite constraints $g_i(x,y) \leq 0$, $\forall y \in Y^i$, {$i \in \mathcal{I}$}.
The formulation below extends the max-min problem~\eqref{eqn:max_min_iter1} for constructing a bounding-focused discretization at the first iteration of Algorithm~\ref{alg:prototype_disc}.
\begin{align*}
(\bar{y}^1, \dots, \bar{y}^{\abs{\mathcal{I}}}) \in \argmax_{(y^1, \dots,y^{\abs{\mathcal{I}}}) \in Y^1 \times \dots \times Y^{\abs{\mathcal{I}}}} \:\: & \min_{x \in X} \:\: f(x) \\
& \:\:\:\text{s.t.} \:\: g_i(x,y^i) \leq 0, \quad \forall i \in \mathcal{I}. \nonumber
\end{align*}
Extensions of the max-min problems~\eqref{eqn:method-2} and~\eqref{eqn:method-1} and the bounding-focused generalized discretization methods in Section~\ref{sec:generalized_discretization} readily follow.

\paragraph*{Generalized discretizations based on nonlinear approximations of $y^*$}

Instead of restricting ourselves to bounding-focused linear approximations of $y^*(x)$ as in Section~\ref{sec:generalized_discretization}, we can construct bounding-focused \textit{nonlinear} approximations of $y^*$ for potentially faster convergence.
Let $\gamma : X \times \Theta \to \R^{d_y}$ be any family of functions.
We propose the following extension of~\eqref{eqn:higher-order-lbp}:
\begin{align}
\label{eqn:genl_higher-order-lbp}
\min_{x \in X} \:\: & f(x) \\
\text{s.t.} \:\: & g\bigl(x,\text{proj}_Y(\gamma(x,\theta))\bigr) \leq 0, \quad \forall \theta \in \Theta_d. \nonumber
\end{align}
Clearly,~\eqref{eqn:higher-order-lbp} is a special case of problem~\eqref{eqn:genl_higher-order-lbp} where $\{\gamma(\cdot,\theta)\}_{\theta}$ is the set of all parametric affine (in $x$) functions.
Extensions of the max-min problems~\eqref{eqn:genl_method-2} and~\eqref{eqn:genl_method-1} to determine an optimal sequence of parameters $\{\theta^k\}_k$ readily follow.

\paragraph*{Mixed-integer SIPs}

The presence of integer variables in~\eqref{eqn:sip} precludes the use of the sensitivity theory in Appendix~\ref{sec:sensitivity_theory} for solving the max-min problems~\eqref{eqn:method-2},~\eqref{eqn:method-1},~\eqref{eqn:genl_method-2}, and~\eqref{eqn:genl_method-1}.
Because we can heuristically  solve these max-min problems \textit{without} sacrificing convergence of our (generalized) discretization methods, one heuristic is to use sensitivities of the value functions of the inner-minimization problems with the integer variables fixed to an optimal solution.
An alternative is to use smoothing-based approaches~\cite{ermoliev1995minimization} for approximating sensitivity information.

\section{Numerical results}
\label{sec:numerical_results}

We compare Algorithms \texttt{GREEDY}, \texttt{2GREEDY}, \texttt{HYBRID}, and \texttt{OPT} in Section~\ref{subsec:outline-opt-disc-alg} and Algorithms~\texttt{G-GREEDY}, \texttt{G-2GREEDY}, \texttt{G-HYBRID}, and \texttt{G-OPT}
in Section~\ref{subsec:genl_algorithm_outlines} with the BF
algorithm on \textit{standard} nonconvex SIP instances from the literature~\cite{watson1983numerical,seidel2020adaptive,tsoukalas2011feasible,mitsos2009test}. 
{These include a larger, parameterized version of Problem~5 from Watson~\cite{watson1983numerical}, which we label ``Watson 5\_10'' (cf.\ \citep{tanaka1988globally})}.
We only consider instances with scalar semi-infinite constraints and sets $Y$ of the form $Y = [y^L, y^U]$.
We omit instances with trigonometric functions as they are not supported by BARON~\cite{baron}. 
Although these instances are small-scale nonconvex SIPs ($d_x \leq 10$ and $d_y \leq 2$), solving them to \textit{global} optimality is not trivial.

We also test our algorithms on larger-scale {convex} SIPs ($d_x \in [21, 105]$ and {$d_y \in [5, 13]$}) from Cerulli et al.\ \cite{cerulli2022convergent}, {where~\eqref{eqn:llp} is a (potentially nonconvex) quadratic program. Our computational times are lower than those reported in Cerulli et al.\ \cite{cerulli2022convergent}, which we attribute to three differences in our implementation: $(i)$ we use BARON instead of Gurobi, $(ii)$ we use a different termination criterion (detailed in the next section), and $(iii)$ we enforce the symmetry constraint in the formulation symbolically using JuMP.
}

Because our focus is on designing discretization methods with tight lower bounds, we only compare our approaches with the BF algorithm (the state-of-the-art discretization method for nonconvex SIPs) and do not consider algorithms for constructing feasible points (see, e.g., \cite{mitsos2011global,tsoukalas2011feasible,djelassi2017hybrid}).

\subsection{Implementational details}
Our codes are compiled using {Julia 1.7.3, JuMP 1.24.0 \cite{JuMP2017}, BARON 24.1.3~\cite{baron} 
% (using CPLEX 22.1.0) 
or  Gurobi 12.0.1 as the global solver, Knitro 14.2.0 as the local NLP solver, Gurobi 12.0.1} as the LP solver, and the bundle solver MPBNGC 2.0 \cite{makela2003multiobjective} for the max-min problems. 
Our codes will be made available at \url{https://github.com/Process-Optimization-and-Control/Bounding-Focused-SIP-Discretizations}.
{All codes were run on a computer with an AMD Ryzen 7 PRO 8840U CPU (3.30 GHz, 8 cores) and 32 GB of RAM.}

\paragraph*{General algorithmic parameters}
Since~\eqref{eqn:disc-lbp} and~\eqref{eqn:higher-order-lbp} may not yield a feasible point finitely, we terminate our algorithms when the lower bound converges to within an absolute or relative tolerance of $10^{-3}$ of $v^*$.
{When available, $v^*$ is specified as the analytical solution or the best known feasible solution reported in the literature~\cite{djelassi2020discretization,cerulli2022convergent,seidel2020adaptive,tsoukalas2011feasible}; otherwise, it is computed offline by running the BF algorithm with $\varepsilon_f = 10^{-8}$ until an approximate feasible point is found.
This termination criterion is chosen to directly assess the impact of our discretization methods on accelerating the convergence of the lower bound.}

We use $\varepsilon_f = \delta = 10^{-8}$ in Algorithms~\ref{alg:bfdisc},~\ref{alg:prototype_disc}, and~\ref{alg:prototype_genl_disc}.
The feasibility and optimality tolerances in Gurobi are set to $10^{-8}$, with a {maximum time limit of five minutes}.
All of BARON's parameters are set to default, {except for $\texttt{MaxTime} = 300$ seconds}. 
Knitro is used with the following parameters: $\texttt{algorithm}=0$,  $\texttt{ftol}=10^{-8}$, $\texttt{maxtime}=10$, and $\texttt{feastol}=10^{-8}$. {We enable Knitro's multistart heuristic with a maximum of five starts. Due to a limitation in Knitro's Julia interface, the multistart procedure is executed serially rather than in parallel, which inflates the reported solve times for the max-min problems.}

\paragraph*{Algorithm-specific parameters}

The starting point for the max-min problem solved by the \texttt{2GREEDY} method is specified as $y^L+\zeta(y^U-y^L)$, where $\zeta$ is a diagonal matrix with random diagonal elements $\zeta_{ii} \sim U([0,1])$.
Similarly, for the \texttt{G-2GREEDY} method, the elements of $A^k$ are initialized randomly from $U([0,1])$ and $b^k$ is set to {$y^L+\zeta(y^U-y^L) - A^kx^k$} with $\zeta_{ii} \sim U([0,1])$. 
We set the parameter $K = 3$ for Algorithms \texttt{HYBRID} and~\texttt{G-HYBRID}.

\subsection{Bounding-focused discretization methods}
\label{sec: disc-methods}

\begin{table}[t]
\centering
\begin{tabular}{rccccccc}
\hline
\textbf{Instance}     & $\mathbf{d_x}$ & $\mathbf{d_y}$ & BF            & \texttt{GREEDY}        & \texttt{2GREEDY}       & \texttt{HYBRID}    & \texttt{OPT}   \\
            &    &    &  \multicolumn{5}{c}{\textbf{number of iterations for convergence}} \\ \hline
Watson 2                     & 2   & 1  & \textbf{2}  & \textbf{2}  & \textbf{2}  & \textbf{2}  & \textbf{2} \\
Watson 5                     & 3   & 1  & 5           & 4           & \textbf{2}  & 3           & 3          \\
Watson 5\_10           & 10  & 1  & 3           & 4           & \textbf{2}  & 3           & 3          \\
Watson 6                     & 2   & 1  & 3           & \textbf{2}  & \textbf{2}  & \textbf{2}  & \textbf{2} \\
Watson 7                     & 3   & 2  & \textbf{2}  & \textbf{2}  & \textbf{2}  & \textbf{2}  & \textbf{2} \\
Watson 8                     & 6   & 2  & 15          & 27          & \textbf{6}  & 11          & 21         \\
Watson 9                     & 6   & 2  & 9           & 8           & \textbf{5}  & 10          & 12         \\
Watson h                     & 2   & 1  & 18          & 21          & \textbf{13} & 23          & 25         \\
Watson n                     & 2   & 1  & 3       & \textbf{2}  & \textbf{2}  & \textbf{2}  & \textbf{2} \\
Seidel \& Küfer 2.1          & 2   & 1  & 8           & \textbf{3}  & \textbf{3}  & \textbf{3}  & \textbf{3} \\
Tsoukalas \& Rustem 2.1      & 1   & 1  & 8           & \textbf{4}  & 5           & 5           & 8          \\
Mitsos 4\_3                  & 3   & 1  & 5           & {4}  & \textbf{3}  & {4}  & {4} \\
Mitsos 4\_6                  & 6   & 1  & 7           & 7           & \textbf{6}  & 8           & 7          \\
Mitsos DP                    & 1   & 1  & 28          & \textbf{2}  & \textbf{2}  & \textbf{2}  & \textbf{2} \\
Cerulli et al.\ PSD 1        & 21  & 5  & \textbf{2}  & \textbf{2}  & \textbf{2}  & \textbf{2}  & \textbf{2} \\
Cerulli et al.\ PSD 2        & 21  & 5  & \textbf{2}  & \textbf{2}  & \textbf{2}  & \textbf{2}  & \textbf{2} \\
Cerulli et al.\ PSD 3        & 21  & 5  & \textbf{2}  & \textbf{2}  & \textbf{2}  & \textbf{2}  & \textbf{2} \\
Cerulli et al.\ PSD 4        & 21  & 5  & \textbf{2}  & \textbf{2}  & \textbf{2}  & \textbf{2}  & \textbf{2} \\
Cerulli et al.\ PSD 5        & 66  & 10 & 5           & \textbf{2}  & 5           & 5           & 5          \\
Cerulli et al.\ PSD 6        & 66  & 10 & 6           & \textbf{2}  & 6           & \textbf{2}  & \textbf{2} \\
Cerulli et al.\ PSD 7        & 105 & 13 & 7           & \textbf{3}  & 7           & 4           & 4          \\
Cerulli et al.\ PSD 8        & 105 & 13 & 5           & \textbf{2}  & 5           & \textbf{2}  & \textbf{2} \\
\hline
\end{tabular}
\caption{
{
Comparison of BF, \texttt{GREEDY}, \texttt{2GREEDY}, \texttt{HYBRID}, and \texttt{OPT} algorithms.
Bold entries correspond to the minimum number of iterations for each instance.
} 
}
\label{tab:resul-simple-probs}
\end{table}

\begin{sidewaystable}[]
\footnotesize
\begin{tabular}{rccccccc}
\toprule
\textbf{Instance} & $\mathbf{d_x}$ & $\mathbf{d_y}$ & BF & \texttt{GREEDY} & \texttt{2GREEDY} & \texttt{HYBRID} & \texttt{OPT} \\[0.02in]
\multicolumn{3}{c}{} & \multicolumn{5}{c}{\textbf{Total time to solve LBP/LLP/max-min (seconds)}} \\
\midrule
Watson 2 & 2 & 1 & 0.07 / 0.06 / 0.00 & 0.06 / 0.05 / 0.27 & 0.06 / 0.06 / 0.01 & 0.07 / 0.04 / 0.29 & 0.09 / 0.04 / 0.25 \\
Watson 5 & 3 & 1 & 0.16 / 0.43 / 0.00 & 0.08 / 0.28 / 0.36 & 0.05 / 0.18 / 0.07 & 0.05 / 0.07 / 0.36 & 0.04 / 0.06 / 0.35 \\
Watson 5\_10 & 10 & 1 & 0.12 / 0.08 / 0.00 & 0.17 / 0.13 / 0.58 & 0.04 / 0.07 / 0.07 & 0.08 / 0.08 / 0.64 & 0.07 / 0.08 / 0.58 \\
Watson 6 & 2 & 1 & 0.15 / 0.05 / 0.00 & 0.09 / 0.03 / 0.02 & 0.08 / 0.03 / 0.01 & 0.07 / 0.04 / 0.02 & 0.06 / 0.02 / 0.03 \\
Watson 7 & 3 & 2 & 0.06 / 0.03 / 0.00 & 0.08 / 0.02 / 0.01 & 0.04 / 0.02 / 0.01 & 0.05 / 0.02 / 0.00 & 0.05 / 0.02 / 0.00 \\
Watson 8 & 6 & 2 & 0.00 / 1.29 / 0.00 & 0.01 / 2.73 / 4.31 & 0.00 / 0.37 / 0.68 & 0.00 / 0.93 / 1.25 & 0.00 / 1.89 / 5.13 \\
Watson 9 & 6 & 2 & 0.00 / 0.29 / 0.00 & 0.00 / 0.31 / 0.37 & 0.00 / 0.22 / 0.29 & 0.00 / 0.46 / 1.02 & 0.00 / 0.56 / 2.26 \\
Watson h & 2 & 1 & 1.83 / 0.65 / 0.00 & 1.67 / 0.84 / 2.52 & 1.13 / 0.46 / 1.43 & 1.82 / 0.82 / 1.58 & 2.11 / 0.99 / 3.13 \\
Watson n & 2 & 1 & 0.05 / 0.04 / 0.00 & 0.03 / 0.03 / 0.01 & 0.05 / 0.06 / 0.02 & 0.03 / 0.03 / 0.01 & 0.02 / 0.02 / 0.02 \\
Seidel \& Küfer 2.1 & 2 & 1 & 0.24 / 0.25 / 0.00 & 0.10 / 0.04 / 0.19 & 0.09 / 0.09 / 0.07 & 0.05 / 0.05 / 0.22 & 0.10 / 0.04 / 0.23 \\
Tsoukalas \& Rustem 2.1 & 1 & 1 & 0.47 / 0.29 / 0.00 & 0.21 / 0.13 / 0.37 & 0.31 / 0.17 / 0.74 & 0.23 / 0.19 / 0.57 & 0.46 / 0.31 / 1.66 \\
Mitsos 4\_3 & 3 & 1 & 0.00 / 0.21 / 0.00 & 0.00 / 0.14 / 0.08 & 0.00 / 0.11 / 0.02 & 0.00 / 0.12 / 0.20 & 0.00 / 0.11 / 0.19 \\
Mitsos 4\_6 & 6 & 1 & 0.00 / 0.74 / 0.00 & 0.00 / 0.60 / 0.20 & 0.00 / 0.49 / 0.39 & 0.00 / 0.78 / 0.68 & 0.00 / 0.57 / 0.73 \\
Mitsos DP & 1 & 1 & 1.51 / 1.30 / 0.00 & 0.06 / 0.08 / 0.03 & 0.04 / 0.08 / 0.03 & 0.06 / 0.06 / 0.03 & 0.04 / 0.06 / 0.02 \\
Cerulli et al.\ PSD 1 & 21 & 5 & 0.29 / 0.09 / 0.00 & 0.21 / 0.07 / 2.11 & 0.28 / 0.09 / 2.12 & 0.25 / 0.06 / 2.23 & 0.15 / 0.04 / 2.24 \\
Cerulli et al.\ PSD 2 & 21 & 5 & 0.24 / 0.09 / 0.00 & 0.18 / 0.08 / 2.11 & 0.26 / 0.09 / 1.81 & 0.24 / 0.06 / 2.24 & 0.15 / 0.08 / 2.26 \\
Cerulli et al.\ PSD 3 & 21 & 5 & 0.25 / 0.09 / 0.00 & 0.27 / 0.09 / 3.35 & 0.17 / 0.06 / 1.88 & 0.26 / 0.03 / 3.57 & 0.19 / 0.04 / 3.58 \\
Cerulli et al.\ PSD 4 & 21 & 5 & 0.28 / 0.04 / 0.00 & 0.21 / 0.04 / 2.00 & 0.28 / 0.04 / 2.27 & 0.26 / 0.03 / 2.12 & 0.15 / 0.03 / 2.11 \\
Cerulli et al.\ PSD 5 & 66 & 10 & 1.53 / 0.13 / 0.00 & 0.61 / 0.04 / 31.50 & 1.69 / 0.11 / 805.54 & 1.73 / 0.11 / 570.17 & 1.72 / 0.09 / 571.48 \\
Cerulli et al.\ PSD 6 & 66 & 10 & 1.75 / 0.14 / 0.00 & 0.68 / 0.06 / 72.45 & 2.03 / 0.13 / 559.22 & 0.58 / 0.04 / 73.35 & 0.61 / 0.03 / 74.03 \\
Cerulli et al.\ PSD 7 & 105 & 13 & 5.55 / 0.17 / 0.00 & 2.44 / 0.07 / 313.27 & 5.76 / 0.16 / 1171.40 & 3.16 / 0.08 / 295.84 & 4.04 / 0.12 / 296.55 \\
Cerulli et al.\ PSD 8 & 105 & 13 & 4.15 / 317.20 / 0.00 & 1.66 / 87.14 / 65.57 & 4.13 / 323.67 / 342.43 & 1.69 / 87.38 / 65.61 & 1.62 / 87.67 / 65.52 \\
\bottomrule
\end{tabular}
\caption{
{Comparison of the total time taken to solve subproblems using the BF, \texttt{GREEDY}, \texttt{2GREEDY}, \texttt{HYBRID}, and \texttt{OPT} algorithms. 
}
}
\label{tab:time-to-solv-disc-algs}
\end{sidewaystable}

\begin{figure}[!ht]
    \centering
    \includegraphics[width=0.7\linewidth]{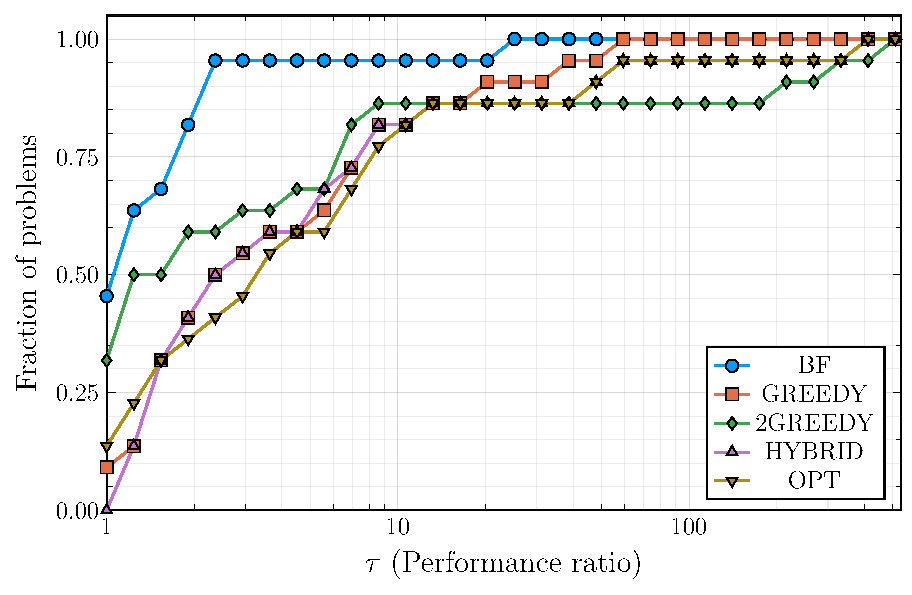}
    \caption{{Performance profiles of the bounding-focused discretization methods.}}
    \label{fig:perf-prof-all-disc}
\end{figure}

Table~\ref{tab:resul-simple-probs} details the instance, dimensions $d_x$ and $d_y$, and the number of iterations~$k$ required by each method for the lower bound to converge to $v^*$ (note that the number of discretization points at termination is $2k-2$ for \texttt{2GREEDY} and $k-1$ for the other methods). 
{Computational times are reported in Table~\ref{tab:time-to-solv-disc-algs}, with performance profiles~\cite{dolan2002benchmarking} shown in Figure~\ref{fig:perf-prof-all-disc}. 
Most problems are solved within a few seconds, with the notable exception of the larger Cerulli instances, where both the max-min solve time and the lower bounding time can be significant. 
Although our new discretization methods reduce the number of global solves, their computational advantage is often offset by the added cost of solving max-min problems, especially since the BF algorithm exhibits relatively low runtime on all but two instances.}

On almost all instances, our proposed bounding-focused discretization methods require fewer iterations for convergence than the BF algorithm (the main exception is ``Watson h'', where numerical issues affect our methods).
Notably, for most instances, our new discretization methods converge within only two iterations (corresponding to a single discretization point for \texttt{GREEDY}, \texttt{HYBRID}, and \texttt{OPT}, and two points for \texttt{2GREEDY}).
Although \texttt{OPT} is theoretically expected to perform at least as well as the other discretization methods in terms of iteration count, Table~\ref{tab:resul-simple-probs} reveals that this is not always the case in practice.{
This discrepancy arises because the solution of the max-min problem~\eqref{eqn:method-2} can sometimes get stuck at poor local maxima.
Overall, among the proposed methods, Algorithms~\texttt{GREEDY} and~\texttt{2GREEDY} deliver the best performance on small-scale instances in terms of both iteration count and computational time.}

Table~\ref{tab:resul-mpcc}, Table \ref{tab:time-to-solv-mpcc}, and Figure \ref{fig:perf-prof-mpcc} present the performance of our discretization methods when the max-min problems~\eqref{eqn:method-2} and~\eqref{eqn:method-1} are {relaxed} and solved as mathematical programs with complementarity constraints (MPCCs), see Section~\ref{subsec:solving_max_min}.
These MPCCs are solved using Knitro's dedicated algorithm for such problems.
Comparing Table~\ref{tab:resul-mpcc} with Table~\ref{tab:resul-simple-probs}, we observe that the MPCC-based discretizations sometimes converge in fewer iterations, but in other cases require more iterations.
{However, the total time spent solving the max-min problems is generally lower with the MPCC approach (see Table~\ref{tab:time-to-solv-mpcc}), which results in improved performance profiles for the proposed methods compared to the BF algorithm (with the exception of \texttt{OPT}, which suffers from significant numerical issues on the ``Tsoukalas \& Rustem 2.1'' and Cerulli et al.\ \cite{cerulli2022convergent} instances).}

\begin{table}[t]
\centering
\begin{tabular}{rccccccc}
\hline
\textbf{Instance} & $\mathbf{d_x}$ & $\mathbf{d_y}$ & BF & \texttt{GREEDY} & \texttt{2GREEDY} & \texttt{HYBRID} & \texttt{OPT} \\
& & & \multicolumn{5}{c}{\textbf{number of iterations for convergence}} \\
\hline
Watson 2 & 2 & 1 & \textbf{2} & \textbf{2} & \textbf{2} & \textbf{2} & \textbf{2} \\
Watson 5 & 3 & 1 & 5 & 6 & \textbf{2} & 7 & 5 \\
Watson 5\_10 & 10 & 1 & 3 & 4 & \textbf{2} & 3 & 3 \\
Watson 6 & 2 & 1 & 3 & \textbf{3} & \textbf{3} & \textbf{3} & \textbf{3} \\
Watson 7 & 3 & 2 & \textbf{2} & \textbf{2} & \textbf{2} & \textbf{2} & \textbf{2} \\
Watson 8 & 6 & 2 & 15 & 7 & \textbf{5} & 15 & 8 \\
Watson 9 & 6 & 2 & 9 & 5 & \textbf{4} & 13 & 5 \\
Watson h & 2 & 1 & \textbf{18} & \textbf{18} & \textbf{18} & \textbf{18} & \textbf{18} \\
Watson n & 2 & 1 & 3 & \textbf{2} & \textbf{2} & \textbf{2} & \textbf{2} \\
Seidel \& Küfer 2.1 & 2 & 1 & 8 & \textbf{2} & 3 & \textbf{2} & \textbf{2} \\
Tsoukalas \& Rustem 2.1 & 1 & 1 & 8 & 9 & \textbf{6} & 9 & $>100$ \\
Mitsos 4\_3 & 3 & 1 & 5 & 3 & \textbf{2} & 3 & 3 \\
Mitsos 4\_6 & 6 & 1 & 7 & 7 & 6 & \textbf{4} & \textbf{4} \\
Mitsos DP & 1 & 1 & 28 & \textbf{2} & \textbf{2} & \textbf{2} & \textbf{2} \\
Cerulli et al.\ PSD 1 & 21 & 5 & \textbf{2} & 3 & \textbf{2} & 5 & $>100$ \\
Cerulli et al.\ PSD 2 & 21 & 5 & \textbf{2} & 3 & \textbf{2} & 5 & $>100$ \\
Cerulli et al.\ PSD 3 & 21 & 5 & \textbf{2} & 3 & \textbf{2} & 5 & $>100$ \\
Cerulli et al.\ PSD 4 & 21 & 5 & \textbf{2} & 3 & \textbf{2} & 5 & $>100$ \\
Cerulli et al.\ PSD 5 & 66 & 10 & \textbf{5}  & \textbf{5} & \textbf{5}  & 8 & 92 \\
Cerulli et al.\ PSD 6 & 66 & 10 &  \textbf{6} & \textbf{6} &  \textbf{6} & 9 & 32 \\
Cerulli et al.\ PSD 7 & 105 & 13 & \textbf{7} & 8 & \textbf{7} & 10  & 32 \\
Cerulli et al.\ PSD 8 & 105 & 13 & \textbf{5} & 6 & \textbf{5} & 8 & 34 \\
\hline
\end{tabular}
\caption{
{Results with the MPCC relaxation of the max-min problems.
Bold entries correspond to the minimum number of iterations for each instance.}
}
\label{tab:resul-mpcc}
\end{table}

\begin{sidewaystable}[]
\footnotesize
\begin{tabular}{rccccccc}
\toprule
\textbf{Instance} & $\mathbf{d_x}$ & $\mathbf{d_y}$ & BF & \texttt{GREEDY} & \texttt{2GREEDY} & \texttt{HYBRID} & \texttt{OPT} \\[0.02in]
\multicolumn{3}{c}{} & \multicolumn{5}{c}{\textbf{Total time to solve LBP/LLP/max-min (seconds)}} \\
\midrule
Watson 2 & 2 & 1 & 0.07 / 0.06 / 0.00 & 0.08 / 0.08 / 0.00 & 0.06 / 0.06 / 0.00 & 0.05 / 0.04 / 0.00 & 0.05 / 0.05 / 0.00 \\
Watson 5 & 3 & 1 & 0.16 / 0.43 / 0.00 & 0.10 / 0.44 / 0.02 & 0.02 / 0.06 / 0.00 & 0.19 / 0.16 / 0.02 & 0.09 / 0.19 / 0.00 \\
Watson 5\_10 & 10 & 1 & 0.12 / 0.08 / 0.00 & 0.12 / 0.13 / 0.02 & 0.05 / 0.06 / 0.00 & 0.07 / 0.07 / 0.00 & 0.06 / 0.05 / 0.02 \\
Watson 6 & 2 & 1 & 0.15 / 0.05 / 0.00 & 0.14 / 0.04 / 0.02 & 0.13 / 0.03 / 0.02 & 0.11 / 0.02 / 0.00 & 0.11 / 0.03 / 0.02 \\
Watson 7 & 3 & 2 & 0.06 / 0.03 / 0.00 & 0.04 / 0.02 / 0.00 & 0.04 / 0.02 / 0.00 & 0.04 / 0.01 / 0.00 & 0.02 / 0.02 / 0.00 \\
Watson 8 & 6 & 2 & 0.00 / 1.29 / 0.00 & 0.00 / 0.33 / 0.03 & 0.00 / 0.15 / 0.02 & 0.00 / 1.03 / 0.08 & 0.00 / 0.27 / 0.03 \\
Watson 9 & 6 & 2 & 0.00 / 0.29 / 0.00 & 0.00 / 0.22 / 0.02 & 0.00 / 0.10 / 0.03 & 0.00 / 0.69 / 0.05 & 0.00 / 0.14 / 0.08 \\
Watson h & 2 & 1 & 1.83 / 0.65 / 0.00 & 1.69 / 0.52 / 0.13 & 1.63 / 0.46 / 0.16 & 1.54 / 0.38 / 0.16 & 1.50 / 0.41 / 0.17 \\
Watson n & 2 & 1 & 0.05 / 0.04 / 0.00 & 0.02 / 0.05 / 0.00 & 0.04 / 0.07 / 0.00 & 0.02 / 0.03 / 0.00 & 0.04 / 0.02 / 0.00 \\
Seidel \& Küfer 2.1 & 2 & 1 & 0.24 / 0.25 / 0.00 & 0.06 / 0.03 / 0.00 & 0.10 / 0.22 / 0.00 & 0.07 / 0.03 / 0.00 & 0.04 / 0.02 / 0.00 \\
Tsoukalas \& Rustem 2.1 & 1 & 1 & 0.47 / 0.29 / 0.00 & 0.50 / 0.22 / 0.09 & 0.31 / 0.19 / 0.06 & 0.35 / 0.20 / 0.08 & Did not converge \\
Mitsos 4\_3 & 3 & 1 & 0.00 / 0.21 / 0.00 & 0.00 / 0.06 / 0.00 & 0.00 / 0.07 / 0.00 & 0.00 / 0.05 / 0.00 & 0.00 / 0.06 / 0.00 \\
Mitsos 4\_6 & 6 & 1 & 0.00 / 0.74 / 0.00 & 0.00 / 0.68 / 0.03 & 0.00 / 0.54 / 0.00 & 0.00 / 0.14 / 0.02 & 0.00 / 0.12 / 0.02 \\
Mitsos DP & 1 & 1 & 1.51 / 1.30 / 0.00 & 0.02 / 0.06 / 0.00 & 0.05 / 0.07 / 0.00 & 0.04 / 0.04 / 0.00 & 0.04 / 0.05 / 0.00 \\
Cerulli et al. PSD 1 & 21 & 5 & 0.29 / 0.09 / 0.00 & 0.38 / 0.11 / 0.02 & 0.18 / 0.07 / 0.00 & 0.56 / 0.15 / 0.02 & Did not converge \\
Cerulli et al. PSD 2 & 21 & 5 & 0.24 / 0.09 / 0.00 & 0.33 / 0.12 / 0.02 & 0.16 / 0.09 / 0.00 & 0.49 / 0.20 / 0.02 & Did not converge \\
Cerulli et al. PSD 3 & 21 & 5 & 0.25 / 0.09 / 0.00 & 0.34 / 0.11 / 0.00 & 0.17 / 0.08 / 0.00 & 0.54 / 0.15 / 0.00 & Did not converge \\
Cerulli et al. PSD 4 & 21 & 5 & 0.28 / 0.04 / 0.00 & 0.40 / 0.06 / 0.02 & 0.32 / 0.03 / 0.00 & 0.49 / 0.09 / 0.00 & Did not converge \\
Cerulli et al. PSD 5 & 66 & 10 & 1.53 / 0.13 / 0.00 & 1.90 / 0.13 / 0.50 & 1.76 / 0.11 / 0.27 & 2.91 / 0.18 / 0.67 & 36.50 / 1.70 / 1120.73 \\
Cerulli et al. PSD 6 & 66 & 10 & 1.75 / 0.14 / 0.00 & 2.28 / 0.10 / 1.08 & 2.03 / 0.09 / 0.38 & 3.11 / 0.15 / 1.47 & 12.08 / 0.69 / 119.33 \\
Cerulli et al. PSD 7 & 105 & 13 & 5.55 / 0.17 / 0.00 & 6.92 / 0.18 / 3.88 & 5.93 / 0.15 / 4.53 & 8.44 / 0.21 / 7.56 & 27.74 / 0.66 / 112.41 \\
Cerulli et al. PSD 8 & 105 & 13 & 4.15 / 317.20 / 0.00 & 5.07 / 322.16 / 4.69 & 4.09 / 318.09 / 3.86 & 6.67 / 332.77 / 2.56 & 28.58 / 918.75 / 404.45 \\
\bottomrule
\end{tabular}
\caption{
{Comparison of the total time taken to solve subproblems using the BF, \texttt{GREEDY}, \texttt{2GREEDY}, \texttt{HYBRID}, and \texttt{OPT} algorithms when the max-min problem is relaxed and solved as an MPCC. Instances marked as ``Did not converge'' either exceeded the maximum time limit while solving~\eqref{eqn:disc-lbp} or~\eqref{eqn:llp}, or reached the maximum number of iterations.}
}
\label{tab:time-to-solv-mpcc}
\end{sidewaystable}

\begin{figure}[!ht]
    \centering
    \includegraphics[width=0.7\linewidth]{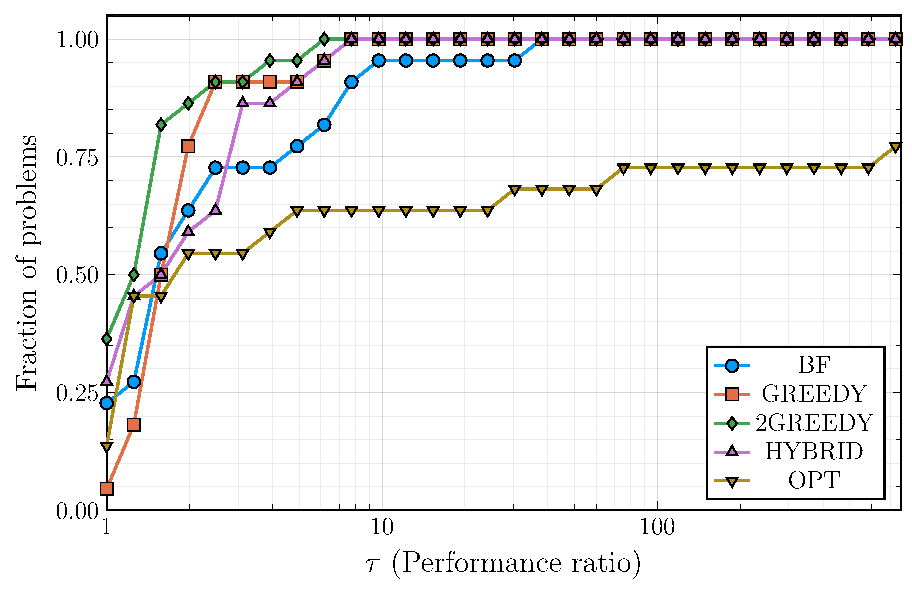}
    \caption{{Performance profiles of the bounding-focused discretization methods when the max-min problem is relaxed and solved as an MPCC.}}
    \label{fig:perf-prof-mpcc}
\end{figure}

\subsection{Bounding-focused generalized discretization methods}

We test our generalized discretization methods on small-scale SIP instances, as solving~\eqref{eqn:higher-order-lbp} to global optimality can be challenging for larger-scale SIPs.

{
Table~\ref{tab:resul-simple-probs-sens}, Table \ref{tab:time-to-solve-gen}, and Figure \ref{fig:perf-prof-gen} summarize the performance of our bounding-focused generalized discretization methods relative to the BF algorithm.
While these methods perform well on some instances, they exceed the time limit for solving~\eqref{eqn:higher-order-lbp} on several others.}
BARON appears to stall while solving these MINLPs, potentially due to weak relaxations of~\eqref{eqn:higher-order-lbp}.
In most cases, the bundle method also terminates after one iteration during the solution of problems~\eqref{eqn:method-2} and~\eqref{eqn:method-1}, which suggests that our initial guess is either already locally optimal or fails to provide a clear direction for improvement (possibly due to the smooth approximation of the projection operator).
With the exception of ``Watson~h'' and ``Tsoukalas \& Rustem~2.1'', the generalized discretization methods do not offer a significant advantage over the bounding-focused discretization methods in Section~\ref{sec:accelerated_discretization}.

\begin{table}[t]
\centering
% \resizebox{1.\linewidth}{!}{
\begin{tabular}{rccccccc}
\hline
\textbf{Instance}     & $\mathbf{d_x}$ & $\mathbf{d_y}$ &  BF            & \texttt{G-GREEDY}        & \texttt{G-2GREEDY}       & \texttt{G-HYBRID}    & \texttt{G-OPT}   \\
            &    &    &  \multicolumn{5}{c}{\textbf{number of iterations for convergence}} \\ \hline
Watson 2    & 2 & 1  & \textbf{2}                 & \textbf{2}            & \textbf{2}           & \textbf{2}             & \textbf{2}  \\
Watson 5    &  3 &  1 &      5       &  5 &  \textbf{3}  &    4      & 4   \\
Watson 5\_10    &  10 &  1 &      \textbf{3}       &  \textbf{3} &  \textbf{3}  &    \textbf{3}      & \textbf{3}   \\
Watson 6    & 2  & 1  & 3             & \textbf{2}            & \textbf{2}           & \textbf{2}             & \textbf{2}  \\
Watson 7    & 3 & 2  & \textbf{2}                 & TLE           & TLE           & 39             & TLE \\
Watson 8    &  6 &  2 &      15       &   TLE   & TLE &    TLE  &  TLE \\
Watson 9    &  6 &  2 &      9     &   TLE   & TLE &    TLE  &  TLE  \\
Watson h   & 2  & 1  & 18            & \textbf{2}    & \textbf{2} & \textbf{2} & \textbf{2}\\
Watson n    &  2&  1 &      3      &  \textbf{3}  &  \textbf{3}   &    \textbf{3}      & \textbf{3}  \\
Seidel \& K{\"u}fer 2.1      & 2  & 1  & 8            & \textbf{2}            & \textbf{2}           & \textbf{2}  & \textbf{2}             \\
Tsoukalas \& Rustem 2.1        & 1  & 1   & 8      & 4 & \textbf{3}   & 4  &4 \\
Mitsos 4\_3 & 3  & 1  & \textbf{5}             &      9   & \textbf{5}     &  \textbf{5}  & 6 \\
Mitsos 4\_6 & 6  & 1  & \textbf{7}             &         TLE    & TLE  & \textbf{7} & \textbf{7} \\
Mitsos DP & 1  & 1  & 28            &         \textbf{2}     & \textbf{2}  & \textbf{2}& \textbf{2} \\
\hline
\end{tabular}
% }%
\caption{
{Comparison of the BF, \texttt{G-GREEDY}, \texttt{G-2GREEDY}, \texttt{G-HYBRID}, and \texttt{G-OPT} algorithms.
Bold entries correspond to the minimum number of iterations for each instance, and TLE denotes the time limit was exceeded while solving~\eqref{eqn:higher-order-lbp} or~\eqref{eqn:llp}.
}
}
\label{tab:resul-simple-probs-sens}
% \vspace*{-0.1in}
\end{table}

\begin{sidewaystable}[]
\footnotesize
\begin{tabular}{rccccccc}
\toprule
\textbf{Instance} & $\mathbf{d_x}$ & $\mathbf{d_y}$ & BF & \texttt{G-GREEDY} & \texttt{G-2GREEDY} & \texttt{G-HYBRID} & \texttt{G-OPT} \\[0.02in]
\multicolumn{3}{c}{} & \multicolumn{5}{c}{\textbf{Total time to solve (G-)LBP/LLP/max-min (seconds)}} \\
\midrule
Watson 2 & 2 & 1 & 0.07 / 0.06 / 0.00 & 0.234 / 0.04 / 1.32 & 0.151 / 0.05 / 0.47 & 0.129 / 0.04 / 0.02 & 0.098 / 0.04 / 0.03 \\
Watson 5 & 3 & 1 & 0.16 / 0.43 / 0.00 & 0.853 / 0.17 / 1.00 & 0.463 / 0.14 / 0.08 & 0.765 / 0.13 / 0.28 & 0.586 / 0.09 / 1.86 \\
Watson 5\_10 & 10 & 1 & 0.12 / 0.08 / 0.00 & 0.23 / 0.05 / 0.12 & 0.473 / 0.08 / 0.03 & 0.19 / 0.08 / 0.05 & 0.18 / 0.07 / 0.05 \\
Watson 6 & 2 & 1 & 0.15 / 0.05 / 0.00 & 0.174 / 0.03 / 0.14 & 0.132 / 0.02 / 0.02 & 0.133 / 0.02 / 0.13 & 0.226 / 0.02 / 0.15 \\
Watson 7 & 3 & 2 & 0.06 / 0.03 / 0.00 & Did not converge & Did not converge & 263.52 / 0.50 / 327.58 & Did not converge \\
Watson 8 & 6 & 2 & 0.002 / 1.29 / 0.00 & Did not converge & Did not converge & Did not converge & Did not converge \\
Watson 9 & 6 & 2 & 0.002 / 0.29 / 0.00 & Did not converge & Did not converge & Did not converge & Did not converge \\
Watson h & 2 & 1 & 1.83 / 0.65 / 0.00 & 0.666 / 0.04 / 0.02 & 0.732 / 0.08 / 0.02 & 0.626 / 0.03 / 0.01 & 0.642 / 0.03 / 0.02 \\
Watson n & 2 & 1 & 0.05 / 0.04 / 0.00 & 0.246 / 0.04 / 1.74 & 0.302 / 0.04 / 4.58 & 0.166 / 0.03 / 1.77 & 0.159 / 0.04 / 1.77 \\
Seidel \& Küfer 2.1 & 2 & 1 & 0.24 / 0.25 / 0.00 & 0.212 / 0.06 / 0.16 & 0.563 / 0.03 / 0.19 & 0.186 / 0.04 / 0.36 & 0.306 / 0.06 / 0.38 \\
Tsoukalas \& Rustem 2.1 & 1 & 1 & 0.47 / 0.29 / 0.00 & 1.099 / 0.14 / 0.40 & 0.457 / 0.11 / 0.18 & 0.781 / 0.12 / 0.49 & 0.887 / 0.11 / 0.52 \\
Mitsos 4\_3 & 3 & 1 & 0.001 / 0.21 / 0.00 & 1.728 / 0.26 / 5.10 & 1.024 / 0.18 / 0.80 & 0.995 / 0.17 / 1.76 & 1.677 / 0.20 / 3.05 \\
Mitsos 4\_6 & 6 & 1 & 0.001 / 0.74 / 0.00 & Did not converge & Did not converge & 46.659 / 0.63 / 14.83 & 46.144 / 0.65 / 14.76 \\
Mitsos DP & 1 & 1 & 1.51 / 1.3 / 0.00 & 0.10 / 0.04 / 0.02 & 0.45 / 0.04 / 0.01 & 0.10 / 0.05 / 0.01 & 0.11 / 0.04 / 0.01 \\
\bottomrule
\end{tabular}
\caption{
{Comparison of the total time taken to solve subproblems using the BF, \texttt{G-GREEDY}, \texttt{G-2GREEDY}, \texttt{G-HYBRID}, and \texttt{G-OPT} algorithms. Instances marked as ``Did not converge'' either exceeded the maximum time limit while solving~\eqref{eqn:higher-order-lbp} or~\eqref{eqn:llp}, or reached the maximum number of iterations.}
}
\label{tab:time-to-solve-gen}
\end{sidewaystable}

\begin{figure}
    \centering
    \includegraphics[width=0.7\linewidth]{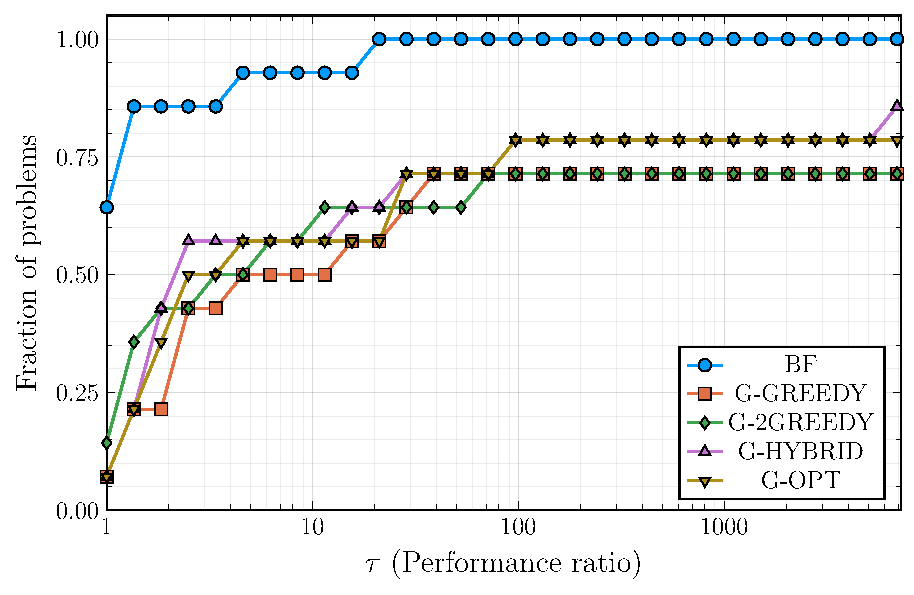}
    \caption{{Performance profiles of the bounding-focused generalized discretization methods.}}
    \label{fig:perf-prof-gen}
\end{figure}

\subsection{Discussion of results}
{Tables~\ref{tab:resul-simple-probs} to~\ref{tab:time-to-solve-gen} and Figures~\ref{fig:perf-prof-all-disc} to~\ref{fig:perf-prof-gen}} show that our bounding-focused (generalized) discretization methods have the potential to significantly reduce the number of iterations for convergence relative to the BF algorithm.
While Algorithms \texttt{GREEDY} and \texttt{2GREEDY} may not result in \textit{optimal} bounding-focused discretizations, we find that they offer \textit{viable} alternatives to the BF algorithm. 
These algorithms require fewer iterations to converge and only necessitate the heuristic solution of the max-min problem~\eqref{eqn:method-1} with $d_y$ variables at each iteration.
{Numerical experiments suggest that these new bounding-focused discretization methods are competitive with the BF algorithm on small and medium-scale SIPs, especially when the max-min problem is relaxed and solved as an MPCC.}

While our preliminary numerical results are encouraging, there is a need for the design of more efficient and reliable algorithms to solve our max-min formulations, particularly addressing the generalized discretization problems~\eqref{eqn:genl_method-2} and~\eqref{eqn:genl_method-1}. 
We anticipate that our new discretization methods will offer significant advantages when the BF algorithm requires numerous iterations to converge, and the solution of~\eqref{eqn:disc-lbp} to global optimality is relatively time-consuming.
Our new bounding-focused discretization methods can also be used as an expert strategy for machine learning approaches that seek to learn an optimal sequence of discretizations for solving families of nonconvex SIPs (cf.\ \cite{kannan2022learning,paulus2022learning,deza2023machine}).
Our future work will delve into exploring the effectiveness of these new discretization methods specifically on large-scale nonconvex SIPs.

\section{Future work}
\label{sec:conclusion}

There are many interesting avenues for future work.
First, we wish to explore the effectiveness of our new discretization methods for large-scale nonconvex SIPs.
Second, we would like to extend our bounding-focused (generalized) discretization methods to generalized semi-infinite programs~\cite{djelassi2021recent,mitsos2015global}.
Third, our bounding-focused discretization methods could be modified (cf.\ \cite{mitsos2011global}) to generate feasible points to~\eqref{eqn:sip}.
Fourth, extensions of our max-min formulations can enable the design of more efficient cutting-plane methods for a broader class of optimization problems (cf.\ \cite{paulus2022learning}).
Finally, using machine learning to learn a sequence of optimal discretizations (cf.\ \cite{paulus2022learning,kannan2022learning,deza2023machine}) can mitigate the computational burden of solving our max-min problems for larger dimensions.

\section*{Acknowledgments}
E.M.T.\ and J.J.\ acknowledge the support of the {Norwegian} Research Council through the AutoPRO project (RN: 309628).
R.K.\ acknowledges funding from the Center for Nonlinear Studies at Los Alamos National Laboratory and projects 20210078DR and 20230091ER of the U.S.\ Department of Energy’s LANL LDRD program.
We thank {Dr.\ Harsha} Nagarajan (LANL) and Dr.\ Qi Zhang (UMN) for helpful discussions.

{
\footnotesize
\section*{References}
\begingroup
\renewcommand{\section}[2]{}%
\bibliographystyle{abbrvnat}
\bibliography{main}

\begin{thebibliography}{54}
\providecommand{\natexlab}[1]{#1}
\providecommand{\url}[1]{\texttt{#1}}
\expandafter\ifx\csname urlstyle\endcsname\relax
  \providecommand{\doi}[1]{doi: #1}\else
  \providecommand{\doi}{doi: \begingroup \urlstyle{rm}\Url}\fi

\bibitem[Agrawal et~al.(2019)Agrawal, Amos, Barratt, Boyd, Diamond, and
  Kolter]{agrawal2019differentiable}
A.~Agrawal, B.~Amos, S.~Barratt, S.~Boyd, S.~Diamond, and J.~Z. Kolter.
\newblock Differentiable convex optimization layers.
\newblock \emph{Advances in Neural Information Processing Systems}, 32, 2019.

\bibitem[Baltean-Lugojan et~al.(2019)Baltean-Lugojan, Bonami, Misener, and
  Tramontani]{baltean2019scoring}
R.~Baltean-Lugojan, P.~Bonami, R.~Misener, and A.~Tramontani.
\newblock Scoring positive semidefinite cutting planes for quadratic
  optimization via trained neural networks.
\newblock \emph{Optimization Online. URL:
  \url{http://www.optimization-online.org/DB\_HTML/2018/11/6943.html}}, 2019.

\bibitem[Bertsekas(1999)]{bertsekas1999nonlinear}
D.~P. Bertsekas.
\newblock \emph{Nonlinear Programming}.
\newblock Athena Scientific, Belmont, Massachusetts, 2nd edition, 1999.

\bibitem[Bhattacharjee et~al.(2005)Bhattacharjee, Lemonidis, Green~Jr, and
  Barton]{bhattacharjee2005global}
B.~Bhattacharjee, P.~Lemonidis, W.~H. Green~Jr, and P.~I. Barton.
\newblock Global solution of semi-infinite programs.
\newblock \emph{Mathematical Programming}, 103\penalty0 (2):\penalty0 283--307,
  2005.

\bibitem[Blankenship and Falk(1976)]{blankenship1976infinitely}
J.~W. Blankenship and J.~E. Falk.
\newblock Infinitely constrained optimization problems.
\newblock \emph{Journal of Optimization Theory and Applications}, 19\penalty0
  (2):\penalty0 261--281, 1976.

\bibitem[Burke et~al.(2020)Burke, Curtis, Lewis, Overton, and
  Sim{\~o}es]{burke2020gradient}
J.~V. Burke, F.~E. Curtis, A.~S. Lewis, M.~L. Overton, and L.~E. Sim{\~o}es.
\newblock \emph{Gradient sampling methods for nonsmooth optimization}, pages
  201--225.
\newblock Springer, Cham, 2020.
\newblock \doi{10.1007/978-3-030-34910-3_6}.

\bibitem[Cerulli et~al.(2022)Cerulli, Oustry, d'Ambrosio, and
  Liberti]{cerulli2022convergent}
M.~Cerulli, A.~Oustry, C.~d'Ambrosio, and L.~Liberti.
\newblock Convergent algorithms for a class of convex semi-infinite programs.
\newblock \emph{SIAM Journal on Optimization}, 32\penalty0 (4):\penalty0
  2493--2526, 2022.

\bibitem[Cheney and Goldstein(1959)]{cheney1959newton}
E.~W. Cheney and A.~A. Goldstein.
\newblock Newton's method for convex programming and {T}chebycheff
  approximation.
\newblock \emph{Numerische Mathematik}, 1\penalty0 (1):\penalty0 253--268,
  1959.

\bibitem[Clarke(1990)]{clarke1990optimization}
F.~H. Clarke.
\newblock \emph{Optimization and nonsmooth analysis}.
\newblock SIAM, Philadelphia, 1990.

\bibitem[Coniglio and Tieves(2015)]{coniglio2015generation}
S.~Coniglio and M.~Tieves.
\newblock On the generation of cutting planes which maximize the bound
  improvement.
\newblock In \emph{International Symposium on Experimental Algorithms}, pages
  97--109. Springer, 2015.

\bibitem[Das et~al.(2022)Das, Aravind, Cherukuri, and Chatterjee]{das2022near}
S.~Das, A.~Aravind, A.~Cherukuri, and D.~Chatterjee.
\newblock Near-optimal solutions of convex semi-infinite programs via targeted
  sampling.
\newblock \emph{Annals of Operations Research}, 318\penalty0 (1):\penalty0
  129--146, 2022.

\bibitem[Dempe(2002)]{dempe2002foundations}
S.~Dempe.
\newblock \emph{Foundations of bilevel programming}.
\newblock Springer Science \& Business Media, New York, 2002.

\bibitem[Dempe(2017)]{dempe2017bilevel}
S.~Dempe.
\newblock \emph{Bilevel Optimization: Reformulation and First Optimality
  Conditions}, pages 1--20.
\newblock Springer, Singapore, 2017.
\newblock \doi{10.1007/978-981-10-4774-9_1}.

\bibitem[Deza and Khalil(2023)]{deza2023machine}
A.~Deza and E.~B. Khalil.
\newblock Machine learning for cutting planes in integer programming: A survey.
\newblock \emph{arXiv preprint arXiv:2302.09166}, 2023.

\bibitem[Djelassi(2020)]{djelassi2020discretization}
H.~Djelassi.
\newblock \emph{Discretization-Based Algorithms for the Global Solution of
  Hierarchical Programs}.
\newblock PhD thesis, RWTH Aachen, 2020.

\bibitem[Djelassi and Mitsos(2017)]{djelassi2017hybrid}
H.~Djelassi and A.~Mitsos.
\newblock A hybrid discretization algorithm with guaranteed feasibility for the
  global solution of semi-infinite programs.
\newblock \emph{Journal of Global Optimization}, 68\penalty0 (2):\penalty0
  227--253, 2017.

\bibitem[Djelassi et~al.(2021)Djelassi, Mitsos, and Stein]{djelassi2021recent}
H.~Djelassi, A.~Mitsos, and O.~Stein.
\newblock Recent advances in nonconvex semi-infinite programming: Applications
  and algorithms.
\newblock \emph{EURO Journal on Computational Optimization}, 9:\penalty0
  100006, 2021.

\bibitem[Dolan and Mor{\'e}(2002)]{dolan2002benchmarking}
E.~D. Dolan and J.~J. Mor{\'e}.
\newblock Benchmarking optimization software with performance profiles.
\newblock \emph{Mathematical Programming}, 91:\penalty0 201--213, 2002.

\bibitem[Dunning et~al.(2017)Dunning, Huchette, and Lubin]{JuMP2017}
I.~Dunning, J.~Huchette, and M.~Lubin.
\newblock {JuMP}: A modeling language for mathematical optimization.
\newblock \emph{SIAM Review}, 59\penalty0 (2):\penalty0 295--320, 2017.

\bibitem[Ermoliev et~al.(1995)Ermoliev, Norkin, and
  Wets]{ermoliev1995minimization}
Y.~M. Ermoliev, V.~I. Norkin, and R.~J. Wets.
\newblock The minimization of semicontinuous functions: mollifier subgradients.
\newblock \emph{SIAM Journal on Control and Optimization}, 33\penalty0
  (1):\penalty0 149--167, 1995.

\bibitem[Fiacco(1983)]{fiacco1983}
A.~V. Fiacco.
\newblock \emph{Introduction to sensitivity and stability analysis in nonlinear
  programming}, volume 165.
\newblock Academic Press, New York, 1983.

\bibitem[Floudas and Stein(2008)]{floudas2008adaptive}
C.~A. Floudas and O.~Stein.
\newblock The adaptive convexification algorithm: a feasible point method for
  semi-infinite programming.
\newblock \emph{SIAM Journal on Optimization}, 18\penalty0 (4):\penalty0
  1187--1208, 2008.

\bibitem[Grossmann et~al.(1983)Grossmann, Halemane, and
  Swaney]{grossmann1983optimization}
I.~E. Grossmann, K.~P. Halemane, and R.~E. Swaney.
\newblock Optimization strategies for flexible chemical processes.
\newblock \emph{Computers \& Chemical Engineering}, 7\penalty0 (4):\penalty0
  439--462, 1983.

\bibitem[Harwood et~al.(2021)Harwood, Papageorgiou, and
  Trespalacios]{harwood2021note}
S.~M. Harwood, D.~J. Papageorgiou, and F.~Trespalacios.
\newblock A note on semi-infinite program bounding methods.
\newblock \emph{Optimization Letters}, 15\penalty0 (4):\penalty0 1485--1490,
  2021.

\bibitem[Hijazi et~al.(2014)Hijazi, Bonami, and Ouorou]{hijazi2014outer}
H.~Hijazi, P.~Bonami, and A.~Ouorou.
\newblock An outer-inner approximation for separable mixed-integer nonlinear
  programs.
\newblock \emph{INFORMS Journal on Computing}, 26\penalty0 (1):\penalty0
  31--44, 2014.

\bibitem[Kannan et~al.(2022)Kannan, Nagarajan, and Deka]{kannan2022learning}
R.~Kannan, H.~Nagarajan, and D.~Deka.
\newblock Strong partitioning and a machine learning approximation for
  accelerating the global optimization of nonconvex {QCQP}s.
\newblock \emph{arXiv preprint arXiv:2301.00306}, 2022.

\bibitem[Kelley(1960)]{kelley1960cutting}
J.~E. Kelley, Jr.
\newblock The cutting-plane method for solving convex programs.
\newblock \emph{Journal of the Society for Industrial and Applied Mathematics},
  8\penalty0 (4):\penalty0 703--712, 1960.

\bibitem[L{\'o}pez and Still(2007)]{lopez2007semi}
M.~L{\'o}pez and G.~Still.
\newblock Semi-infinite programming.
\newblock \emph{European Journal of Operational Research}, 180\penalty0
  (2):\penalty0 491--518, 2007.

\bibitem[M{\"a}kel{\"a}(2003)]{makela2003multiobjective}
M.~M. M{\"a}kel{\"a}.
\newblock Multiobjective proximal bundle method for nonconvex nonsmooth
  optimization: Fortran subroutine {MPBNGC} 2.0. {URL}:
  \url{http://napsu.karmitsa.fi/publications/pbncgc_report.pdf}.
\newblock \emph{Reports of the Department of Mathematical Information
  Technology, Series B. Scientific Computing, B}, 13, 2003.

\bibitem[Marendet et~al.(2020)Marendet, Goldsztejn, Chabert, and
  Jermann]{marendet2020standard}
A.~Marendet, A.~Goldsztejn, G.~Chabert, and C.~Jermann.
\newblock A standard branch-and-bound approach for nonlinear semi-infinite
  problems.
\newblock \emph{European Journal of Operational Research}, 282\penalty0
  (2):\penalty0 438--452, 2020.

\bibitem[Mitsos(2011)]{mitsos2011global}
A.~Mitsos.
\newblock Global optimization of semi-infinite programs via restriction of the
  right-hand side.
\newblock \emph{Optimization}, 60\penalty0 (10-11):\penalty0 1291--1308, 2011.

\bibitem[Mitsos(2016)]{mitsos2009test}
A.~Mitsos.
\newblock A test set of semi-infinite programs (revised by hatim djelassi).
\newblock Last Accessed on January 11, 2023.
  \url{https://www.avt.rwth-aachen.de/cms/AVT/Forschung/Systemverfahrenstechnik/~kpdo/A-Test-Set-of-Semi-Infinite-Programs/},
  2016.

\bibitem[Mitsos and Tsoukalas(2015)]{mitsos2015global}
A.~Mitsos and A.~Tsoukalas.
\newblock Global optimization of generalized semi-infinite programs via
  restriction of the right hand side.
\newblock \emph{Journal of Global Optimization}, 61\penalty0 (1):\penalty0
  1--17, 2015.

\bibitem[Mitsos et~al.(2008)Mitsos, Lemonidis, Lee, and
  Barton]{mitsos2008relaxation}
A.~Mitsos, P.~Lemonidis, C.~K. Lee, and P.~I. Barton.
\newblock Relaxation-based bounds for semi-infinite programs.
\newblock \emph{SIAM Journal on Optimization}, 19\penalty0 (1):\penalty0
  77--113, 2008.

\bibitem[Mordukhovich et~al.(2009)Mordukhovich, Nam, and
  Yen]{mordukhovich2009subgradients}
B.~S. Mordukhovich, N.~M. Nam, and N.~D. Yen.
\newblock Subgradients of marginal functions in parametric mathematical
  programming.
\newblock \emph{Mathematical Programming}, 116\penalty0 (1-2):\penalty0
  369--396, 2009.

\bibitem[Mutapcic and Boyd(2009)]{mutapcic2009cutting}
A.~Mutapcic and S.~Boyd.
\newblock Cutting-set methods for robust convex optimization with pessimizing
  oracles.
\newblock \emph{Optimization Methods \& Software}, 24\penalty0 (3):\penalty0
  381--406, 2009.

\bibitem[Nesterov(2018)]{nesterov2018lectures}
Y.~Nesterov.
\newblock \emph{Lectures on convex optimization}, volume 137.
\newblock Springer, Switzerland, 2018.

\bibitem[Paulus et~al.(2022)Paulus, Zarpellon, Krause, Charlin, and
  Maddison]{paulus2022learning}
M.~B. Paulus, G.~Zarpellon, A.~Krause, L.~Charlin, and C.~Maddison.
\newblock Learning to cut by looking ahead: Cutting plane selection via
  imitation learning.
\newblock In \emph{International Conference on Machine Learning}, pages
  17584--17600, 2022.

\bibitem[Ralph and Dempe(1995)]{ralph1995directional}
D.~Ralph and S.~Dempe.
\newblock Directional derivatives of the solution of a parametric nonlinear
  program.
\newblock \emph{Mathematical Programming}, 70\penalty0 (1-3):\penalty0
  159--172, 1995.

\bibitem[Sahinidis(1996)]{baron}
N.~V. Sahinidis.
\newblock {BARON}: A general purpose global optimization software package.
\newblock \emph{Journal of Global Optimization}, 8\penalty0 (2):\penalty0
  201--205, 1996.

\bibitem[Seidel and K{\"u}fer(2022)]{seidel2020adaptive}
T.~Seidel and K.-H. K{\"u}fer.
\newblock An adaptive discretization method solving semi-infinite optimization
  problems with quadratic rate of convergence.
\newblock \emph{Optimization}, 71\penalty0 (8):\penalty0 2211--2239, 2022.

\bibitem[Shapiro(2009)]{shapiro2009semi}
A.~Shapiro.
\newblock Semi-infinite programming, duality, discretization and optimality
  conditions.
\newblock \emph{Optimization}, 58\penalty0 (2):\penalty0 133--161, 2009.

\bibitem[Stechlinski et~al.(2018)Stechlinski, Khan, and
  Barton]{stechlinski2018generalized}
P.~Stechlinski, K.~A. Khan, and P.~I. Barton.
\newblock Generalized sensitivity analysis of nonlinear programs.
\newblock \emph{SIAM Journal on Optimization}, 28\penalty0 (1):\penalty0
  272--301, 2018.

\bibitem[Stechlinski et~al.(2019)Stechlinski, J{\"a}schke, and
  Barton]{stechlinski2019generalized}
P.~Stechlinski, J.~J{\"a}schke, and P.~I. Barton.
\newblock Generalized sensitivity analysis of nonlinear programs using a
  sequence of quadratic programs.
\newblock \emph{Optimization}, 68\penalty0 (2-3):\penalty0 485--508, 2019.

\bibitem[Stein(2003)]{stein2003bi}
O.~Stein.
\newblock \emph{Bi-level strategies in semi-infinite programming}, volume~71.
\newblock Springer Science \& Business Media, New York, 2003.

\bibitem[Stein and Steuermann(2012)]{stein2012adaptive}
O.~Stein and P.~Steuermann.
\newblock The adaptive convexification algorithm for semi-infinite programming
  with arbitrary index sets.
\newblock \emph{Mathematical Programming}, 136\penalty0 (1):\penalty0 183--207,
  2012.

\bibitem[Stein and Still(2003)]{stein2003solving}
O.~Stein and G.~Still.
\newblock Solving semi-infinite optimization problems with interior point
  techniques.
\newblock \emph{SIAM Journal on Control and Optimization}, 42\penalty0
  (3):\penalty0 769--788, 2003.

\bibitem[Still(2001)]{still2001discretization}
G.~Still.
\newblock Discretization in semi-infinite programming: the rate of convergence.
\newblock \emph{Mathematical Programming}, 91\penalty0 (1):\penalty0 53--69,
  2001.

\bibitem[Still(2018)]{still2018lectures}
G.~Still.
\newblock Lectures on parametric optimization: An introduction.
\newblock \emph{Optimization Online. URL:
  \url{https://optimization-online.org/2018/04/6587/}}, 2018.

\bibitem[Tanaka et~al.(1988)Tanaka, Fukushima, and Ibaraki]{tanaka1988globally}
Y.~Tanaka, M.~Fukushima, and T.~Ibaraki.
\newblock A globally convergent {SQP} method for semi-infinite nonlinear
  optimization.
\newblock \emph{Journal of Computational and Applied Mathematics}, 23\penalty0
  (2):\penalty0 141--153, 1988.

\bibitem[Tsoukalas and Rustem(2011)]{tsoukalas2011feasible}
A.~Tsoukalas and B.~Rustem.
\newblock A feasible point adaptation of the {Blankenship and Falk} algorithm
  for semi-infinite programming.
\newblock \emph{Optimization Letters}, 5\penalty0 (4):\penalty0 705--716, 2011.

\bibitem[Turan et~al.(2022)Turan, Kannan, and J{\"a}schke]{turan2022design}
E.~M. Turan, R.~Kannan, and J.~J{\"a}schke.
\newblock Design of {PID} controllers using semi-infinite programming.
\newblock In Y.~Yamashita and M.~Kano, editors, \emph{14th International
  Symposium on Process Systems Engineering}, volume~49 of \emph{Computer Aided
  Chemical Engineering}, pages 439--444. Elsevier, 2022.
\newblock \doi{https://doi.org/10.1016/B978-0-323-85159-6.50073-7}.

\bibitem[Vielma(2015)]{vielma2015mixed}
J.~P. Vielma.
\newblock Mixed integer linear programming formulation techniques.
\newblock \emph{SIAM Review}, 57\penalty0 (1):\penalty0 3--57, 2015.

\bibitem[Watson(1983)]{watson1983numerical}
G.~A. Watson.
\newblock Numerical experiments with globally convergent methods for
  semi-infinite programming problems.
\newblock In \emph{Semi-infinite programming and applications}, pages 193--205.
  Springer, Berlin, Heidelberg, 1983.

\end{thebibliography}
\endgroup
}

\appendix

\section{{Review of sensitivity theory}}
\label{sec:sensitivity_theory}

We briefly review standard results from parametric sensitivity theory~\cite{fiacco1983,still2018lectures}.
Consider the parametric nonlinear program (NLP):
\begin{align}
\label{eqn: NLP}
    \min_{z \in \R^n}\:\: &F(z,p) \\
    \text{s.t. } & c_i(z,p) \le 0, \quad \forall i \in \I, \nonumber \\
    & c_i(z,p) = 0, \quad \forall i \in \E, \nonumber
\end{align}
where $z\in\mathbb{R}^{n}$ are decision variables, $p\in\mathbb{R}^{d}$ are parameters, $F: \mathbb{R}^{n}\times\mathbb{R}^{d} \rightarrow \mathbb{R}$ is the objective function, $c:\mathbb{R}^{n}\times\mathbb{R}^{d} \rightarrow \mathbb{R}^{\abs{\I} + \abs{\E}}$ are the constraint functions,
and $\I$ and $\E$ are finite index sets.
We write $z^*(p)$ and $\nu^*(p)$ to denote a local minimum of problem~\eqref{eqn: NLP} and its optimal value $\nu^*(p) := F(z^*(p),p)$.

The Lagrangian for problem~\eqref{eqn: NLP} is $L(z,\lambda,p) := F(z,p) + \tr{\lambda} c(z,p)$, for Lagrange multipliers $\lambda \in \R^{\abs{\I} + \abs{\E}}$.
Let $\lambda^*(p)$ denote Lagrange multipliers satisfying the KKT conditions at $z^*(p)$, 
and $\A(z,p) := \{i \in \I : c_i(z,p) = 0\} \cup \E$ denote the indices of active constraints at a feasible point $z$.

We now present sufficient conditions under which $\nabla_p \nu^*(p)$ and $\nabla_p z^*(p)$ may be computed (see Fiacco~\cite{fiacco1983} or Still~\cite{still2018lectures} for details).

% \vspace*{-0.15in}
\begin{theorem}[Parametric sensitivities]
\label{thm:parametric_sens} 
Let $z^*(p)$ be a KKT point for problem~\eqref{eqn: NLP} with associated Lagrange multipliers $\lambda^*(p)$.
Suppose for some $\bar{p} \in \R^d$, functions $F$ and $c$ are twice continuously differentiable in a neighborhood of $(z^*(\bar{p}), \bar{p})$. 
{Assume that the following conditions hold at $(z^*(\bar{p}), \lambda^*(\bar{p}))$:
\begin{itemize}
\item Linear independence constraint qualification (LICQ): the vectors $\nabla_z c_i(z^*(\bar{p}),\bar{p})$, $i \in \A(z^*(\bar{p}),\bar{p})$, are linearly independent.
\item Strict complementarity (SC): $\lambda^*_i(\bar{p}) - c_i(z^*(\bar{p}),\bar{p})) > 0$, $\forall i \in \I$.
\end{itemize}
}
\noindent Additionally, suppose either 
\begin{enumerate}[label=(\alph*)]
\item $\abs{\A(z^*(\bar{p}),\bar{p})} = n$, or
\item the strong second order sufficient condition (SSOSC) holds at $(z^*(\bar{p}), \lambda^*(\bar{p}))$:
\begin{align*}
&{\tr{w} \nabla^2_{z} L(z^*(\bar{p}),\lambda^*(\bar{p}), \bar{p}) w > 0, \quad \forall w \in W \backslash \{0\},}
\end{align*}
\begin{align*}
{W := \big\{w \in \R^n : \:} &{\tr{(\nabla_z c_i(z^*(\bar{p}),\bar{p}))} w = 0, \:\: \forall i \in \{i \in \A(z^*(\bar{p}),\bar{p}) \cap \I: \lambda^*_i(\bar{p}) > 0\},} \\ &{\tr{(\nabla_z c_i(z^*(\bar{p}),\bar{p}))} w = 0, \:\: \forall i \in \E\big\}.}
\end{align*}
\end{enumerate}
Then $\exists \delta > 0$ such that $\forall p \in B_{\delta}(\bar{p})$, we can choose the mappings $z^*(p)$ and $\lambda^*(p)$ to be continuously differentiable on $B_{\delta}(\bar{p})$ and $z^*(p)$ to be a strict local minimizer of~\eqref{eqn: NLP}.
Additionally, for all $p \in B_{\delta}(\bar{p})$, the gradient of the value function~$\nu^*$ is given by 
\[
\nabla_p \nu^*(p) = \nabla_p L(z^*(p),\lambda^*(p),p),
\]
and the gradient of the solution mapping $z^*(p)$ may be computed for each $p \in B_{\delta}(\bar{p})$ as follows depending on whether condition (a) or (b) above holds:
\begin{enumerate}[label=(\alph*)]
\item Let $J_z(p) \in \R^{n \times n}$ and $J_p(p) \in \R^{n \times p}$ be matrices with rows $\tr{(\nabla_z c_i(z^*(p),p))}$, $i \in \A(z^*(p),p)$, and $\tr{(\nabla_p c_i(z^*(p),p))}$, $i \in \A(z^*(p),p)$, respectively. Then
\[
\nabla_p z^*(p) = -[J_z(p)]^{-1} J_p(p).
\]
\item 
Let $H_{z,\lambda}(p) := \begin{bmatrix}
\nabla^2_z L(z^*(p),\lambda^*(p),p) & J_z(p) \\
\tr{(J_z(p))} & 0 
\end{bmatrix}$, where $J_z(p)$ is a $\abs{\A(z^*(p),p)} \times n$ matrix with rows $\tr{(\nabla_z c_i(z^*(p),p))}$, $i \in \A(z^*(p),p)$.
Then
\[
\begin{bmatrix}
\nabla_p z^*(p) \\
\nabla_p \lambda^*_{\A}(p)
\end{bmatrix}  = - [H_{z,\lambda}(p)]^{-1}
\begin{bmatrix}
&\nabla_{pz} L(z^*(p),\lambda^*(p),p) \\
&\bigl(\nabla_p c_i(z^*(p),p)\bigr)_{i \in \A(z^*(p),p)}
\end{bmatrix},
\]
where $\lambda^*_{\A}(p)$ denotes the Lagrange multipliers of the active constraints at $z^*(p)$. 
\end{enumerate}
\end{theorem}%
\begin{proof}
% \vspace*{-0.15in}
See Chapter~3 of Fiacco~\cite{fiacco1983}, or the unified Theorem~4.4 in Still~\cite{still2018lectures}.
\end{proof}

Lemma~6.2 of Still~\cite{still2018lectures} presents weaker assumptions under which the (local) value function $\nu^*$ is locally Lipschitz continuous.
Theorem~1.12 of Dempe~\cite{dempe2017bilevel} and its surrounding discussion provides estimates of generalized gradients of~$\nu^*$ in the above setting.
Weaker assumptions for the solution mapping~$z^*$ to be H{\"o}lder continuous or locally Lipschitz continuous are presented in Theorems~6.2 to~6.5 of Still~\cite{still2018lectures}.

\section{Proofs}
\label{sec:proofs}

\subsection{Proof of Proposition~\ref{prop:alg_opt}}

The first part follows directly from the definition of the max-min problem~\eqref{eqn:method-2}, since the sequence of lower bounds obtained using Algorithm~\texttt{OPT} dominates the sequence of lower bounds obtained using the BF algorithm.
The second part follows, e.g., from Theorem~3.2 of Shapiro~\cite{shapiro2009semi}.

\subsection{Proof of Proposition~\ref{prop:convrate_bf}}

The fact that $N \leq \Big\lceil\Big(\frac{\textup{diam}(X) L_{g,x}}{\varepsilon_f} + 1\Big)^{d_x}\Big\rceil$ follows, e.g., from Section~5.2 of Mutapcic and Boyd~\cite{mutapcic2009cutting}.
We now argue that $N \leq \Big\lceil\Big(\frac{\textup{diam}(Y) L_{g,y}}{\varepsilon_f} + 1\Big)^{d_y}\Big\rceil$.

Suppose the BF algorithm has not converged by iteration $k > 1$.
Then for each $1 \leq j < k$, the candidate BF solution $x^k$ at iteration $k$ satisfies:
\begin{align*}
g(x^k,y^{BF,k}) > \varepsilon_f \:\: \text{and} \:\: g(x^k,y^{BF,j}) \leq 0 &\implies g(x^k, y^{BF,k}) - g(x^k, y^{BF,j}) > \varepsilon_f \\*
&\implies \norm{y^{BF,k} - y^{BF,j}} > \frac{\varepsilon_f}{L_{gy}}.
\end{align*}
Therefore, {the given} upper bound on the number of iterations required for the BF algorithm to converge can be obtained by calculating the number of Euclidean balls of radius $\frac{\varepsilon_f}{L_{gy}}$ needed to cover $\big( Y + \frac{\varepsilon_f}{2L_{gy}} B \big)$, where $B$ denotes the unit ball in $\R^{d_y}$ and $+$ denotes the Minkowski sum (cf.\ \citep{mutapcic2009cutting}).

\subsection{Proof of Theorem~\ref{thm:calmness}}

{Lipschitz continuity of the value function $V$ implies 
\[
V(\tilde{\varepsilon}) \geq V(0) - L_V \tilde{\varepsilon} = v^* - L_V \tilde{\varepsilon}, \quad \forall \tilde{\varepsilon} \in (0,\bar{\varepsilon}).
\]
By mirroring the proof of Proposition~\ref{prop:convrate_bf}, we conclude that whenever 
\[
k \geq \min\left\lbrace\left\lceil\left(\frac{\textup{diam}(Y) L_{g,y}}{\tilde{\varepsilon}} + 1\right)^{d_y}\right\rceil, \left\lceil\left(\frac{\textup{diam}(X) L_{g,x}}{\tilde{\varepsilon}} + 1\right)^{d_x}\right\rceil\right\rbrace, 
\]
the iterate $x^k$ produced by the BF algorithm satisfies $G(x^k) \leq \tilde{\varepsilon}$.
Consequently, for any such $k$, we have $LBD^k \geq V(\tilde{\varepsilon}) \geq v^* - L_V \tilde{\varepsilon}$ for the BF algorithm.
The desired result for the BF algorithm follows by setting $\tilde{\varepsilon} = \frac{\varepsilon}{L_V}$.
}

The result for Algorithm~\texttt{OPT} then readily follows since the sequence of lower bounds obtained using Algorithm~\texttt{OPT} dominate the sequence of lower bounds obtained using the BF algorithm.

\subsection{Proof of Lemma~\ref{lem:pwlest}}

The integral form of Taylor's theorem implies for any $z, \bar{z} \in Z$:
\begin{align*}
\norm{F(z) - F(\bar{z}) - \tr{\nabla F(\bar{z})} (z - \bar{z})} &\leq \frac{L_{\nabla F}}{2} \norm{z - \bar{z}}^2.
\end{align*}
The inequality $\norm{F(z) - F(\bar{z}) - \tr{\nabla F(\bar{z})} (z - \bar{z})} \leq \varepsilon$ holds whenever we have $z \in \Big\{v \in Z : \norm{v - \bar{z}} \leq \sqrt{\frac{2\varepsilon}{L_{\nabla F}}}\Big\}$.
The stated result follows by covering $Z$ using balls of radius $\sqrt{\frac{2\varepsilon}{L_{\nabla F}}}$ (cf.\ proof of Proposition~\ref{prop:convrate_bf}), setting $z^j$ to be the center of the $j$th ball, and setting $\alpha^j = \nabla F(z^j)$, $\beta^j = F(z^j) - \tr{\nabla F(z^j)} z^j$.

\subsection{Proof of Theorem~\ref{thm:convrate_gopt}}

Suppose Algorithm~\texttt{G-OPT} has not converged by iteration $k > 1$.
The candidate solution $x^k$ at iteration $k$ of Algorithm~\texttt{G-OPT} satisfies for each $1 \leq j < k$:
\begin{align*}
&g(x^k,y^*(x^k)) > \varepsilon_f \:\: \text{and} \:\: g(x^k,\text{proj}_Y(\bar{A}^j x^k + \bar{b}^j)) \leq 0 \\
\implies \: & g(x^k, y^*(x^k)) - g(x^k, \text{proj}_Y(\bar{A}^j x^k + \bar{b}^j)) > \varepsilon_f \\
\implies \: & L_{g,y} \norm{y^*(x^k) - \text{proj}_Y(\bar{A}^j x^k + \bar{b}^j)} > \varepsilon_f, \\
\implies \: & { L_{g,y} \norm{\text{proj}_Y(y^*(x^k)) - \text{proj}_Y(\bar{A}^j x^k + \bar{b}^j)} > \varepsilon_f,} \\
\implies \: & \norm{y^*(x^k) - (\bar{A}^j x^k + \bar{b}^j)} > \frac{\varepsilon_f}{L_{g,y}},
\end{align*}
where $Y^G_d = \{(\bar{A}^1, \bar{b}^1), \dots, (\bar{A}^{k-1}, \bar{b}^{k-1})\}$ denotes the generalized discretization at the start of iteration $k$, {and the final step follows by the projection theorem (see Proposition~2.1.3 of Bertsekas~\cite{bertsekas1999nonlinear})}.
Therefore, an upper bound on the number of iterations for Algorithm~\texttt{G-OPT} to converge can be obtained by estimating the minimal number $k$ of generalized discretization cuts required for $\underset{x \in X}{\sup} \: \underset{j \in [k]}{\min} \: \norm{y^*(x) - (\bar{A}^j x + \bar{b}^j)} \leq \frac{\varepsilon_f}{L_{g,y}}$.
The stated result then follows from Lemma~\ref{lem:pwlest}.

\end{document}